\documentclass[12pt]{amsart} 
\pdfoutput=1
\usepackage{amsmath, amsthm}
\usepackage{amssymb, bbm, dsfont}
\usepackage{mathbbol, bbold}
\usepackage{pzccal}

\usepackage[T1]{fontenc}

\usepackage{url}
\usepackage{enumitem}
\usepackage{parskip}

\let\oldproof\proof
\let\endoldproof\endproof

\newenvironment{proofr}[1][]{
\renewcommand{\qed}{\nobreak \ifvmode \relax \else
      \ifdim\lastskip<1.5em \hskip-\lastskip
      \hskip1.5em plus0em minus0.5em \fi \nobreak
      \hfill $\blacktriangle_{#1}$
      \fi}%
\oldproof}{%
\endoldproof}

\let\proof\proofr
\let\endproof\endproofr

\def\endof{\vphantom{i}$ \hfill $\blacktriangle} 

\newtheorem{THEOREM}{Theorem}[section]

\newtheorem{Conclusion}[THEOREM]{Conclusion}

\newtheorem{Theorem}[THEOREM]{Theorem}
\newenvironment{theorem}{\begin{Theorem}}{\end{Theorem}}

\newtheorem{Lemma}[THEOREM]{Lemma}
\newenvironment{lemma}{\begin{Lemma}}{\end{Lemma}}

{\theoremstyle{definition}
\newtheorem{notation}[THEOREM]{Notation}
\newtheorem{definition}[THEOREM]{Definition}

\newtheorem{observation}[THEOREM]{Observation}
\newtheorem{example}[THEOREM]{Example}
\newtheorem{setting}[THEOREM]{Setting}
}

\newtheorem{Claim}[THEOREM]{Claim}
\newenvironment{claim}{\begin{Claim}}{\end{Claim}}

\newtheorem{Subclaim}[THEOREM]{Subclaim}

\newtheorem{Corollary}[THEOREM]{Corollary}
\newenvironment{corollary}{\begin{Corollary}}{\end{Corollary}}

\newtheorem{Corollary to the proof}[THEOREM]{Corollary to the proof}
\newenvironment{coroll-to-proof}{\begin{Corollary to the proof}}{\end{Corollary to the proof}}

\newtheorem{Proposition}[THEOREM]{Proposition}
\newenvironment{proposition}{\begin{Proposition}}{\end{Proposition}}

\newcommand{\sss}{\hskip2pt}
\newcommand{\ssss}{\hskip1pt}
\newcommand{\tupof}[2]{\langle\ssss #1\sss|\sss #2\ssss\rangle}
\newcommand{\tup}[1]{\langle\ssss #1\ssss\rangle}
\newcommand{\setof}[2]{\{\ssss #1\sss|\sss #2\ssss\}}
\newcommand{\set}[1]{\{\ssss #1\ssss\}}

\newcommand{\name}[1]{\dot{#1}}   
\newcommand{\on}{\upharpoonright}  
\newcommand{\forces}{\Vdash}
\newcommand{\forcesq}[2]{\Vdash_{#1}\hbox{`` }#2\hbox{ ''}}
\newcommand{\decides}{\parallel}

\newcommand{\thinks}{\models}

\font\sevenrm = cmr7

\newcommand{\cl}{\mathop{\rm cl}}

\newcommand{\Lbb}{\mathbb{L}}         
\newcommand{\Pbb}{\mathbb{P}}         
\newcommand{\Qbb}{\mathbb{Q}}         
\newcommand{\Rbb}{\mathbb{R}}         
\newcommand{\Rad}[1]{\Rbb_{#1}}   
\newcommand{\nameRad}[1]{\name{\Rbb}_{#1}} 

\newcommand{\onebb}{\mathbb{1}}
\newcommand{\noughtbb}{\mathbb{0}}

\newcommand{\lh}{\mathop{\rm lh}}   
\newcommand{\cf}{\mathop{\rm cf}}   
\newcommand{\rge}{\mathop{\rm rge}} 
\newcommand{\dom}{\mathop{\rm dom}} 
\renewcommand{\lim}{\mathop{\rm lim}} 
\newcommand{\otp}{\mathop{\rm otp}} 
 
\newcommand{\ssup}{\mathop{\rm ssup}} 
 
\renewcommand{\ni}{\notin}            
\newcommand{\concat}{\kern-.25pt\raise4pt\hbox{$\frown$}\kern-.25pt}
\newcommand{\Card}{\mathop{\rm Card}} 
\newcommand{\On}{\mathop{\rm On}} 
\newcommand{\lk}{\lower.04em\hbox{$<$}\kappa}
\newcommand{\lkplus}{\lower.04em\hbox{$<$}\kappa^+}
\newcommand{\lomegaone}{\lower.04em\hbox{$<$}\omega_1}

\newcommand{\supp}{\mathop{\rm supp}}
\newcommand{\height}{{\rm ht}}

\newcommand{\cardrule}{\hrule height.2pt}
\newcommand{\cardrulefill}{\cleaders\cardrule\hfill}
\newcommand{\cardchar}[1]{\vbox{\ialign{##\crcr
    \cardrulefill\crcr\noalign{\kern1pt\nointerlineskip}
    $\hfil\displaystyle{#1}\hfil$\crcr}}}
\newcommand{\card}[1]{\cardchar{\cardchar{#1}}}

\newcommand{\obar}[1]{\cardchar{#1}}
\renewcommand{\bar}[1]{\obar{#1}}

\font\ninerm=cmr9  \font\eightrm=cmr8  \font\sixrm=cmr6
\font\ninebf=cmbx9 \font\eightbf=cmbx8 \font\sixbf=cmbx6
\font\ninett=cmtt9 \font\eighttt=cmti8
\font\nineit=cmmi9 \font\eightit=cmmi8 \font\sixit=cmmi6
\font\fiveit=cmmi5
\font\ninesl=cmsl9 \font\eightsl=cmsl8
\font\ninesy=cmsy9 \font\eightsy=cmsy8 \font\sixsy=cmsy6
\font\nineex=cmex9 \font\eightex=cmex8 

\catcode`@=11
\newskip\ttglue

\newcommand{\ninepoint}{\def\rm{\fam0\ninerm}
  \textfont0=\ninerm \scriptfont0=\sixrm \scriptscriptfont0=\fiverm
  \textfont1=\nineit  \scriptfont1=\sixit  \scriptscriptfont1=\fiveit
  \textfont2=\ninesy \scriptfont2=\sixsy \scriptscriptfont2=\fivesy
  \textfont3=\nineex  \scriptfont3=\sixex \scriptscriptfont3=\tenex
  \textfont\itfam=\nineit \def\it{\fam\itfam\nineit}%
  \textfont\slfam=\ninesl \def\sl{\fam\slfam\ninesl}%
  \textfont\ttfam=\ninett \def\tt{\fam\ttfam\ninett}%
  \textfont\bffam=\ninebf \scriptfont\bffam=\sixbf
   \scriptscriptfont\bffam=\fivebf \def\bf{\fam\bffam\ninebf}%
  \tt \ttglue=.5em plus.25em minus.15em
  \normalbaselineskip=11pt
  \setbox\strutbox=\hbox{\vrule height8pt depth3pt width0pt}%
  \let\sc=\sevenrm \let\big=\ninebig \normalbaselines\rm}

\newcommand{\nar}{\advance\leftskip by 40pt \advance\rightskip by 40pt}

\newcommand{\eightpoint}{\def\rm{\fam0\eightrm}
  \textfont0=\eightrm \scriptfont0=\sixrm \scriptscriptfont0=\fiverm
  \textfont1=\eightit  \scriptfont1=\sixit  \scriptscriptfont1=\fiveit
  \textfont2=\eightsy \scriptfont2=\sixsy \scriptscriptfont2=\fivesy
  \textfont3=\eightex  \scriptfont3=\tenex \scriptscriptfont3=\tenex
  \textfont\itfam=\eightit \def\it{\fam\itfam\eightit}%
  \textfont\slfam=\tensl \def\sl{\fam\slfam\eightsl}%
  \textfont\ttfam=\eighttt \def\tt{\fam\ttfam\eighttt}%
  \textfont\bffam=\eightbf \scriptfont\bffam=\sixbf
   \scriptscriptfont\bffam=\fivebf \def\bf{\fam\bffam\eightbf}%
  \tt \ttglue=.5em plus.25em minus.15em
  \normalbaselineskip=10pt
  \setbox\strutbox=\hbox{\vrule height8pt depth3pt width0pt}%
  \let\sc=\sevenrm \let\big=\eightbig \normalbaselines\rm}

\title[Forcing constructions]{A framework for forcing constructions at successors of singular cardinals}

\begin{document}

\author[Cummings]{James Cummings}
\address{Department of Mathematical Sciences,
Carnegie Mellon University,
Pittsburgh, PA 15213, USA}

\email{jcumming@andrew.cmu.edu}

\thanks{Cummings thanks the National Science Foundation for their support through grant  DMS-1101156}

\thanks{Cummings, D{\v z}amonja and Morgan thank the Institut Henri Poincar{\' e} for their
support through the ``Research in Paris'' program during the period 24-29 June 2013.}  

\author[D{\v z}amonja]{Mirna D\v zamonja}

\address{School of Mathematics,
University of East Anglia,
Norwich, NR4 7TJ, UK}
\email{M.Dzamonja@uea.ac.uk}

\thanks{D{\v z}amonja, Magidor and Shelah thank the Mittag-Leffler Institute for their support
during the month of September 2009. 
Mirna D{\v z}amonja thanks EPSRC for their support through their grants EP/G068720  and EP/I00498.}

\author[Magidor]{Menachem Magidor}

\address{Institute of Mathematics,
Hebrew University of Jerusalem,
91904 Givat Ram, Israel}

\email{mensara@savion.huji.ac.il}

\author[Morgan]{Charles Morgan}

\thanks{Morgan thanks EPSRC for their support through grant EP/I00498.}

\address{Department of Mathematics,
University College London, Gower St.,
London, WC1E 6BT, UK
and
School of Mathematics,
University of Edinburgh,
Edinburgh, EH9 3FD, UK}
\email{charles.morgan@ucl.ac.uk}

\author[Shelah]{Saharon Shelah}

\address{Institute of Mathematics,
Hebrew University of Jerusalem,
91904 Givat Ram, Israel
and
Department of Mathematics,
Rutgers University,
New Brunswick, NJ 08854, USA}

\email{shelah@math.huji.ac.il}

\thanks{This publication is denoted [Sh963] in Saharon Shelah's list of
publications. Shelah thanks the United States-Israel Binational Science 
Foundation (grant no.~2006108), which partially supported this research.}

\thanks{We thank  Jacob Davis for his attentive reading of and comments on this paper.}

\thanks{We thank the anonymous referee for their meticulous reading of the originally submitted version of this paper.}

\begin{abstract}

We describe a framework for proving consistency results 
about singular cardinals of arbitrary cofinality and their successors.
This framework allows the construction of models in which
the Singular Cardinals Hypothesis fails at a singular cardinal $\kappa$
of uncountable cofinality, while $\kappa^+$ enjoys various  combinatorial
properties.

As a sample application, we prove the consistency (relative to
that of ZFC plus a supercompact cardinal) of there being a strong limit singular
cardinal $\kappa$ of uncountable cofinality where SCH fails and 
such that there is a collection of size
less than $2^{\kappa^+}$ of graphs on $\kappa^+$ such that any graph
on $\kappa^+$ embeds into one of the graphs in the collection.
\end{abstract}



\maketitle

\section*{Introduction}

The class of uncountable regular cardinals is naturally divided into
three disjoint classes: the successors of regular cardinals, the
successors of singular cardinals and the weakly inaccessible
cardinals.  When we consider a combinatorial question about
uncountable regular cardinals, typically these classes require
separate treatment and very frequently the successors of singular
cardinals present the hardest problems. In particular there are subtle
constraints (for example in cardinal arithmetic) on the combinatorics
of successors of singular cardinals, and consistency results in this
area often involve large cardinals.

To give some context for our work, we review a standard strategy for
proving consistency results about the successors of regular
cardinals. This strategy involves iterating $\lk$-closed $\kappa^+$-cc
forcing with $\lk$-supports for some regular cardinal $\kappa$, with
the plan that the whole iteration will also enjoy the $\kappa^+$-chain
condition. The $\kappa^+$-chain condition of the iteration will of
course ensure that all cardinals are preserved, and is also very
helpful in the ``catch your tail'' arguments which frequently appear
in iterated forcing constructions.

When $\kappa = \omega$ this proof strategy is completely
straightforward as any finite support iteration of ccc forcing is ccc;
for regular $\kappa > \omega$ we need to assume that $\kappa^{\lk} =
\kappa$ and that the iterands have some strong form of $\kappa^+$-cc
and some other properties in order to ensure $\kappa^+$-cc for the
iteration (this issue is extensively discussed in
\S\ref{iteration-theory} below).  When $\kappa$ is singular the
strategy is no longer available, and this is one difficulty among many
in proving consistency results involving singular cardinals and their
successors.

D{\v z}amonja and Shelah \cite{Dzamonja-Shelah-univ-models} introduced
a new idea, which we briefly describe.  Initially $\kappa$ is a
supercompact cardinal whose supercompactness is indestructible under
$\lk$-directed closed forcing.  The final model is obtained by a
two-step forcing. The first step, $\Pbb$, is a $\lk$-directed closed
and $\kappa^+$-cc iteration, whilst the second step is Prikry forcing,
$Pr_U$, defined from a normal measure $U$ on $\kappa$ in $V[G]$, where
$G$ is a $\Pbb$-generic filter over $V$.  The iteration $\mathbb P$ is
designed to anticipate and deal with $Pr_U$-names for subsets of
$\kappa^+$, so that after forcing over $V[G]$ with $Pr_U$ we obtain
the desired consistency result.  D{\v z}amonja and Shelah used this
method to obtain the consistency, relative to that of a supercompact cardinal,
 of the existence for $\kappa$ singular strong limit
of cofinality $\omega$, of a family of
${\kappa^{++}}$ many graphs on $\kappa^+$ which are jointly universal
for all graphs on $\kappa^+$. In this model $2^\kappa = 2^{\kappa^+}$
 and this value can be made arbitrarily large. 

A well known early interaction between 
model theory and set theory gives one that if $2^\kappa=\kappa^+$ there is
a saturated graph ({\em i.e.}, model of the theory of graphs) on $\kappa^+$
and any such graph is universal. So the point of results such as that
of D{\v z}amonja and Shelah is to address the possibility of having small 
universal families of graphs on $\kappa^+$ when $2^\kappa>\kappa^+$.

In this paper we build a similar framework in which the final forcing
step is a version of Radin forcing, and changes the cofinality of
$\kappa$ to become some uncountable cardinal less than $\kappa$.
After building the framework, we prove a version of the result on
universal graphs mentioned in the last paragraph.

{
\renewcommand{\theTHEOREM}{\ref{universalg}}
\begin{theorem}
  Suppose $\kappa$ is a supercompact cardinal,
$\lambda<\kappa$ is a regular cardinal and 
$\Theta$ is a cardinal with $\cf(\Theta)\ge\kappa^{++}$ and $\kappa^{+3}\le\Theta$.
There is a forcing extension in which cofinally many cardinals below $\kappa$, $\kappa$ itself and 
all cardinals greater than $\kappa$ are preserved, $\cf(\kappa)=\lambda$,  $2^{\kappa}=2^{\kappa^+} =
\Theta$ and there is a universal family of graphs on $\kappa^+$ of size $\kappa^{++}$.
\end{theorem}
\addtocounter{THEOREM}{-1}
}

The need for large cardinals in the broad context of combinatorics
at singular cardinals and their successors is at least partially explained by the
theory of core models and covering lemmas.  If there is no inner model
with a Woodin cardinal then there is an inner model $K$ with many
strong combinatorial properties (for example GCH and square hold), and
such that $\kappa^+ = (\kappa^+)^K$ for every singular cardinal
$\kappa$. This resemblance between $V$ and $K$ in the absence of inner
models with large enough cardinals exerts a strong influence on the
combinatorics of $\kappa^+$ in $V$, for example it implies that
$\square_\kappa$ holds in $V$ under this hypothesis.

In the instance of the results on universal graphs we obtain
a model in which $\kappa$ is a singular strong limit cardinal where 
$2^\kappa>\kappa^+$, {i.e.} the singular cardinal hypothesis fails. It is
known through work of Gitik, Mitchell, Shelah and Woodin that the consistency strength of the 
singular cardinal hypothesis failing alone is exactly that of a measurable 
cardinal $\kappa$ of Mitchell order $o(\kappa)=\kappa^{++}$. Specifically, in one direction Woodin gave 
a forcing construction of a model in which the singular cardinal hypothesis fails from a certain
large cardinal hypothesis and Gitik showed this hypothesis follows from that of the existence
of a measurable cardinal $\kappa$ of Mitchell order $o(\kappa)=\kappa^{++}$ (\cite{Gitik-negation-of-SCH}).
In the other direction Gitik showed (\cite{Gitik-strength-of-SCH}), building on
Mitchell's Covering Lemma (\cite{Mitchell-core-model-for-seqs-of-measures}) and ideas of Shelah
developed in the course of his pcf theory (\cite{Shelah-card-arith}),
that if the singular cardinal hypothesis fails there is a model in which
there is a measurable cardinal $\kappa$ of Mitchell order $o(\kappa)=\kappa^{++}$.
Thus some large cardinal hypotheses are necessary for these results on universal graphs that we prove.

This paper is in some sense a sequel to the work by D{\v z}amonja and 
Shelah \cite{Dzamonja-Shelah-univ-models},
and so we briefly record some of the  innovations introduced here.
\begin{itemize}
\item  Prikry forcing at $\kappa$ is homogeneous and adds no bounded subsets of $\kappa$,
while Radin forcing has neither of these features. This entails major changes
in the analysis of names for $\kappa^+$ in the final model, the proof of
$\kappa^+$-cc for the main iteration, and the proof that the final model
has a small jointly universal family.
\item The arguments of  \cite{Dzamonja-Shelah-univ-models} involve a complex
iteration scheme which is used to build a Prikry forcing
$Pr_U$,  and would have become even more complex if we had used it to build 
a suitable Radin forcing.  In this paper we use diamond sequences
to achieve similar goals.
\item One of the central points is that our main iteration enjoys a strong form
of the $\kappa^+$-chain condition. None of the standard preservation theorems
were quite suitable to show this, so we took a detour into iteration theory
to formulate and prove a suitable preservation theorem. 
\end{itemize} 

The paper is organized as follows:
\S\ref{iteration-theory} shows that cardinals are preserved in $\lk$-support 
iterations of certain types of $\lk$-closed,  $\kappa^+$-stationary chain condition 
forcings (and as a spin-off we give a generalized Martin's axiom for these forcings). 
\S\ref{preservation} is on preservation of diamond through forcing iterations. 
\S\ref{Radinmat} collects relevant material on Radin forcing.
\S\ref{Pm} describes the long Mathias forcing and its variant for adding 
Radin names for universal graphs. 
\S\ref{chain} contains the proof of the 
stationary $\kappa^+$-chain condition for the forcings of \S\ref{Pm}. 
\S\ref{maini} gives the main iteration. 
\S\ref{small} gets a small family of universal
graphs at a cardinal of uncountable singular cofinality.

The problem of the existence of a small family of universal graphs has some 
independent interest, and is also a natural test question for our forcing framework. 
In \cite{Cummings-Dzamonja-Morgan} the methods of this paper are applied to prove a result parallel to
  Theorem \ref{universalg} concerning graphs on $\aleph_{\omega + 1}$. As part of his thesis work in progress, Jacob Davis has
  also obtained a parallel result for graphs on  $\aleph_{\omega_1 + 1}$.

\subsection*{Notation}

$\Card$ and $\On$ are the classes of cardinals and ordinals
respectively.  The size of a set $A$ is denoted by either $\card{A}$
or $|A|$.  If $X$ is a set and $\kappa$ is a cardinal then 
${\mathcal P}(X)$ is the set of all subsets of $X$ and $[X]^\kappa$ is the set
of all subsets of $X$ of size $\kappa$. We define $[X]^{\lk}$, etc.,
in the obvious way.  If $X$, $Y$ are sets then ${}^X Y$ is the set of
functions from $X$ to $Y$. As in \cite{Kunen}, Definition (VII.6.1), if 
$I$ and $J$ are sets and $\kappa$ is a cardinal we write $\hbox{Fn}(I,J,\kappa)$
for the set of partial functions from $I$ to $J$ of size less than $\kappa$,
$\setof{p}{p\hbox{ is a function } \sss\sss\&\sss\sss 
\card{p}<\kappa \sss\sss\&\sss\sss
\dom(p)\subseteq I \sss\sss\&\sss\sss
\rge(p)\subseteq J}$.

We designate names for the projection functions on cartesian products. 
If $X$, $Y$ are sets we define $\pi_0:X\times Y \longrightarrow X$ and 
$\pi_1:X\times Y \longrightarrow Y$ by, if $x\in X$ and $y\in Y$ then 
$\pi_0(x,y)=x$ and $\pi_1(x,y)=y$.

We write $f \cdot g$ for the composition of the functions $f$ and $g$. 
When $f$ and $g$ agree on $\dom(f) \cap \dom(g)$ we write $f + g$ for the
unique function $h$ such that $\dom(h) = \dom(f) \cup \dom(g)$,
$h \restriction \dom(f) = f$ and  $h \restriction \dom(g) = g$.

If $f:A\times B \longrightarrow C$ is a function, $a\in A$ and $b\in B$ we write
$f(a,\ssss .\ssss)$ for the function $g:B\longrightarrow C$ such that $g(d)=f(a,d)$ for all $d\in B$ and
$f(\ssss .\ssss,b)$ for the function $h:A\longrightarrow C$ such that $h(c)=f(c,b)$ for all $c\in A$.

If $X$ is a set of ordinals we write $\ssup(X)$ for the 
{\em strong supremum of $X$,} that is the least ordinal $\alpha$ such that
$X\subseteq \alpha$, and $\cl(X)$ for the closure of $X$ in the order
topology. If $\alpha$, $\beta$ are ordinals we write $\alpha \beta$ for their 
ordinal multiplication product.

One \emph{non-standard} piece of notation which will be useful and
which we define here is the following.  If
$y=\tupof{y_{\tau}}{\tau<\lh(y)}$ is a sequence of sets and
$y_0\in\Card$ we often write $\kappa_y$ for $y_0$, the first entry in $y$.

If $\mu$, $\kappa$ are regular cardinals with $\mu<\kappa$ we write,
on occasion, $S^{\kappa}_\mu$ for $\setof{\xi<\kappa}{\cf(\xi)=\mu}$,
and similarly for $S^{\kappa}_{<\mu}$, $S^\kappa_{\ge\mu}$ and so
on. If $S$ is a stationary subset of $\kappa$ we say that $T\subseteq
S$ is a \emph{club relative to} $S$ if there is some club
$C\subseteq\kappa$ such that $T=S\cap C$.

A function $g:\kappa\longrightarrow\kappa$ is \emph{regressive} on a
set $S \subseteq \kappa$ if for every $\varepsilon \in S$ we have
$g(\varepsilon)<\varepsilon$.

A subset $X$ of a partial order $(\mathbb{P},\le)$ is 
\emph{linearly ordered} if for all $p$, $q\in X$ either $p\le q$ or $q\le p$,
\emph{directed} if every finite subset of $X$ has a lower bound in
$X$, and \emph{centred} if every finite subset of $X$ has a lower
bound in $\Pbb$.  A partial order $(\mathbb{P},\le)$ is
$\lk$-\emph{closed} if every linearly ordered subset of size $\lk$ of
$\mathbb{P}$ has a lower bound, $\lk$-\emph{directed closed} if every
directed subset of size $\lk$ of $\Pbb$ has a lower bound, and
$\lk$-\emph{compact} if every centred subset of size $\lk$ of $\Pbb$
has a lower bound.  $\Pbb$ is {\em countably compact} if and only if
it is $\lomegaone$-compact.  Clearly any $\lk$-compact partial order
is $\lk$-directed closed, and any $\lk$-directed closed partial order
is $\lk$-closed.

If $(\Pbb,\le)$ is a partial order and $q$, $r\in \Pbb$ we write
$q\parallel r$ to mean that $q$ and $r$ are compatible in $\Pbb$, that
is there is some $p\in \Pbb$ such that $p\le q$, $r$.  If $(\Qbb,\le)$
is a sub-partial order of $(\Pbb,\le)$ and $q$, $r\in \Qbb$ we write
$q\parallel_\Qbb r$ to mean $q$ and $r $ are compatible 
\emph{in $\Qbb$}, that is there is some $p\in \Qbb$ such that $p\le q$, $r$. A
partial order is {\em splitting} if every element has two incompatible
extensions.

A partial order $(\Pbb,\le)$ is \emph{well-met} if every compatible
pair of elements has a greatest lower bound: \emph{i.e.}, for all $p$, $q\in \Pbb$
if $p\parallel q$ then there is some $r\in \Pbb$ such that $r\le p$,
$q$ and for all $s\in \Pbb$ with $s\le p$, $q$ we have $s\le r$.

For a regular cardinal $\kappa$, a partial order $\Pbb$ has the
$\kappa^+$-\emph{stationary chain condition} (abbreviated later as the $\kappa^+$-\emph{stationary cc})
  if and only if for every
sequence $\tupof{ p_i }{ i < \kappa^+ }$ of conditions in $\Pbb$ there
is a club set $C \subseteq \kappa^+$ and there is a regressive function $f$
on $C \cap S^{\kappa^+}_\kappa$ such that for all $\alpha,\, \beta \in
C \cap S^{\kappa^+}_\kappa$ with $f(\alpha) = f(\beta)$ the conditions
$p_\alpha$ and $p_\beta$ are compatible. We note that by an easy
application of Fodor's lemma, this property implies that $\Pbb$ enjoys
the strengthened form of the $\kappa^+$-Knaster property in which any
$\kappa^+$-sequence of conditions has a stationary subsequence of
pairwise compatible conditions.
A stronger notion (see Lemma (\ref{sh-conds-imply-double-star}) and Example 
(\ref{Cohen-forcing-is-statcc-but-not-kappa-linked}) below) is that 
$\Pbb$ is $\kappa$-\emph{linked} if $\Pbb$ is the union of 
$\kappa$ many sets of pairwise compatible elements.

When forcing with a partial order $(\Pbb,\le)$ over a model $V$ we
take the notions of \emph{names} and \emph{canonical names} to be as
in \cite{Kunen}, writing $\name{x}$ for a $\Pbb$-name in $V$ and
$\hat{y}$ for a standard $\Pbb$-name for $y\in V$. We will freely use
the well-known {\em Maximum Principle}, which states that if $\forces
\exists x \phi(x, \hat{y}, \dot G)$ then there is a name $\name{x}$
such that $\forces \phi(\name{x}, \hat{y}, \dot G)$.  Usually our
usage should be clear: as is customary we take the Boolean truth
values $\noughtbb$, $\onebb$ to be identified with the (ordinal)
elements of $\set{0,1}$. Thus forcing names for truth values are names
for ordinals $<2$. In the context of forcing, if $p$, $q\in \Pbb$ then
$p\le q$ means that $p$ is stronger than $q$.

Cohen forcing to add $\lambda$ many subsets of a regular cardinal $\kappa$, which we denote by
$\hbox{Add}(\kappa,\lambda)$, is the collection of partial functions of size less than $\kappa$ 
from $\lambda$ to $2$, $\hbox{Fn}(\lambda,2,\kappa)$, ordered by 
reverse inclusion.

The Laver iteration $\Lbb$ for a supercompact cardinal $\kappa$ is a
$\kappa$-cc forcing poset of cardinality $\kappa$ which make the
supercompactness of $\kappa$ indestructible under $\lk$-directed
closed forcing \cite{Laver}.

If $\kappa^+$ is a successor cardinal a $\kappa^+$-{\em tree} is a tree of
size $\kappa^+$, height $\kappa^+$ and each level of size at most $\kappa$.  
A \emph{binary $\kappa^+$-tree} is a $\kappa^+$-tree 
such that each point in the tree has exactly two successors in the tree order.
If $\Upsilon$ is a cardinal greater than $\kappa^+$ a 
{\em $\kappa^+$-Kurepa tree with $\Upsilon$-many branches} is a
$\kappa^+$-tree with $\Upsilon$-many branches of
length $\kappa^+$. If there is such a tree there is one which is binary and a subtree
of the complete binary tree $^{<\kappa^+} 2$ (see, for example, \cite{Kunen},
Chapter (2)).  In a slight abuse of terminology, we refer here to binary
$\kappa^+$-Kurepa subtrees of $^{<\kappa^+} 2$ as binary $\kappa^+$-Kurepa trees.

The {\em usual forcing to add a $\kappa^+$-Kurepa tree with $\Upsilon$-many branches}, due to Stewart \cite{Stewart}, 
is the forcing notion in which conditions are pairs $(t,f)$ consisting of a binary sub-tree, $t$, of $^{<\kappa^+} 2$
of successor height, say $\gamma+1 < \kappa^+$, each level of size at most $\kappa$ and 
each point in the tree having a successor of height $\gamma$, and a bijection $f$ between a subset of $\Upsilon$ and the set of 
points of height $\gamma$. The ordering is that $p = (t^p,f^p) \le q = (t^q,f^q)$ 
if $t^q$ is obtained by,  in Kunen's \cite{Kunen} vivid pr\'ecis, `sawing off $t^p$ parallel to the ground': \emph{i.e.}
there is some $\alpha < \kappa^+$ such that $t^q = t^p \cap {}^\alpha 2$,
$\dom(f^q)\subseteq \dom(f^p)$ and
$f^q(\xi) <^{p} f^p(\xi)$, where $<^p$ is the tree order on $t^p$.
This forcing has the $\kappa^{++}$-chain condition and is $\lkplus$-directed closed with greatest lower bounds:
every directed set of conditions of size at most $\kappa$ has a greatest lower bound.

We recall here in one place, for the reader's convenience, the definitions of various variants of the $\diamondsuit$ principle.
The reader can also consult, for example, \cite{Kunen} and Rinot's article \cite{Rinot-diamond}
for further information about these principles.
Let $\kappa$ be a regular cardinal and $B$ a stationary subset of $\kappa$. 

A sequence $\tupof{A_\alpha}{\alpha\in B}$ is a
$\diamondsuit_\kappa(B)$-\emph{sequence} if $A_\alpha\subseteq \alpha$ for each $\alpha\in B$ and whenever
$A\subseteq \kappa$ the sequence frequently predicts $A$ correctly: $\setof{\alpha\in B}{A\cap \alpha = A_\alpha}$ is a stationary subset of $\kappa$.

A sequence $\tupof{{\mathcal A}_{\alpha}}{\alpha\in B}$ is a $\diamondsuit^*(B)$-\emph{sequence} if 
${\mathcal A}_\alpha\in [{\mathcal P}(\alpha)]^{\le \card{\alpha}}$ for each $\alpha<\kappa$ and whenever
$A\subseteq \kappa$ there is a closed unbounded subset $C$ of $\kappa$ such that
$C\cap B\subseteq \setof{\alpha\in B}{A\cap \alpha\in {\mathcal A}_\alpha}$.

A sequence $\tupof{{\mathcal A}_{\alpha}}{\alpha\in B}$ is a $\diamondsuit^+(B)$-\emph{sequence} if 
${\mathcal A}_\alpha\in [{\mathcal P}(\alpha)]^{\le \card{\alpha}}$ for each $\alpha<\kappa$ and whenever
$A\subseteq \kappa$ there is a closed unbounded subset $C$ of $\kappa$ such that
$C\cap B\subseteq \setof{\alpha\in B}{A\cap \alpha,\sss\sss C\cap\alpha \in {\mathcal A}_\alpha}$.

We say that $\diamondsuit_\kappa(B)$ \emph{holds} (\emph{resp.} $\diamondsuit^*_\kappa(B)$ \emph{holds}, 
$\diamondsuit^+_\kappa(B)$ \emph{holds}) if there is a $\diamondsuit_\kappa(B)$-sequence 
(\emph{resp.} a $\diamondsuit^*_\kappa(B)$-sequence, a $\diamondsuit^+_\kappa(B)$-sequence). 

When $B$ is the whole of $\kappa$, \emph{i.e.}, $B=\setof{\alpha}{\alpha<\kappa}$, we omit mention of it. 
So, for example, a $\diamondsuit_\kappa$-sequence is a $\diamondsuit_\kappa(\setof{\alpha}{\alpha<\kappa})$-sequence and so on.

A $\diamondsuit_\kappa(B)$-sequence \emph{on} $\kappa\times\kappa$ is a sequence $\tupof{A_\alpha}{\alpha\in B}$ such that
$A_\alpha\subseteq \alpha\times\alpha$ for each $\alpha\in B$ and whenever $A\subseteq \kappa\times\kappa$ one has that
$\setof{\alpha\in B}{A\cap \alpha\times\alpha = A_\alpha}$ is stationary in $\kappa$. 

Clearly, if there is a $\diamondsuit_\kappa(B)$-sequence then there is
a $\diamondsuit_\kappa(B)$-sequence on $\kappa\times\kappa$ -- the
closure points of any enumeration of $\kappa\times\kappa$ in order
type $\kappa$ form a closed unbounded subset of $\kappa$. We talk
similarly of $\diamondsuit^*_\kappa(B)$-sequences and
$\diamondsuit^+_\kappa(B)$-sequences \emph{on} $\kappa\times\kappa$.

If ${\mathcal G}=(X,E)$ and ${\mathcal H}=(Y,F)$ are graphs then
$f:{\mathcal G}\longrightarrow {\mathcal H}$ is an \emph{embedding of
${\mathcal G}$ into ${\mathcal H}$ as an induced subgraph} if $\forall x_0\in
X\sss\sss\forall x_1\in X\sss (\sss x_0\mathrel{E} x_1
\longleftrightarrow f(x_0)\mathrel{F} f(x_1) \sss)$.  Since this is
the only kind of graph embedding which concerns us we will simply call
them {\em embeddings}.

Let $\kappa$ be a regular cardinal. We say that a family ${\mathcal F}$ 
of graphs on $\kappa$ is {\em jointly universal} for graphs of size 
$\kappa$ if for every graph ${\mathcal G}=(\kappa,E)$ there is some 
${\mathcal H}=(\kappa,F)\in {\mathcal F}$ and some embedding 
$f:{\mathcal G}\longrightarrow {\mathcal H}$ of ${\mathcal G}$ into ${\mathcal H}$. 
We say that ${\mathcal F}$ is a {\em small} universal
family if $\card{{\mathcal F}} < 2^\kappa$.

\section{Some iterated forcing theory}\label{iteration-theory} 

In the classical exposition of iterated forcing Baumgartner
\cite[\S{}4]{Baumgartner} wrote
\begin{quote}
The search for extensions of MA for larger cardinals has proved to be
rather difficult.
\end{quote}
One reason for this is that for $\kappa$ regular and uncountable, an
iteration of $\lk$-closed and $\kappa^+$-cc forcing posets with
supports of size less than $\kappa$ does not in general have the
$\kappa^+$-cc.  For example, a construction due to Mitchell described
in a paper of Laver and Shelah \cite{Laver-Shelah} shows that in $L$
there is an iteration of length $\omega$ of countably closed
$\aleph_2$-cc~forcing posets such that the inverse limit at stage
$\omega$ does not have the $\aleph_2$-cc.

The literature contains several preservation theorems for iterations
involving strengthened forms of closure and chain condition, along
with corresponding forcing axioms. The first results in this direction
are in unpublished work by Laver \cite[\S{4}]{Baumgartner}.
Baumgartner \cite{Baumgartner} proved that under CH an iteration with
countable supports of countably compact $\aleph_1$-linked forcing
posets is $\aleph_2$-cc, and proved the consistency of some related
forcing axioms.  Shelah \cite{Shelah-GMA} proved that under CH an
iteration with countable supports of posets which are countably closed
and well-met and which enjoy the $\aleph_2$-stationary cc also enjoys
the $\aleph_2$-stationary cc.  Shelah also proved more general results
for certain iterations of $\kappa^+$-stationary cc~posets with
supports of size less than $\kappa$, and proved the consistency of a
number of related forcing axioms.

The main result of this section is a common generalisation of the
results of Baumgartner and Shelah quoted above.  In order to state the
theorem we require a definition.

\begin{definition}\label{compat-omega-seqs} 
Let $(\Pbb,\le)$ be a partial order.
\begin{itemize}
\item Two descending sequences $\tupof{ q_i }{ i < \omega }$ and
$\tupof{ r_i }{ i < \omega }$ from $\Pbb$ are \emph{pointwise
compatible} if for each $i < \omega$ we have $q_i \parallel r_i$.
\item $(\Pbb,\le)$ is {\em countably parallel-closed} if each pair of
pointwise compatible descending $\omega$-sequences has a common
lower bound.
\end{itemize}
\end{definition}

\begin{theorem}\label{double-star-and-statcc-itn-thm}
Let $\kappa$ be an uncountable regular cardinal with $\kappa^{<\kappa} =
\kappa$. Every iteration of countably parallel-closed, $\lk$-closed, 
$\kappa^+$-stationary cc forcing with supports of size less than
$\kappa$ has the $\kappa^+$-stationary cc.
\end{theorem}

\begin{proof}[\ref{double-star-and-statcc-itn-thm}]

Let $\Pbb=\tup{\tupof{\Pbb_\xi}{\xi\le\chi},\tupof{\name{\Qbb}_\xi}{\xi<\chi}}$
be a $\lk$-support iteration of forcings such that for each $\xi<\chi$
it is forced by $\Pbb_\xi$ that $\name{\Qbb}_\xi$ is countably
parallel-closed, $\lk$-closed and $\kappa^+$-stationary cc.  We prove,
by induction on $\xi$, and following the proof of Lemma (1.3) of
\cite{Shelah-GMA} closely, that for each $\xi\le\chi$, $\Pbb_\xi$ has
the stationary $\kappa^+$-chain condition.

Suppose that $\Pbb_\varepsilon$ has the stationary $\kappa^+$-chain
condition for all $\varepsilon<\xi$.  Let $\tupof{p_i}{i<\kappa^+}\in
{}^{\kappa^+} \Pbb_\xi$. We construct for each $i<\kappa^+$ a
decreasing $\omega$-sequence $\tupof{p^n_i}{n<\omega}$ such that
$p^0_i=p_i$ and $p^n_i\in \Pbb_\xi$ for each $n<\omega$.

Induction step $n+1$. Suppose we have already defined
$\tupof{p^n_i}{i<\kappa^+}$. For each $i<\kappa^+$ and
$\varepsilon<\xi$ we have $\forces_{P_\varepsilon}{\hbox{``}
  p^n_i(\varepsilon) \in\name{\Qbb}_\varepsilon\hbox{''}}$ and
{$p^n_i(\varepsilon)=\onebb_{\name{q}_\varepsilon}$}\break if
$\varepsilon \ni \supp(p^n_i)$.  As
$\forces_{P_\varepsilon}{\hbox{``$\name{\Qbb}_\varepsilon$ has
the $\kappa^+$-stationary cc''}}$ and $\tupof{p^n_i(\varepsilon)}{i <
\kappa^+}$ names a $\kappa^+$-sequence of elements of
$\name{\Qbb}_\varepsilon$, we may find $\Pbb_\varepsilon$-names
$\name{C}^n_\varepsilon$ for a club subset of $\kappa^+$ and
$\name{g}^n_\varepsilon$ for a regressive function on
$\name{C}^n_\varepsilon$ witnessing the $\kappa^+$-stationary cc for
the sequence named by $\tupof{p^n_i(\varepsilon)}{i < \kappa^+}$.

As $\Pbb_\varepsilon$ has the $\kappa^+$-chain condition, if
$\forcesq{\Pbb_\varepsilon}{\name{D}\hbox{ is a club subset of  }\kappa}$ 
there is some club $D'\in V$ such that
$\forcesq{\Pbb_\varepsilon}{D'\subseteq\name{D}}$. So we may as well
assume that $\name{C}^n_\varepsilon=\hat{C}^n_\varepsilon$ for some
club $C^n_\varepsilon\in V$.

For each $i<\kappa^+$, dealing with each $\varepsilon\in \supp(p^n_i)$
inductively and using the $\lk$-closure of $\Pbb_\xi$, find some
$p^{n+1}_i\le p^n_i$ such that
\[ \forall \varepsilon\in \supp(p^n_i) \sss \sss p^{n+1}_i \on\varepsilon
\forcesq{\Pbb_\varepsilon}{\name{g}^n_\varepsilon(i)= \rho^n_\varepsilon(i)}\]
for some ordinal $\rho^n_\varepsilon(i) < i$. Let 
$\rho^n_\varepsilon(i) = 0$ for $\varepsilon\ni \supp(p^n_i)$.

Let  $\setof{ \varepsilon_\alpha }{ \alpha < \mu }$, for some $\mu \le \kappa^+$, enumerate
$\bigcup\setof{\supp( p^n_i )}{ n<\omega \sss\sss\&\sss\sss i < \kappa^+ } $.
For each $\varepsilon<\xi$ let $C_\varepsilon = \bigcap_{n<\omega} C^n_\varepsilon$
and let $C= \setof{i<\kappa^+}{ \forall \alpha < i\sss\sss (i \in C_{\varepsilon_\alpha})}$.

\begin{claim}\label{delta-sys-technical-claim}
Using the hypothesis that $\kappa^{<\kappa} = \kappa$,
we can find a club $E \subseteq C$ and a regressive function $k$ on $S^{\kappa^+}_\kappa \cap E$ 
such that if
$i$, $i'\in S^{\kappa^+}_\kappa \cap E$, $k(i)=k(i')$ and $i<i'$ then 
\begin{align*}
(1)\quad & \bigcup\setof{\supp(p^n_i)}{n<\omega} \cap \setof{\varepsilon_\gamma}{\gamma<i} =\\
&\quad \bigcup\setof{\supp(p^n_{i'})}{n<\omega} \cap \setof{\varepsilon_\gamma}{\gamma<i'}\\
(2)\quad &  \bigcup\setof{\supp(p^n_i)}{n<\omega}\subseteq \setof{\varepsilon_\gamma}{\gamma<i'}\\
(3)\quad  &  \hbox{If $\gamma < i'$, $n<\omega$ and 
$\varepsilon_\gamma \in \supp(p^n_{i'})$ then
$\rho^n_{\varepsilon_\gamma}(i) = \rho^n_{\varepsilon_\gamma}(i')$.} \\
\end{align*}
\vskip-24pt
(Note that, by (1), for every $\gamma$ to which clause (3) applies we have $\gamma < i$ and
$\varepsilon_\gamma \in \bigcup\setof{\supp(p^n_i)}{n<\omega}$.)
\end{claim}

\begin{proof}[\ref{delta-sys-technical-claim}]
We start by making some auxiliary definitions.

Let $H=\setof{h}{\dom(h) \in [\kappa^+]^{<\kappa} \hbox{ and } h:\dom(h)\longrightarrow {}^\omega\kappa^+}$. 
By the hypothesis that $\kappa^{<\kappa}=\kappa$ we have that $\card{H}=\kappa^+$. 
So we can enumerate $H$ as, say, $\tupof{h_\eta}{\eta<\kappa^+}$.

Noticing that $H\subset [\kappa^+ \times {}^\omega\kappa^+]^{<\kappa}$,
for $i<\kappa^+$, define $H_i = H\cap [i \times {}^\omega i]^{<\kappa}$.
Since $\kappa^{<\kappa}=\kappa$ we have, for each $i<\kappa$, that $\card{H_i}=\kappa$.

So we can define $f:\kappa^+\longrightarrow \kappa^+$ by, for $i<\kappa^+$,
 letting $f(i)=$ the least $\tau<\kappa^+$ such that 
$H_i\subseteq \setof{h_\eta}{\eta<\tau}$. 

Also, for each $i<\kappa^+$, set 
$F_i = \setof{\gamma<\kappa^+}{\varepsilon_\gamma \in \bigcup\setof{\supp(p^n_{i})}{n<\omega} }$
and $D_i = F_i \cap i$.

Let $E=\setof{i<\kappa^+}{f``i \cup \bigcup \setof{F_j }{j<i} \subseteq i }$. It follows that $E$ is a closed unbounded subset of $\kappa^+$. 

In order to see this let us define for $n<\omega$ the functions $g_n:\kappa^+\longrightarrow \kappa^+$ by 
$g_0(i)$ is the least ordinal $\obar{i}$ such that
$f``i \cup \bigcup \setof{F_j }{j<i} \subseteq \obar{i}$, and $g_{n+1}(i) = g_0(g_n(i))$ for $n<\omega$. 
Then for any $i<\kappa^+$ we have that $i< \bigcup_{n<\omega} g_n(i) \in E$. 
This shows that $E$ is unbounded. Moreover, if $B\subseteq E$ is of limit order type less than $\kappa^+$
and $d \in \sup(B) \cup \bigcup \setof{F_j }{j<\sup(B)}$ then there is some $b\in B$ such that
$d \in b \cup \bigcup \setof{F_j }{j<b}$ and hence $f(d)< b <\sup(B)$. 

Now define, for each $i<\kappa^+$, the function $k_i:D_i\longrightarrow {}^\omega i$ by letting
$k_i(\gamma)(n) = \rho^n_{\varepsilon_{\gamma}}(i)$ if $\varepsilon_\gamma \in \supp(p^n_i)$, and $=0$ otherwise.

So for each $i<\kappa^+$ we have $k_i\in H_i$.

Define $k:\kappa^+\longrightarrow \kappa^+$ by setting $k(i) =$ that $\eta$ such that $k_i= h_\eta$. 

For each $i<\kappa^+$ we have $k_i = h_{k(i)}$ and thus $k(i) < f(i)$. 

Now if $i\in E$ and $\cf(i)=\kappa$, simply because $k_i$ is a function of size $\lk$, 
there is some $i^* <i$ such that $k_i \in H_{i^*}$. Hence for such $i$ we have $k(i) < f(i^*)$. 
We also have $f(i^*)< i$ since $i\in E$. 

Hence if $i\in E$ and $\cf(i)=\kappa$ we have $k(i)<i$ and so we have shown that $k$ is regressive on $E\cap S^{\kappa^+}_{\kappa}$. 

Moreover, if $i$, $i' \in E\cap S^{\kappa^+}_\kappa$ and $k(i) = k(i')$ we have that $k_i = k_{i'}$, and hence that
(1) and (3) hold. Finally, if $i< i'$ as well we have, since $i'\in E$, that $F_i\subseteq i'$ and hence (2) holds.
\end{proof}

With Claim (\ref{delta-sys-technical-claim}) in hand, suppose that $i<i'<\kappa^+$, $i$, $i'\in S^{\kappa^+}_\kappa \cap E$ and $k(i)=k(i')$.
We now construct, by an induction of length $\xi$, a condition $q\in
\Pbb_\xi$ which is a common refinement of $p^n_i$ and $p^n_{i'}$ for
all $n$, and hence of $p_i$ and $p_{i'}$. The support of $q$ will be
the union of $\bigcup\setof{\supp(p^n_i)}{n<\omega}$ and $\bigcup\setof{\supp(p^n_{i'})}{n<\omega}$.

For $\sigma \in \bigcup\setof{\supp(p^n_i)}{n<\omega} 
\setminus  \bigcup\setof{\supp(p^n_{i'})}{n<\omega}$, as\break 
$\forcesq{\Pbb_\sigma}{\name{\Qbb}_\sigma \hbox{ is countably closed}}$, take 
$q(\sigma)$ to be any $r$ such that 
$q\on\sigma\forcesq{\Pbb_\sigma}{ \forall n < \omega \sss\sss (r\le p^n_i(\sigma))}$.
Similarly, for 
$\sigma \in \bigcup\setof{\supp(p^n_{i'})}{n<\omega}$ 
$ \setminus  \bigcup\setof{\supp(p^n_{i})}{n<\omega}$,
take $q(\sigma)$ to be any $r$ such that 
$q\on\sigma\forcesq{\Pbb_\sigma}{ \forall n < \omega \sss\sss (r\le p^n_{i'}(\sigma))}$.

Finally, if $\sigma\in \bigcup\setof{\supp(p^n_i)}{n<\omega} 
\cap\bigcup\setof{\supp(p^n_{i'})}{n<\omega}$ 
then $\sigma=\varepsilon_\gamma$ for some $\gamma< i$ (by conditions (1) and (2) above), 
for each $n<\omega$ we have $\rho^n_{\varepsilon_\gamma}(i) = \rho^n_{\varepsilon_\gamma}(i')$ 
(by (3) above),  and $i$, $i'\in C_{\varepsilon_\gamma}$
(by the construction of $C$ and the choice of $E$). 

By construction, for each $n<\omega$ we have
$q\on \sigma \forcesq{\Pbb_\sigma}{p^n_i(\sigma) \parallel_{\name{\Qbb}_\sigma} p^n_{i'}(\sigma)}$. 
By the fact that 
$\forcesq{\Pbb_\sigma}{\name{\Qbb}_\sigma\hbox{ is countably parallel-closed}}$, 
we can choose $q(\sigma)$ to be some $r$ such that 
$q\on \sigma \forcesq{\Pbb_\sigma}{\forall n<\omega \sss\sss (r\le p^n_i(\sigma),
\sss\sss p^n_{i'}(\sigma))}$. Thus 
for all $n<\omega$ we have that  
$q\on(\sigma + 1) \le p^n_i\on (\sigma+1)$, $p^n_{i'}\on (\sigma+1)$. \break $\vphantom{wwwww}$
\end{proof}

We remark that the previous theorem allows us to give `generalised
Martin's axiom' forcing axioms similar to those formulated by
Baumgartner \cite{Baumgartner} and Shelah \cite{Shelah-GMA}. One
example is given by the following theorem.
\begin{theorem} Let $\kappa$ be an uncountable regular cardinal such that 
$\kappa^{<\kappa} = \kappa$, and let $\lambda > \kappa^+$ be a cardinal 
such that $\gamma^{<\kappa} < \lambda$ for every $\gamma < \lambda$. Then
there is a $\lk$-closed and $\kappa^+$-stationary cc forcing poset $\Pbb$ 
such that if $G$ is $\Pbb$-generic, then in $V[G]$ we have 
$2^\kappa = \lambda$ and the following forcing axiom holds:

For every  poset $\Qbb$ which is $\lk$-closed, countably parallel-closed and
$\kappa^+$-stationary cc, every $\gamma < \lambda$  and every sequence 
$\tupof{D_i }{ i < \gamma }$ of dense
subsets of $\Qbb$ there is a filter on $\Qbb$ which meets each set $D_i$.

\end{theorem}
As in \cite{Baumgartner} and \cite{Shelah-GMA}, there are variations
with weaker hypotheses on $\lambda$, yielding weaker forcing axioms
which only apply to posets $\Qbb$ of bounded size.

The following straightforward lemma indicates some connections between the hypotheses used
in \cite{Baumgartner} and \cite{Shelah-GMA}, and the ones used here.

\begin{lemma}\label{sh-conds-imply-double-star}
Let $\Pbb$ be a forcing poset.
\begin{enumerate}
\item \label{x1}  If $\Pbb$ is well-met and countably closed then it is countably compact.
\item \label{x2} If $\Pbb$ is countably compact then it is countably parallel-closed.
\item  \label{x3} If $\Pbb$ is $\kappa$-linked then it is  $\kappa^+$-stationary cc.
\end{enumerate}
\end{lemma}

\begin{proof}[\ref{sh-conds-imply-double-star}]

For (\ref{x1}), note that by an easy induction any finite subset of a well-met poset with a common lower bound has a 
greatest lower bound. Now suppose that $C \subseteq \Pbb$ is a countable centred set and
enumerate it as $\tupof{ p_n }{ n < \omega }$, then let $q_m$ be the greatest lower bound
for $\tupof{ p_n }{ n \le m }$; the conditions  $q_m$ form a decreasing sequence which has a lower bound
by countable closure. 

For (\ref{x2}), suppose $(\Pbb,\le)$ is countably compact and 
that  $Q = \tupof{ q_i }{ i < \omega }$ and
$R = \tupof{ r_j }{ j < \omega }$ are two pointwise
compatible descending sequences from $\Pbb$.
If $i_0 < i_1<\dots<i_m <\omega$ and $j_0<j_1<\dots<j_n<\omega$
let $k = \max(\set{i_m,j_n})$ and note that since
$q_k \parallel r_k$ there is some $p\in \Pbb$ such that 
for all $l\le m$ and all $e\le n$ we have $p\le q_{i_l}$, $r_{j_e}$.
That is, any finite subset of $Q\cup R$ has a lower bound.
Hence, by the countable compactness of $(\Pbb,\le)$, we have
that $Q\cup R$ has a lower bound. Thus we have shown
that $(\Pbb,\le)$ is countably parallel-closed.

For (\ref{x3}), suppose $\Pbb$ is $\kappa$-linked and that $\tupof{A_\gamma}{\gamma<\kappa}$ 
is a partition of $\Pbb$ such that each $A_\gamma$ is a pairwise compatible subset.
Let $\tupof{p_i}{i<\kappa^+}\subseteq \Pbb$. Define $f:S^{\kappa^+}_\kappa\longrightarrow \kappa$ by 
$f(i)=\gamma$ if and only if $p_i\in A_\gamma$ for $i\in S^{\kappa^+}_{\kappa}$. Then $f$ is a regressive function 
and if $f(i)=f(j)$ we have that $p_i$ and $p_j$ are compatible.
\end{proof}

One cannot reverse Lemma (\ref{sh-conds-imply-double-star}.(\ref{x3})). 
For example, it is a folklore result that $\hbox{Add}(\kappa,\lambda)$ is $\kappa^+$-stationary cc if $\kappa^{<\kappa}=\kappa$
but is not $\kappa$-linked if $2^\kappa<\lambda$. We give proofs of these facts for completeness. 
The reader may also find the proof of the former useful as a dry run for the considerably more elaborate proof 
in \S\ref{chain} that the forcing $Q(w)$ of \S\ref{Pm} is $\kappa^+$-stationary cc;
the notation used here closely mirrors that used there.

\begin{example}\label{Cohen-forcing-is-statcc-but-not-kappa-linked} Let $\kappa$ and $\lambda$ be cardinals, 
with $\kappa$ regular, $\kappa^{<\kappa}=\kappa$ and $2^\kappa<\lambda$. 
Then $\hbox{Add}(\kappa,\lambda)$ is not $\kappa$-linked, however it is $\kappa^+$-stationary cc.
\end{example}

\begin{proof}[\ref{Cohen-forcing-is-statcc-but-not-kappa-linked}]
In order to see that $\hbox{Add}(\kappa,\lambda)$ is not $\kappa$-linked let us
suppose that $\tupof{B_i}{i<\kappa}$ is a collection of $\kappa$ many pairwise compatible subsets of $\hbox{Fn}(\lambda,2,\kappa)$. 
For each $i<\kappa$ and $\alpha<\lambda$ we have that 
if $p$, $q\in B_i$ and $\alpha\in\dom(p)\cap \dom(q)$ then $p(\alpha)=q(\alpha)$.
Hence for each $i<\kappa$ there is a function $f_i\in {^\lambda{ 2}}$ such that
for each $p\in B_i$ we have $p \subseteq f_i$. 

For each $\alpha<\lambda$ define $d_\alpha:\kappa\longrightarrow 2$ by $d_\alpha(i) = f_i(\alpha)$. 
Since $2^\kappa<\lambda$ there are $\alpha$, $\beta<\lambda$ such that $d_\alpha=d_\beta$. 
Now consider $p\in \hbox{Fn}(\lambda,2,\kappa)$
such that $\alpha$, $\beta\in \dom(p)$ and $p(\alpha)\ne p(\beta)$. Then for each $i<\kappa$ we cannot have that both
$p(\alpha)=f_i(\alpha)$ and $p(\beta)=f_i(\beta)$. Hence $\bigcup\setof{B_i}{i<\kappa}\ne \hbox{Fn}(\lambda,2,\kappa)$.

We now verify that $\hbox{Add}(\kappa,\lambda)$ is $\kappa^+$-stationary cc. 

Suppose $\tupof{p^i}{i<\kappa^+}$ is 
a collection of conditions in $\hbox{Add}(\kappa,\lambda)$. For each $i<\kappa^+$ let ${\mathcal t^i}=\dom(p^i)$. 
Let $\setof{\alpha_\gamma}{\gamma<\gamma^*}$ be an enumeration of $\bigcup \setof{{\mathcal t^{i}}}{i<\kappa^+}$, for some 
$\gamma^*\le \kappa^+$. 

Next, for each $i<\kappa^+$ let
$\setof{\alpha^{i}_\gamma}{\gamma<\gamma^{i}}$, for some $\gamma^{i}<\kappa$, 
be the increasing enumeration of ${\mathcal t}^{i}$, let
$\theta^{i}= \ssup(\setof{\gamma}{\alpha_\gamma\in {\mathcal t^{i}}})$, let
$T^{i}=\setof{\gamma<i}{\alpha_\gamma\in {\mathcal t}^{i}}$, and let
$q^i \in \hbox{Fn}(T^i,2,\kappa)$ be defined by 
$q^i(\gamma) = p^i(\alpha_\gamma)$.

So each $\gamma < \kappa^+$, each $\alpha^i_\gamma < \lambda$, each $\theta^i<\kappa^+$, each $T^i \in [i]^{<\kappa}$
and each $q^i \in \hbox{Fn}(T^i,2,\kappa)\subseteq \hbox{Fn}(i,2,\kappa)$.
 
For $i\in [\kappa,\kappa^+]$ let $H_i = [i]^{<\kappa}\times \hbox{Fn}(i,2,\kappa)\times \kappa$, and 
write $H$ for $H_{\kappa^+}$. Let $h^*$ be an injection from $H$ into $\kappa^+$. 
Define $k:[\kappa,\kappa^+)\longrightarrow \kappa^+$ by setting $k(i)$ to be the least 
$i^*<\kappa^+$ such that $H_i \subseteq h^{*-1}``i^*$. 

Let $\tilde{C} = \setof{j<\kappa^+}{\forall i < j \sss (\theta^i, k(i) < j)}$. As the intersection of the sets of 
closure points of the two given functions, $\tilde{C}$ is a club subset of $\kappa^+$. 

Let $h(i) = h^*(T^i,q^i,\otp({\mathcal t^i}))$ for
$i \in \tilde{C}\cap S^{\kappa^+}_\kappa$, and $h(i)=0$ otherwise.

We have that $h^{*-1}(h(i)) \in H_i$ for all $i\in [\kappa,\kappa^+)$.
If $i\in \widetilde{C}\cap S^{\kappa^+}_\kappa$, since $\card{h^{*-1}(h(i))} < \kappa$,
there is some $i'<i$ such that $h^{*-1}(h(i))\in H_{i'}$, and hence 
there is some $\widetilde{i}<i$ such that $h(i)<k(\widetilde{i})$. 

Hence, as $i$ is a closure point of $k$, we have $h(i)<i$ for all nonzero $i<\kappa^+$. 

Now suppose that $i, j \in \widetilde{C} \cap S^{\kappa^+}_\kappa$, $i < j$ and $h(i) = h(j)$. 
So we have  $T^{i}=T^j$, $q^i = q^j$ and $\otp({\mathcal t}^{i}) = \otp({\mathcal t}^{j})$. 

Set ${\mathcal t} = {\mathcal t}^{i}\cap {\mathcal t}^{j}$.

\begin{lemma}\label{Cohen-intersection} 
\begin{itemize}
\item[(a)] ${({\mathcal t}^{j}\setminus {\mathcal t}^{i})\cap\setof{\alpha_\gamma}{i\le \gamma<j} = \emptyset}$, and 
\item[(b)] ${\mathcal t}\subseteq \setof{\alpha_\gamma}{\gamma<i}$.
\end{itemize}
\end{lemma}

\begin{proof}[\ref{Cohen-intersection}] Suppose $\alpha_\gamma\in {\mathcal t}^{j}$. If $\gamma<j$ then $\gamma\in T^{j}$. 
But $T^{j}=T^{i}$, so $\gamma\in T^{i}$. Hence $\gamma<i$ and $\alpha_\gamma\in {\mathcal t}^{i}$, 
proving (a). If $\alpha_\gamma\in {\mathcal t}^{i}$ then $\gamma<\theta^{i}<j$. (For the definition 
of $\theta^i$ immediately gives that $\gamma<\theta^i$; and since $i$, $j\in \widetilde{C}$ and $i<j$
one has that $\theta^i<j$.) Thus if $\alpha_\gamma\in {\mathcal t}^{i}\cap {\mathcal t}^j$ we have $\gamma<i$ by (a). So (b) holds.
\end{proof}

By Lemma (\ref {Cohen-intersection}(b)) we have for $\alpha_\gamma\in {\mathcal t}$ that
$p^i(\alpha_\gamma) = q^i(\gamma) = q^j(\gamma) = p^j(\alpha_\gamma)$, and hence
$p^i$ and $p^j$ agree on the intersection of their domains, and thus are compatible conditions.
\end{proof}

\section{Preserving diamond under forcing}\label{preservation}

In this section we give an account of how versions of the diamond principle at a regular cardinal $\chi$
are preserved by certain forcing posets. We are most interested in the situation where 
the forcing poset is an iteration $\Pbb_\chi$ of length $\chi$, and we can find a diamond
sequence $\tupof{S_\alpha}{ \alpha < \chi } \in V[G_\chi]$ such that
$S_\alpha \in V[G_\alpha]$ for all $\alpha$.  We will use the results of this section in \S\ref{maini} and \S\ref{small}.

The following result is well-known. For the reader's
convenience we will sketch a proof. 
Similar arguments appear, for example,
in the proof that $\clubsuit$ does not imply $\diamondsuit_{\omega_1}$ in 
\cite{Shelah-proper-forcing-2nd-ed}, \S{}I.7 and in 
a preservation theorem for $\diamondsuit^+_{\omega_1}$
in \cite{Cummings-Foreman-Magidor}, \S{}12.

\begin{theorem} \label{general-diamond-preservation}
Let $\chi$ be a regular uncountable cardinal, let $\Pbb$ be a forcing poset of
cardinality $\chi$ which preserves the regularity of $\chi$  and let $G$ be $\Pbb$-generic over $V$.
\begin{enumerate}
\item If $\diamondsuit_\chi$ holds in $V$ and $\Pbb$ preserves stationary subsets of $\chi$
then $\diamondsuit_\chi$ holds in $V[G]$. \label{g-d-p-one} 
\item If $\diamondsuit^*_\chi$ (resp $\diamondsuit^+_\chi$) holds in $V$ then 
$\diamondsuit^*_\chi$ (resp $\diamondsuit^+_\chi$) holds in $V[G]$.
\end{enumerate}
\end{theorem} 

\begin{proof}[\ref{general-diamond-preservation}]
(1). Enumerate the elements of $\Pbb$ as $p_i$ for $i < \chi$,
fix in $V$ a diamond sequence $\tupof{ T_\alpha \subseteq \alpha \times \alpha }{ \alpha < \chi }$ 
which guesses subsets
of $\chi \times \chi$, and in $V[G]$ define $S_\alpha = \setof{ \eta < \alpha}{ \exists i < \alpha \;
\mbox{($p_i \in G$ and $(i, \eta) \in T_\alpha$)} }$. Now if $S = \name{S}^G \subseteq \chi$ 
we let $T = \setof{ (i, \eta) }{ p_i \forces \hat{\eta} \in \name{S} }$, and observe that 
$\setof{ \alpha }{ T \cap \alpha \times \alpha = T_\alpha }$ is stationary in $V$ (hence also in $V[G]$) and
$E = \setof{ \alpha }{ \forall \eta \in S \cap \alpha \; \exists i < \alpha \; 
p_i \in G\hbox{ and }p_i \forces \hat{\eta} \in \dot S }$
is club in $V[G]$; at any point $\alpha\in E$ such that $T\cap \alpha\times\alpha = T_\alpha$ 
we have that $S \cap \alpha = S_\alpha$.  Hence $\tupof{S_\alpha}{\alpha<\chi}$ is a $\diamondsuit_\chi$-sequence in $V[G]$.

(2). The proofs for $\diamondsuit^*_\chi$ and $\diamondsuit^+_\chi$ are very similar. The key difference is that
the hypothesis that the regularity of $\chi$ is preserved suffices to see that 
$E$ is club in $V[G]$, and the stronger form of guessing occurring in the  $\diamondsuit^*_\chi$ and $\diamondsuit^+_\chi$
gives a club from $V$ with which to intersect it.
\end{proof} 

\begin{observation} \label{general-diamond-observations} 
For future reference we note that:
\begin{enumerate}
\item An entirely similar argument shows that if $B$ is a stationary subset of $\chi$ and $\Pbb$ preserves stationary
subsets of $B$, then $\Pbb$ preserves $\diamondsuit_\chi(B)$.
If $\Pbb$ preserves the stationarity of $B$ then $\Pbb$ preserves $\diamondsuit^*_\chi(B)$ and $\diamondsuit^+_\chi(B)$. 

\item  In both parts of Theorem (\ref{general-diamond-preservation}), the value of the diamond
sequence in $V[G]$ at $\alpha$ is computed in a uniform way from $G \restriction \alpha$ where $G \restriction \alpha = G \cap \setof{ p_i }{ i < \alpha }$.
In fact the restriction of the diamond sequence to $\alpha + 1$ (\emph{i.e.}, $\tupof{S_i}{i\le\alpha}$) can be computed in a uniform way from
$G \restriction \alpha$.

\item In the proof of part (\ref{g-d-p-one}) of Theorem (\ref{general-diamond-preservation}), 
let $A$ be the stationary set in $V[G]$ given by $A = \setof{ \alpha }{ S \cap \alpha = S_\alpha }$.
For every $\alpha \in E$, $S \cap \alpha \in V[G \restriction \alpha]$ and so easily $E \cap \alpha \in V[G \restriction \alpha]$.
Also, by the preceding remarks, $A \cap \alpha \in V[G \restriction \alpha]$. If we now let
$H = A \cap E$ then $H$ is a stationary set on which $S$ is guessed correctly and additionally
$H \cap \alpha \in V[G \restriction \alpha]$ for all $\alpha$ in the closure of $H$.
\end{enumerate}
\end{observation}

\begin{proposition} \label{dnotes}
Let $\kappa$, $\chi$ be  regular cardinals with $\kappa<\chi$. Let $\Pbb=\tup{\tupof{\Pbb_\xi}{\xi\le\chi}, \tupof{\name{\Qbb}_\xi}{\xi<\chi}}$ be
a forcing iteration with $\lower0.05em\hbox{$<$}\kappa$-supports with $\card{\Pbb}_\xi<\chi$ for each
$\xi<\chi$.  Let $G$ be $\Pbb$-generic over $V$.

Suppose  $A$ is a stationary subset of $S^\chi_{\ge\kappa}$ and $\diamondsuit_\chi(A)$ holds.
Then there is  a sequence $\tupof{ \name{S}_\xi }{ \xi \in A }$ such that
\begin{enumerate}
\item  For all $\xi \in A$, $\name{S}_\xi$ is a $\Pbb_\xi$-name for a subset of $\xi$. 
\item  If in $V[G]$ we define $S_\xi$ as the interpretation of $\name{S}_\xi$ by
$G_\xi$ then the sequence $\tupof{ S_\xi }{ \xi \in A }$ has the following
strengthened form of the $\diamondsuit_\chi(A)$-property: for every $S \subseteq \chi$ 
there is a stationary set $H \subseteq A$
such that $S \cap \xi = S_\xi$ for all $\xi \in H$, and in addition $H \cap \xi \in V[G_\xi]$ for all $\xi$ in the closure of $H$. 
\end{enumerate}

\end{proposition}

\begin{proof}[\ref{dnotes}]
We start by observing that by arguments as in  Baumgartner's survey paper on iterated forcing \cite[\S{2}]{Baumgartner}
the poset $\Pbb$ is $\chi$-cc. We sketch the proof briefly: given a sequence $\tupof{ p_i }{ i < \chi }$ of
elements of $\Pbb$ we apply Fodor's theorem and the bound on the size of initial segments to find a stationary
set $U \subseteq S^\chi_\kappa$ and an ordinal $\eta < \chi$ such that $\supp(p_i) \cap i \subseteq \eta$ and
$p_i \restriction \eta$ is constant for $i \in U$, then find $i, j \in U$ such that
$\supp(p_i) \subseteq j$ and argue that $p_i$ is compatible with $p_j$. 

Noting that $\vert \Pbb \vert = \chi$, we enumerate  $\Pbb$ as $\tupof{p_\xi}{\xi<\chi}$ 
and identify each $\Pbb_\xi$ with the set of conditions $p\in\Pbb$ such that $\supp(p)\subseteq\xi$ for $\xi<\chi$.

Let  $F$ be $\setof{\xi<\chi}{\forall \varepsilon < \xi \sss \sss (\supp(p_\varepsilon) <\xi \sss\sss\&\sss\sss
\Pbb_\xi=\setof{p_\varepsilon}{\varepsilon<\xi})}$. 
Since $\Pbb$ is an iteration with $\lk$-support, $F$ is club relative to $S^\chi_{\ge\kappa}$.
For let us define $h:\chi\longrightarrow \chi$ by setting $h(\varepsilon)=$ the least $\xi$ such that 
$\supp(p_\varepsilon)\subseteq \xi$ and $\Pbb_\varepsilon \subseteq \setof{p_\zeta}{\zeta < \xi}$, 
and define $C$ to be the set of closure points of $h$. 

For each $\xi\in C\cap S^{\chi}_{\ge\kappa}$ and $\varepsilon<\xi$ we have that 
$\supp(p_\varepsilon) \subseteq h(\varepsilon) < \xi$ and, since $\cf(\xi)\ge \kappa$ and 
the size of the support of each condition is less than $\kappa$, there is some $\gamma\in C\cap \xi$ such 
that $h(\varepsilon)<\gamma$ and so $\Pbb_\varepsilon \subseteq \setof{p_\zeta}{\zeta<\gamma}$.
Hence $\xi\in F$. Thus $A \setminus F$ is non-stationary. 

Since $\Pbb$ is $\chi$-cc it preserves stationary subsets of $\chi$, and we may therefore
appeal to Theorem (\ref{general-diamond-preservation}) and Observation (\ref{general-diamond-observations}) to obtain a sequence
$\tupof{ \name{S}_\xi}{ \xi \in A \cap F }$ such that
\begin{enumerate}
\item For all $\xi \in A \cap F$, $\name{S}_\xi$ is a $\Pbb_\xi$-name for a subset of $\xi$.
\item If in $V[G]$ we define $S_\xi$ as the realisation of $\name{S}_\xi$ by $G_\xi$,
then for every $S \subseteq \kappa$ there is a stationary set $H \subseteq A \cap F$
such that $S \cap \xi = S_\xi$ for all $\xi \in H$ and additionally
$H \cap \xi \in V[G_\xi]$ for all $\xi$ in the closure of $H$. 
\end{enumerate}  
To finish the proof we  fill in the missing values by defining $\name{S}_\xi = \hat{\emptyset}$ for $\xi \in A \setminus F$.
\end{proof}

\section{Radin forcing}\label{Radinmat} 

As we commented in the preamble, our proof involves Radin
forcing. There are several accounts of this forcing in the literature
(see \cite{Radin}, \cite{Mitchell-club-ulf}, \cite{Cummings-Woodin}, \cite{Gitik} and \cite{Jech-03}), each subtly different from the
others, and it turns out that it does matter which version of the
forcing we use.  Although the proofs of the various properties of the
forcing are easiest, or at least slickest, for the versions given in
Cummings-Woodin (\cite{Cummings-Woodin}) and in Gitik (\cite{Gitik}), using either of these here
creates technical difficulties.  Consequently, we shall define and use a version
of Radin forcing which, except for one small alteration to which we 
draw attention below, closely follows that of Mitchell in \cite{Mitchell-club-ulf}.

First of all, we give, by induction on $\kappa \in \Card$, 
the definition of the set of {\em ultrafilter sequences at $\kappa$.}

\begin{definition}\label{ultrafilterseq}
Let $\kappa$ be a cardinal and let 
${\mathcal U}_\kappa$ be the collection of ultrafilter sequences at
cardinals smaller than $\kappa$. A sequence  
$u=\tupof{u_{\tau}}{\tau<\lh(u)}$ is an
\emph{ultrafilter sequence at} $\kappa$ if $\lh(u)$ is a non-zero
ordinal, $u_{0} =\kappa$, so that, using the notation of the introduction, $\kappa_u=\kappa$, 
and, for $\tau\in (0,\lh(u))$, each $u_{\tau}$ is a 
$\kappa$-complete {ultrafilter} on $V_\kappa$
with ${\mathcal U}_\kappa\in u_{\tau}$ and satisfies the following \emph{normality}
and \emph{coherence} conditions (with respect to
$u$):

\emph{(normality)} if $f: {\mathcal U}_\kappa\longrightarrow V_\kappa$ and 
$\setof{w\in {\mathcal U}_\kappa}{f(w)\in V_{\kappa_w}}\in
u_{\tau}$ then there is some $x\in V_\kappa$ such that 
$\setof{w\in {\mathcal U}_\kappa}{f(w)=x}\in u_\tau$.

\emph{(coherence 1)}  if $f: {\mathcal U}_\kappa\longrightarrow \kappa$ and
$\setof{w\in {\mathcal U}_\kappa}{f(w) < \lh(w)}\in u_\tau$ there
is $\sigma<\tau$ such that 
$u_\sigma = \setof{X \subseteq V_\kappa} { \setof{w\in {\mathcal U}_\kappa}{X \cap V_{\kappa_w} \in w_{f(w)}} \in u_\tau }$.

\emph{(coherence 2)} if $\sigma < \tau$ and $X \in u_\sigma$ then 
$\setof{w\in {\mathcal U}_\kappa}{\exists \bar\sigma<\lh(w)\sss X \cap V_{\kappa_w} 
\in w_{\obar{\scriptstyle \sigma}}} \in u_\tau$.

\end{definition} 

\begin{definition}\label{ultrafilterclass}
The class ${\mathcal U}$ is the
class of all ultrafilter sequences on any cardinal: ${\mathcal U}=\bigcup_{\kappa\in\Card}
{\mathcal U}_\kappa$.
\end{definition}

\begin{definition}\label{ultrafilterseqsatkappa} For clarity we re-iterate the following
special cases of  notation of Definition (\ref{ultrafilterseq}): 
${\mathcal U}_{\kappa^+}$ is the set of ultrafilter sequences at cardinals less than or equal to $\kappa$ and
${\mathcal U}_{\kappa^+} \setminus {\mathcal U}_{\kappa}$ is the set of ultrafilter sequences at $\kappa$.
\end{definition}

\begin{observation} As Mitchell comments in \cite{Mitchell-club-ulf}, by the coherence properties,
if $u\in {\mathcal U}$ and $\lh(u)\le \kappa_u$ then for $\tau\in(0,\lh(u))$ one has that 
$u_\tau$ concentrates on $\setof{w\in {\mathcal U}_{\kappa_u}}{ \lh(w)=\tau}$ 
({\it i.e.\/}, $\setof{w\in {\mathcal U}_{\kappa_u}}{ \lh(w)=\tau}\in u_{\tau}$).
\end{observation}

\begin{observation} We need some large cardinal assumption in order to construct non-trivial ultrafilter
sequences. For the purposes of this paper we can use a construction due to Radin \cite{Radin}.
Let $j: V \rightarrow M$ witness that $\kappa$ is $2^\kappa$-supercompact. Derive a sequence
$u$ by setting $u_0 = \kappa$ and then $u_\alpha = \setof{ X \subseteq V_\kappa }{ u \restriction \alpha \in j(X) }$
for $\alpha > 0$. One can verify that for every $\alpha < (2^\kappa)^+$ we have
$u \restriction \alpha \in {\mathcal U}_\kappa$. In fact we will only need ultrafilter sequences
of length less than $\kappa$ in the sequel.
\end{observation}

\begin{definition}\label{filterF} If $u\in {\mathcal U}$ then 
${\mathcal F}(u)=\bigcap_{\tau\in(0,\lh(u))} u_\tau$.
\end{definition}

\begin{definition}\label{pair} A \emph{pair} is some $(u,A)$ with $u\in {\mathcal U}$
and $A\in {\mathcal F}(u)$ if $\lh(u)>1$ and $A=\emptyset$ if $\lh(u)=1$.
\end{definition}

(Note that our notation thus far is marginally, but inessentially, 
different from that of \cite{Mitchell-club-ulf}).

\begin{definition}\label{Radinfor} Let $w\in {\mathcal U}$. $\Rad{w}$, 
\emph{Radin forcing at} $w$, has as conditions sequences of pairs 
$\tup{(u_0,B_0),\dots,(u_n,B_n)}$ such that $u_n=w$, and, writing 
$\kappa_i$ for $\kappa_{u_i}$, $(u_i,B_i)\in {\mathcal U}_{\kappa_{{i+1}}}$
and $B_{i+1}\cap V_{\kappa_{i}^+} = \emptyset$  for $i<n$.

Let $p=\tup{(u_0,B_0),\dots,(u_n,B_n)}$
and $q=\tup{(v_0,D_0),\dots,(v_m,D_m)} \in \Rad{w}$. Then $q\le p$
($q$ refines $p$) if 
$m \ge n$ and
\begin{enumerate}[label=(\roman*)]
\item  For every $i \le n$ there is $j \le m$ such that $u_i = v_j$ and $D_j \subseteq B_i$.
\item  For every $j \le m$, either $v_j = u_i$ for some $i$ or for the least $i$ such that
$\kappa_{v_j} < \kappa_{u_i}$, $v_j \in B_i$ and $D_j \subseteq B_i$. 
\end{enumerate}

We also define $q\le^* p$ if
$\tupof{u_i}{i\le n}=\tupof{v_i}{i\le n}$ and for each $i\le n$ 
we have $D_i\subseteq B_i$. We say that 
$q$ is a \emph{(Radin-)direct extension} of $p$.
Thus $q\le ^*p$ implies $q\le p$. 

\end{definition}

\begin{definition}\label{appears} Let $p=\tup{(u_0,B_0),\dots,(u_n,B_n)}$ be a condition in $\Rad{u_n}$.
A pair $(u,B)$ {\em appears in $p$}
if there is some $i\le n$ such that
$(u,B)=(u_i,B_i)$.  Similarly an ultrafilter sequence $u$ {\em appears in $p$}
if there is some $i\le n$ such that
$u = u_i$.
\end{definition}
Thus clause (ii) in the definition of $\le$ in Definition (\ref{Radinfor}) reads: 
if $v_j$ does not appear in $p$ and $i$ is minimal such that $\kappa_{v_j}<\kappa_{u_i}$
then $v_j \in B_i$ and $D_j \subseteq B_i$. 

\begin{observation} \cite{Mitchell-club-ulf} omits the last clause 
in the definition of what it is to be a
condition. However our conditions form a dense subset of the
conditions as defined in \cite{Mitchell-club-ulf} and the facts that
we quote from \cite{Mitchell-club-ulf} are also true of our forcing.
This minor change is advantageous for technical reasons in order to
make the proof below run smoothly.
\end{observation}

To help orient the reader, we record a few remarks (without proof)
about the nature of the generic object for the forcing poset
$\Rad{w}$. This generic object is best viewed as a sequence 
$\tupof{ u_i }{ i < \delta }$ where $u_i \in {\mathcal U}_{\kappa_w}$ and
$\tupof{ \kappa_{u_i} }{ i < \delta }$ is increasing and continuous.
When $\lh(w) > 1$ the sequence $\tupof{ \kappa_{u_i} }{ i < \delta }$
is cofinal in $\kappa_w$, and we will view this sequence as
enumerating a club set in $\kappa_w$ which we call the 
{\em Radin-generic club set.}  When $i = 0$ or $i$ is a successor ordinal
then $\lh(u_i) = 1$, otherwise $\lh(u_i) > 1$.  The translation
between the generic sequence and the generic filter is given by the
following.

\begin{definition}\label{generic-sequence-from-generic-filter}
Let $G$ be an $\Rad{w}$-generic filter over $V$. The sequence ${\tupof{ u_i }{ i < \delta }}$
is the {\em corresponding generic sequence} if it enumerates 
${\setof{u\in {\mathcal U}_{\kappa_w} }{\exists p\in G\sss\sss u\mbox{ appears in }p}}$, {\em i.e.},
the set of $u \in {\mathcal U}_{\kappa_w}$ which appear in some condition in $G$.
\end{definition} 

\begin{lemma} \label{filterandsequence1}
Let $G$ be an $\Rad{w}$-generic filter over $V$ and $\tupof{ u_i }{ i < \delta }$
the corresponding generic sequence. Then 
\begin{align*}
G= \setof{p\in \Rad{w}}{\forall u \in {\mathcal U}_{\kappa_w} \sss(u & \mbox{ appears in }  p 
\longrightarrow \sss \exists i<\delta \sss u=u_i) \sss\sss\&\sss\sss \\
& \forall i< \delta \sss\exists q\le p \sss\sss( u_i\mbox{ appears in }q) },
\end{align*}
{\em i.e.} $G$ is the set of conditions $p \in \Rad{w}$ 
such that every $u \in {\mathcal U}_{\kappa_w}$ which appears
in $p$ is among the sequences $u_i$, and every sequence $u_i$ appears in some extension of $p$.
\end{lemma} 
\vskip6pt

\begin{definition}\label{lowerparts}
A {\em lower part} for $\Rad{w}$ is a condition in $\Rad{u}$ for some $u\in {\mathcal U}_{\kappa_w}$.
\end{definition}
We note that any condition in $\Rad{w}$ has the form $y \concat (w, B)$ where $y$ is empty or  
a lower part for
$\Rad{w}$, and $B \in {\mathcal F}(w)$. In the case when $y$ is non-empty we will say that
{\em $y$ is the lower part of $y \concat (w, B)$}.  It is easy to see that any two conditions
with the same lower part are compatible, so that $\Rad{w}$ is the union of $\kappa_w$-many
$\kappa_w$-complete filters and in particular it enjoys the $\kappa_w^+$-chain condition.

A key point is that below a condition of the form  $y \concat (w, B)$ with $y$ non-empty, everything
up to the last ultrafilter sequence appearing in $y$ is controlled by pairs in $y$. More formally:

\begin{definition} \label{lowerpartdefs} 
Let $y = \tupof{(u_i, B_i)}{i \le n}$ be a lower part for $\Rad{w}$.  Then
$\kappa_y = \kappa_{u_n}$ and $\Rad{y} = \setof{q \in \Rad{u_n}}{q \le y}$.
\end{definition}

\begin{lemma} \label{lowerpartlemma}  Let $p = y \concat (w, B)$ be a condition in $\Rad{w}$ with
$y$ non-empty. Then
the subforcing $\setof{q  \in \Rad{w}}{q \le p}$ is isomorphic to the product
$\Rad{y} \times \setof{q \in \Rad{w}}{q \le (w, B \setminus {\mathcal U}_{\kappa_y^+})}$. 
\end{lemma}

It follows from these considerations that if 
$\tupof{ u_i }{ i <  \delta }$ is a generic sequence for $\Rad{w}$ and 
$\zeta < \delta$, then $\tupof{ u_i }{ i < \zeta }$ is a generic sequence for
$\Rad{u_\zeta}$. More generally, if there is a condition in the
generic filter with lower part $y$ then the generic sequence induces
an $\Rad{y}$-generic object in the natural way.

\begin{definition} \label{conformance} Let $y$ be a lower part for $\Rad{w}$ and let 
$G$ be $\Rad{w}$-generic over $V$. We say that 
{\em $y$ conforms with $G$} if and only if $y$ is the lower part of some condition in $G$.
\end{definition} 
The following lemma is a ``local'' version of Lemma (\ref{filterandsequence1}).
\begin{lemma} \label{conforming-and-sequence}
Let  $G$ be an $\Rad{w}$-generic filter over $V$ and let  $\tupof{ u_i }{ i < \delta }$
be the corresponding generic sequence.
Let $y = \tupof{(v_k, B_k)}{k \le n}$ be a lower part. Then the following are equivalent:
\begin{enumerate}
\item $y$ conforms with $G$.
\item For every $k\le n$ the sequence $v_k$ appears among the sequences $u_i$, and for every 
$i$ with $\kappa_{u_i} \le \kappa_{v_n}$ the ultrafilter sequence
$u_i$ appears in some extension of $y$ in $\Rad{y}$.
\end{enumerate}
Moreover, if $i<\delta$ then $u_i$ appears in some lower part $y'$ 
which conforms with $G$ and with $\kappa_{y'} = \max(\set{\kappa_y,\kappa_{u_i}})$,
and such that if $y\concat(w,B)\in G$ there is 
some $B'\in {\mathcal F}(w)$ such that $y'\concat (w,B') \le y \concat (w,B)$. 
\end{lemma}

The following result by Radin captures a key property of $\Rad{w}$ 

\begin{theorem}[Radin] \label{Prikry-prop-for-Radin} (The Prikry property for Radin forcing.) 
Let $p\in \Rad{w}$ and let $\phi$ be a sentence in the forcing language. 
Then either there is some $p' \le^* p$ such that $p'\forces_{\Rad{w}} \phi$ or there is some $p'\le^* p$ 
such that $p'\forces_{\Rad{w}} \neg \phi$.
\end{theorem}

Combining Theorem (\ref{Prikry-prop-for-Radin}) and  Lemma (\ref{lowerpartlemma}), we obtain a
lemma (due to Radin) which will be very important in the proof of $\kappa^+$-cc for the main iteration.
Recall that by our conventions a name for a truth value  is just a name for an
ordinal which is either $0$ (false) or $1$ (true).

\begin{lemma}[Radin] \label{truthvalue-reduction}
Let $y \concat (w, B)$ be a condition in $\Rad{w}$ with $y$ non-empty, and let
$\dot b$ be an $\Rad{w}$-name for a truth value. Then there exist $C \subseteq B$
with $C \in {\mathcal F}(w)$ and an $\Rad{y}$ name for a truth value $\dot c$ such that
whenever $G$ is $\Rad{w}$-generic with $y \concat (w, C) \in G$ and $G'$ is the induced generic
object for $\Rad{y}$, then $(\dot b)^G = (\dot c)^{G'}$.
\end{lemma}

\begin{proof}[\ref{truthvalue-reduction}]
Using the factorisation given by Lemma (\ref{lowerpartlemma}) and 
the Product Lemma,  we may view
$\dot b$ as an ${\mathbb R}^*$-name for an $\Rad{y}$-name for a truth value
where ${\mathbb R}^* = \setof{q \in \Rad{w}}{q \le (w, B \setminus {\mathcal U}_{\kappa_y^+})}$. 
Since $\vert \Rad{y} \vert < \kappa$ and a name for a truth value amounts to an antichain
in $\Rad{y}$, there are fewer than $\kappa$ many $\Rad{y}$-names for  truth values.
Since ${\mathcal F}(w)$ is $\kappa$-complete we may appeal to the Prikry property
for ${\mathbb R}^*$, and 
shrink $(w, B \setminus {\mathcal U}_{\kappa_y^+})$  to $(w, C)$ in order to
decide which $\Rad{y}$-name $\dot c$ is in question.
\end{proof} 

In the situation of Lemma (\ref{truthvalue-reduction}),
we will sometimes say that the condition $y \concat (w, C)$ {\em reduces}
the $\Rad{w}$-name $\dot b$ to the $\Rad{y}$-name $\dot c$.
Similar arguments (which we omit) about reducing names for sets of ordinals give 
another important result.

\begin{theorem}[Radin]\label{Radinsets} Let $w \in {\mathcal U}$, let
$G$ be $\Rad{w}$-generic over $V$ and let $\tupof{ u_j }{ j < \delta }$ be the
corresponding generic sequence.
For every $\alpha < \kappa_w$, if $i$ is largest such that
$\kappa_{u_i} \le \alpha$,  then every subset
of $\alpha$ lying in $V[G]$ lies in the $\Rad{u_i}$-generic extension 
given by $\tupof{ u_j }{ j < i }$.
\end{theorem}

It follows readily that forcing with $\Rad{w}$ preserves all cardinals. However
cofinalities may change, and the general situation is a little complicated. The main point
for us is given by the following result.

\begin{theorem}[Radin]\label{Radinlth} Let $w \in {\mathcal U}$, let
$G$ be $\Rad{w}$-generic over $V$ and assume that $\lh(w)$ is a regular
cardinal $\lambda$ with $\lambda < \kappa_w$. 
Let $G$ be $\Rad{w}$-generic. If $\lambda$ is not
a limit point of the Radin-generic club, then
$V[G] \models \lambda \hbox{ is regular and }\cf(\kappa) = \lambda$.
\end{theorem}

We note that the condition on $\lambda$ in this theorem is easy to arrange by
working below a suitable condition in $\Rad{w}$. For example we may arrange that
the least point of the generic club is greater than $\lambda$.  

We will also require a characterisation of Radin-genericity which is due to Mitchell. 

\begin{theorem}[Mitchell]\label{Mitchell-characterization}
Let $\tupof{ u_i }{ i < \delta }$ be a sequence of ultrafilter 
sequences in some outer model of $V$.
Then the following are equivalent:
\begin{enumerate}
\item  The sequence $\tupof{ u_i }{ i < \delta }$ is $\Rad{w}$-generic. 
\item  For every $j < \delta$ the sequence $\tupof{ u_i }{ i < j }$ is $\Rad{u_j}$-generic,
and if $\lh(w) > 1$ then ${\mathcal F}(w) = \setof{ X\in V }{ \exists j < \delta \; \forall i \;
j < i < \delta \implies u_i \in X }$ -- \emph{i.e.}, if $\lh(w)>1$ then ${\mathcal F}(w)$ is the tail filter generated on the 
$V$-powerset of $V_\kappa$ by the generic sequence.
\end{enumerate}  

\end{theorem}

\begin{definition}
Sequences which satisfy the Mitchell criterion from Theorem 
(\ref{Mitchell-characterization}) are sometimes called {\em geometric sequences.}
\end{definition}

The following considerations will play a central
role in the proof of Theorem (\ref{universalg}).

\begin{observation}\label{Radin-generics-go-down}
Suppose that $V\subseteq V'$ with $V_\kappa^V=V_\kappa^{V'}$
({\em e.g.}, $V'$ might be a generic extension of $V$ by $\lk$-closed forcing).
Suppose that in $V'$ there is
an ultrafilter sequence $w'$ such that $\lh(w) = \lh(w')$ and 
$w(\alpha) = w'(\alpha) \cap V$ for $0 < \alpha < \lh(w)$. 
Let $\tupof{ u_i }{ i < \delta }$ be an $\Rad{w'}$-generic sequence over $V'$. 
\begin{itemize}
\item By the Mitchell criterion for genericity and the observation that 
${\mathcal F}(w) \subseteq {\mathcal F}(w')$, the sequence
$\tupof{ u_i }{ i < \delta }$  is also an $\Rad{w}$-generic sequence over $V$. 
\end{itemize}

\end{observation}

\vskip6pt

\begin{lemma} \label{filterandsequence2}
Let  $G$ be an $\Rad{w}$-generic filter over $V$ and let  $\tupof{ u_i }{ i < \delta }$
be the corresponding generic sequence.
Let $y = \tupof{(v_k, B_k)}{k \le n}$ be a lower part. Then the following are equivalent:
\begin{enumerate}
\item $y$ conforms with $G$.
\item For every $k$ the sequence $v_k$ appears among the sequences $u_i$, and for every 
$i$ with $\kappa_{u_i} \le \kappa_{v_n}$ either $u_i$ appears among the sequences $v_k$ or 
$u_i \in B_k$ for the least $k$ such that $\kappa_{u_i} < \kappa_{v_k}$. 
\end{enumerate}
\end{lemma} 

\begin{proof}[\ref{filterandsequence2}]
It is immediate that the first condition implies the second.  For the
converse direction, by Lemma (\ref{conforming-and-sequence}), we need to verify
that every sequence $u_i$ with $\kappa_{u_i} \le \kappa_{v_n}$ appears
in some element of $\Rad{y}$. This is immediate when $u_i$ appears
among the $v_k$, so we assume that it does not and let $k$ be least
with $\kappa_{u_i} < \kappa_{v_k}$.  If $\lh(u_i) = 1$ then we may
extend $y$ by shrinking $B_k$ to $B_k \setminus {\mathcal U}_{\kappa^+_{u_i}}$
and inserting the pair $(u_i, \emptyset)$.  If $\lh(u_i) > 1$ then $i$
is limit and by hypothesis $u_j \in B_k$ for all large $j < i$, so
that (by the geometric condition from part (2) of Theorem (\ref{Mitchell-characterization})) $B_k \cap V_{\kappa_{u_i}} \in
{\mathcal F}(u_i)$.  In this case we may extend $y$ by shrinking $B_k$
to $B_k \setminus {\mathcal U}_{\kappa^+_{u_i}}$ and inserting the pair $(u_i,B_k \cap V_{\kappa_{u_i}})$.
\end{proof} 

A very similar argument gives:
\begin{lemma} \label{filterandsequence3} Let  $G$ be an $\Rad{w}$-generic filter over $V$ and let 
$\tupof{ u_i }{ i < \delta }$ be the corresponding generic sequence.  Let 
$y \concat (w, B)$ be a condition. Then $y \concat (w, B) \in G$ if and
only if $y$ conforms with $G$ and $u_i \in B$ for all $i$ such that $\kappa_{u_i} > \kappa_y$.
\end{lemma}

For use later we make a technical definition which is  motivated by the Mitchell
criterion for genericity, and will be used in the definition of the main iteration.

\begin{definition}\label{defn-harmonious}

For $y=\tupof{(u_i,B_i)}{i<n}$ a lower part for $\Rad{w}$, 
$A=\tupof{A_\rho}{\rho<\rho^*}$ a sequence of subsets of $V_{\kappa_w}$ and
$\eta<\kappa_w$ we say $y$ is \emph{harmonious with} $A$ \emph{past} or \emph{above} $\eta$,  if
for each $i$ with $i < n$ one of the following conditions holds.

\begin{itemize}

\item $\kappa_{u_i} < \eta$.

\item $\kappa_{u_i} = \eta$, $u_i = \tup{ \eta }$ and $B_i = \emptyset$.

\item $\kappa_{u_i} > \eta$, and $\set{ u_i } \cup B_i \subseteq \bigcup_{\rho < \rho^*} A_\rho \setminus 
{\mathcal U}_{\eta^+}$. 

\end{itemize}

\end{definition}

\begin{observation}
We record some remarks about the preceding definition.

\begin{itemize}

\item
If a lower part $y$ is harmonious with $A$ past $\eta$ it divides
rather strictly into a part below $\eta$ and a part above: there is no
$(u_i,B_i)$ appearing in $y$ such that $\eta\le \kappa_{u_i}$ while
$\kappa_v \le \eta$ for some $v\in B_i$.

\item 
Definition (\ref{defn-harmonious}) depends only on the set
$\bigcup_{\rho < \rho^*} A_\rho$ rather than the sequence $A$
itself. It is phrased in this way to avoid encumbering later
definitions with union signs.

\end{itemize}

\end{observation}

The following lemma shows how to thin a lower part conforming with a Radin generic sequence to 
one which still conforms and which is also harmonious past some $\eta$ with 
a sequence of sets whose union contains the interval of the generic sequence
consisting of those measure sequences with critical point at least $\eta$ and
which it can `see'. This, also, is useful for the proof of
Theorem (\ref{universalg}).

\begin{lemma} \label{new-harmony-lemma}  
Let  $G$ be an $\Rad{w}$-generic filter over $V$ and let  $\tupof{ u_i }{ i < \delta }$
be the corresponding generic sequence. Let $A=\tupof{A_\rho}{\rho<\rho^*}$ be such that each
$A_\rho$ is a set of measure sequences with a common critical point $\kappa_\rho < \kappa$, 
let $a=\setof{\kappa_\rho}{\rho<\rho^*}$
and let $D=\bigcup_{\rho<\rho^*}{A_\rho}$.

Suppose that $\eta$ is a successor point in the generic club set $C$ such that 
$u_i\in D$ for all $i$ such that $\eta\le \kappa_{u_i} <\ssup(a)$.

Let $y$ be a lower part conforming with $G$ with $\kappa_y<\ssup(a)$, and suppose that
$\tup{\eta}$ appears in $y$.  Then there is a lower part $y'$ such that
\begin{enumerate}
\item  $y'$ is a direct extension of $y$ in $\Rad{y}$,
\item  $y'$ conforms with $G$, and
\item  $y'$ is harmonious with $A$ past $\eta$. 
\end{enumerate} 
\end{lemma}

\begin{proof}[\ref{new-harmony-lemma}]
By the definition of conformity all sequences appearing in $y$ have the form $u_i$
for some $i$.  We will obtain $y'$ by some judicious shrinking of the measure one sets
appearing as the second entries of pairs in $y$. 

Consider the pairs $(v, B)$ appearing in $y$.  
If $\kappa_v < \eta$ there is no problem, and if 
$\kappa_v = \eta$ then (as $\eta$ is a successor point in $C$) we have
$v = \tup{ \eta_\varepsilon }$ and $B = \emptyset$,
so again there is no problem.

Suppose now that $(v, B)$ appears in $y$ and $\kappa_v > \eta$.
By the definition of conditionhood, $B$ contains no measure sequence
$u$ with $\kappa_u \le \eta$.
Let  $j$ be such that $v = u_j$, and recall that $\lh(v) = 1$ and  $B = \emptyset$ 
for $j$ a successor, and $\lh(v) > 1$ and $B \in {\mathcal F}(v)$ for $j$ a limit.

By hypothesis we have $u_i \in D$ for all $i$ such that $\kappa_{u_i} \ge \eta$;
it follows by the ``geometric'' criterion for genericity (see Theorem (\ref{Mitchell-characterization}))
that if $j$ is a limit then  $D \cap V_{\kappa_v} \in  \mathcal{F}(v)$.  
So when $v = v_j$ for $j$ a limit,  we may shrink $B$ to obtain a set 
$B' = B \cap D$ with
$B' \in {\mathcal F}(v)$. We define $y'$ to be the resulting lower part.
\end{proof}

\begin{lemma} \label{harmonious-refinement}
If $y$ is a lower part which is harmonious with $A$ past $\eta$ and
$y' \in \Rad{y}$, then $y'$ is harmonious with $A$ past $\eta$.
\end{lemma}

\begin{proof}[\ref{harmonious-refinement}]
We consider each pair $(u, B)$ appearing in $y'$.  If $u$
already appears in $y$ then we have that $(u, B')$ appears in
$y$ for some $B'$ such that $B \subseteq B'$, and it is easy to see
that in all cases the harmoniousness conditions are satisfied.
If $u$ does not appear in $y$ then let $(v, C)$ be the
unique pair appearing in $y$ such that $u \in C$, and
consider the various cases of the definition for the
pair $(v, C)$: if $\kappa_v < \eta$ then $\kappa_u < \eta$ and
we are done, the case $\kappa_v = \eta$ cannot occur as $C \neq \emptyset$,
and if $\kappa_v > \eta$ then 
$\set{ u } \cup B \subseteq C \subseteq  \bigcup_{\rho < \rho^*} A_\rho \setminus {\mathcal U}_{\eta^+}$.  
\end{proof}

\section{The definition of $Q(w)$, and useful properties of it.}\label{Pm}

Suppose that $U$ is a measure on a measurable cardinal $\kappa$.
Recall that $\kappa$-\emph{Mathias forcing} using $U$ 
has as conditions pairs $(s,S)$ with $s\in [\kappa]^{<\kappa}$, $S\in U$ and $\ssup(s)\le \min(S)$ and 
ordering given by $(t,T)\le (s,S)$ if $s=t\cap\ssup(s)$ and $T\cup (t\setminus s)\subseteq S$. 
This forcing preserves cardinals as it is $\lk$-directed closed and has the 
$\kappa^+$-chain condition. If $G$ is a generic filter over $V$ for the forcing 
and $y=\bigcup\setof{s}{\exists S\in U \sss (s,S)\in G}$ then for all 
$S\in {\mathcal P}(\kappa)$ we have that $S\in U$ if and only if 
there is some $\alpha<\kappa$ such that $y\setminus\alpha\subseteq S$. 
Consequently, the forcing is also known as the forcing to \emph{diagonalise} 
$U$.\footnote{The earliest uses of this forcing of which we have evidence are in unpublished work from 1992 by Shizuo Kamo \cite{Kamo} 
and Tadatoshi Miyamoto, independently, on the splitting number for uncountable cardinals. See, for example, the reports on this work in
\cite{Suzuki} and \cite{Zapletal}. The latter paper, in particular, gives a full account of their principal result, that it is consistent 
relative to there being a supercompact cardinal that there is an uncountable cardinal $\kappa$ with splitting number 
${\mathfrak s}_\kappa$ greater than $\kappa^+$, and its proof via iterating this forcing. Unfortunately we have not had access 
to the preprint \cite{Kamo} itself. The forcing has also at times been named long Prikry forcing.

A rather similar forcing, however, was given by Henle \cite{Henle} in the early 1980s in the choiceless context of cardinals 
with strong partition properties. Henle called his forcing Radin-like forcing -- the descriptor referring 
to properties of the subset of the cardinal added rather than the form of the conditions.} 

We also recall Ellentuck's topologically inspired notation \cite{Ellentuck}
that if $(s,S)$ is a condition in this forcing 
then $[s,S] = \setof{y\in [\kappa]^{\kappa}}{s\subseteq y\subseteq s\cup S}$. Using this
notation we have that $(t,T)\le (s,S)$ if and only if $[t,T]\subseteq [s,S]$.

In this section we define two forcing posets. The first poset $M(w)$ is the analogue of
the $\kappa$-Mathias forcing for ultrafilter sequences $w$; it will diagonalise the filter ${\mathcal F}(w)$. 
The second poset $Q(w)$  is the variant of $M(w)$ which is tailored, in the style of the analogous forcing 
from \cite{Dzamonja-Shelah-univ-graphs-succ-sing}, to deal with Radin forcing names for graphs, 
and binary relations more generally, on $\kappa^+$. 

One can view $\kappa$-Mathias forcing as adding a subset $y$ of $\kappa$ which is
potentially a member of $U'$ for some normal measure $U'$ extending $U$ in some generic extension.
In the same spirit, $M(w)$ is designed to add a set $A \subseteq V_{\kappa_w}$ which is potentially a member of
${\mathcal F}(w')$ for some sequence $w'$  in a $<\kappa_w$-closed generic extension where
$\lh(w') = \lh(w)$ and $w'_i \supseteq w_i$ for $0 < i < \lh(w)$.

For the rest of this section we fix an ultrafilter sequence $w$ with $\lh(w) > 1$,
and write $\kappa$ for $\kappa_w$.

\begin{definition}\label{Mforcing}
$M(w)$ is the forcing with conditions $p=(A,B)$, with:
\begin{itemize}
\item
$A=\tupof{A_\rho}{\rho<\rho^{\ssss p}}$, where $\rho^{\ssss p}<\kappa$,
\item $\forall \rho<\rho^{\ssss p} \sss  (A_\rho\subseteq {\mathcal U}_\kappa
\sss\sss\sss \& \sss\sss\sss A_\rho \neq \emptyset \sss\sss\sss \& \sss\sss\sss
\exists \kappa_\rho<\kappa\sss \forall u\in A_\rho \sss\sss \kappa_u=\kappa_\rho)$,
\item $\forall \rho<\rho^{\ssss p} \sss\sss \forall u\in A_\rho \sss \forall \tau\in (0,\lh(u)) \sss \sss  u\on \tau\in A_\rho$,
\item $\tupof{\kappa_\rho}{\rho<\rho^{\ssss p}}$ is strictly increasing,
\item $ B\in {\mathcal F}(w)$,
\item $\forall v\in B \sss \forall \tau\in (0,\lh(v))\sss\sss (v\on \tau\in B)$,
and 
\item $\ssup(\setof{\kappa_{\rho}}{\rho<\rho^{\ssss p}}) 
\le \min(\setof{\kappa_v}{v\in B})$.
\end{itemize}
If $p$, $q\in M(w)$ then $q\le p$ if $A^p=A^q\on\rho^{\ssss p}$ and  
$B^q\cup \bigcup_{\rho\in [\rho^{\ssss p},\ssss \rho^{\ssss q})}  A_\rho^q \subseteq B^p$, while $q\le^* p$ if 
$q\le p$ and $A^q=A^p$. If $p$, $q\in M(w)$ and $q\le^* p$ we say that $q$ is a \emph{direct extension} of $p$.
\end{definition}

\begin{definition}\label{ellentuck-open-sets}
Set $[A^p,B^p]$ to be 
\begin{multline*}
\setof{\tupof{D^\rho}{\rho<\kappa} }{
\forall \rho<\kappa \sss \sss (\sss D^\rho\subseteq {\mathcal U}_\kappa\break
\sss\sss\sss \& \sss\sss\sss \exists \kappa_\rho<\kappa\sss \forall u\in D^\rho \sss (\kappa_u=\kappa_\rho)
\sss\sss \& \sss\sss \\
\forall u\in D^\rho \sss \forall \tau\in (0,\lh(u)) \sss \sss  (u\on \tau\in D^\rho  ) \sss)
\sss\sss\sss \&\sss\sss\sss\tupof{\kappa_\rho}{\rho<\kappa}\hbox{ is strictly}\\
\hbox{ increasing } 
\sss\sss\&\sss\sss
\tupof{D^\rho}{\rho<\rho^{\ssss p}} = A^p \sss\sss\sss\&\sss\sss\sss \bigcup \setof{D^\rho}{{\rho\in [\rho^{\ssss p},\kappa)}} \subseteq B^p}.
\end{multline*}
\end{definition}

Then $(A^q,B^q)\le (A^p,B^p)$ if and only if $[A^q,B^q]\subseteq [A^p,B^p]$.

It is useful to have a name for the set of cardinals which are the
first elements of the various ultrafilter sequences appearing anywhere
in $A^p$ for $p\in M(w)$. Accordingly we make the following
definition.

\begin{definition} If $p=(A,B)\in M(w)$ let $a^p=\setof{\kappa_\rho}{\rho<\rho^{\ssss p}}$.
\end{definition}

\vskip6pt
Now we move on to the definition of the forcing $Q(w)$. This is carried out under the following running
combinatorial assumption.

\begin{setting}\label{running-tree-assumption-for-defn-of-Q}
Suppose ${\mathcal T}$ is a binary $\kappa^+$-Kurepa tree with $\Upsilon$ many branches.
\end{setting}

\begin{definition}\label{qstar-forcing}
Let $\tupof{b_\alpha}{\alpha<\Upsilon}$ enumerate a set of branches through ${\mathcal T}$.  
Let $\tupof{\name{E}_{\alpha}}{\alpha <\Upsilon }$ 
be a list of canonical $\Rad{w}$-names for binary relations on $\kappa^+$.  
We will use the sequences $\tupof{b_\alpha}{\alpha<\Upsilon}$ and
$\tupof{\name{E}_{\alpha}}{\alpha <\Upsilon }$ as parameters in the
definition of the forcing $Q(w)$.

$Q^*(w)$ is the forcing with conditions $p=(A,B,t,f)$ satisfying the following four clauses.

\begin{enumerate}
\item \label{one} $(A,B)\in M(w)$ (see Definition (\ref{Mforcing})). We set $a = a^{(A, B)}$.

\item $t\in [(a \cap \sup(a)) \times \Upsilon]^{<\kappa}\hbox{ and } f=\tupof{f^\eta_\alpha}{(\eta,\alpha)\in t}$.
For $\eta\in a \cap \sup(a)$,  set $t^\eta=\setof{\alpha}{(\eta,\alpha)\in t}$.
\item  $\forall \eta\in a \cap \sup(a) \sss\sss \forall \alpha\in t^\eta \sss \sss
d^\eta_\alpha = \dom(f^\eta_\alpha) \in [\kappa^+]^{<\kappa}$.
\item\label{four}
$\forall \eta\in a \cap \sup(a) \sss\sss \forall \alpha\in t^\eta \sss \sss \forall \zeta\in d^\eta_\alpha\sss \exists \nu<\kappa\sss \sss \sss
f^\eta_\alpha(\zeta) = (b_\alpha\on\zeta,\nu)$.
\end{enumerate}
\medskip

If $p$, $q\in Q^* (w)$ then $q\le p$ if $[A^q,B^q]\subseteq
[A^p,B^p]$, $t^p\subseteq t^q$ and\break
$\forall (\eta,\alpha)\in t^p\sss \sss (f^\eta_\alpha)^p\subseteq (f^\eta_\alpha)^q$; 
and $q\le^* p$ if $q\le p$ and $A^q=A^p$, $t^q=t^p$ and $f^q=f^p$.
(If $q\le^* p$ we say $q$ is a \emph{direct extension} of $p$.)

We write $Q(w)$ for the suborder of $Q^*(w)$ consisting of conditions which also satisfy:
\begin{enumerate}
\item[(5)]\label{new-five}
for all $\eta\in a \cap \sup(a)$, for all $\alpha$, $\beta \in t^\eta$, for every lower part $y$ for $\Rad{w}$ harmonious with $A$ past $\eta$, 
and for all $\zeta$, $\zeta'\in d^\eta_\alpha\cap d^\eta_\beta$ we have:
\begin{align*} 
f^\eta_\alpha(\zeta)=f^\eta_\beta(\zeta) \ne &
f^\eta_\alpha(\zeta')=f^\eta_\beta(\zeta')\sss\sss\sss\Longrightarrow\\
& y\concat (w,B)\forces_{\Rad{w}} \hbox{``}\sss \zeta \ssss\name{E}_\alpha \ssss\zeta'
\longleftrightarrow \zeta\ssss\name{E}_\beta \ssss\zeta' \sss\hbox{''}.\\
\end{align*}
\end{enumerate}
\end{definition}

\begin{observation} Observe that in clause (5), if $y = \tupof{(u_i, B_i)}{i \le n}$ 
and $y$ is harmonious with $A$ past $\eta$ for some $\eta \in a$, then $\kappa_{u_n} \le \sup(a)$.
\end{observation}

\begin{observation} It is important to remember that 
$Q(w)$ and $Q^*(w)$ depend on the tree ${\mathcal T}$ and the sequences
$\tupof{b_\alpha}{\alpha<\Upsilon}$ and 
$\tupof{\name{E}_{\alpha}}{\alpha <\Upsilon }$ even though their 
names do not make this dependence explicit.
\end{observation}

The poset $Q(w)$ is designed to add a sequence $A^*$ which diagonalises ${\mathcal F}(w)$
in a sense made precise  in Corollary (\ref{jcol}) below,  
together with various objects that can potentially be understood as $\Rad{w^*}$-names where
$w^*$ is an ultrafilter sequence existing in some $\lk$-directed closed 
forcing extension of $V$, with $\lh(w) = \lh(w^*)$, $w^*_\tau \cap V = w_\tau$
for $0 < \tau < \lh(w)$ and $\bigcup A^* \in {\mathcal F}(w^*)$.

The aim is that if $\tupof{ u_i }{ i < \eta }$ is an $\Rad{w^*}$-generic sequence
(where we note that by the Mitchell criterion this sequence is also $\Rad{w}$-generic)
and $u_i \in \bigcup A^*$ for all $i$ such that $\kappa_{u_i} \ge \eta$ (along with
various other technical conditions), then the ``$\eta$-coordinate'' in the poset $Q(w)$ will add an $\Rad{w^*}$ name for
a binary relation of size $\kappa^+$ together with embeddings of all the relations named by the names $\dot E_\alpha$
into this relation. 

We are particularly interested in the posets $Q(w)$ for certain specific lists of relations, for example when
$\tupof{\name{E}_\alpha}{\alpha<\Upsilon}$ is a list of canonical names for all graphs on $\kappa^+$.  In this case the
``$\eta$-coordinate'' will add a name for a graph whose realization will be universal for the graphs named by the names in the list.

\vskip12pt

\begin{lemma} \label{jdensity1}For every $\varepsilon < \kappa$, the set of $q \in Q(w)$ such that 
$\varepsilon < \kappa^q_\rho$ for some $\rho < \rho^q$ is dense and open. 
\end{lemma}

\begin{proof}[\ref{jdensity1}]
It is immediate that this set is open, so we only need to verify that it is dense.
Let $p \in Q(w)$ and let $\rho = \rho^p$. Find $v \in B^p$ such that
$\kappa_v > \max (\set{ \varepsilon, \sup(a^p) })$ and note that $v \restriction \tau \in B^p$
for $0 < \tau \le \lh(v)$. Now let $\rho^q = \rho+1$,
$A^q \restriction \rho = A^p$,
$A^q_\rho = \setof{ v \restriction \tau }{ 0 < \tau \le \lh(v) }$, 
$B^q = \setof{ v' \in B^p}{ \kappa_{v'} > \kappa_v }$, $t^q = t^p$ and $f^q = f^p$. 

We claim that $q \in Q(w)$ and $q \le p$. The only non-trivial point is to check that
$q$ satisfies Clause (5) in the definition of $Q(w)$. Let $\eta, \alpha, \beta, \zeta, \zeta'$
be such that  $f^\eta_\alpha(\zeta)=f^\eta_\beta(\zeta) \neq f^\eta_\alpha(\zeta')=f^\eta_\beta(\zeta')$,
and note that since $(\eta, \alpha) \in t^q = t^p$ we have that
$\eta \in a^p \cap \sup(a^p)$. 

Let $y$ be harmonious with $A^q$ above $\eta$. By the definition of harmonious, Definition (\ref{defn-harmonious}), we have that 
$\kappa_y\le \kappa_v$.

If $\kappa_y < \kappa_v$ then clearly $y$ is
harmonious with $A^p$ above $\eta$, and we are done by Clause (5) for $p$ and the remark that
$y \concat (w, B^q) \le y \concat (w, B^p)$.  

If $\kappa_y = \kappa_v$ then we have 
$y = y' \concat (v \restriction \tau, D)$ for some $\tau\le \lh(v)$, where $\kappa_{y'} < \kappa_v$, 
$y'$ is harmonious with $A^p$ and also $D \subseteq \bigcup_{\rho < \rho^p} A^p_\rho$.
By the choice of $\kappa_v$ to be strictly greater than $\sup(a^p)$ we see that $D$ is bounded in $V_{\kappa_v}$ and so
cannot be of measure $1$ for any measure on $v$, hence necessarily $\tau  =1$ and $D = \emptyset$,
so that actually $v \restriction \tau = \tup{ \kappa_v }$. 
It follows that $y \concat (w, B^q) \le y' \concat (w, B^p)$, and again we are done
by Clause (5) for $p$.  \break $\vphantom{wwww}$
\end{proof} 

\begin{corollary} \label{jcol}
If $G$ is $Q(w)$-generic over $V$ and $A^*$ is the union of the sequences
$A^p$ for $p \in G$, then $A^*$ is a $\kappa$-sequence and for every $B \in {\mathcal F}(w)$
there is $\rho^* < \kappa$ such that $A^*_\rho \subseteq B$ for $\rho^* < \rho < \kappa$. 
\end{corollary}

The proof of Lemma (\ref{jdensity1}) also gives the following result.

\begin{corollary} \label{jdensity2} The set of $p \in Q(w)$ such that $\rho_p$ is a successor ordinal
is dense in $Q(w)$. 
\end{corollary} 

\begin{notation} For $\eta<\kappa$, $\alpha<\Upsilon$ and $\zeta<\kappa^+$ 
let ${\mathcal D}^\eta_{\zeta,\alpha}=
\setof{p\in Q(w)}{\alpha\in (t^\eta)^p\,\,\&\,\, \zeta\in (d^\eta_\alpha)^p)}$.
\end{notation}

\begin{lemma}\label{one-point-ext}
Suppose $p\in Q(w)$, $\eta\in a^p$, $\alpha<\Upsilon$ and $\zeta<\kappa^+$. Then there is 
$q\in {\mathcal D}^\eta_{\zeta,\alpha}$ with $q\le p$, $A^q=A^p$, $B^q=B^p$,
$(t^\eta)^q=(t^\eta)^p\cup\set{\alpha}$ and 
$(d^\eta_\alpha)^q=(d^\eta_\alpha)^p\cup\set{\zeta}$. (In the last clause we formally take $(d^\eta_\alpha)^p=\emptyset$ if
$\alpha\ni (t^\eta)^p$.)
\end{lemma}

\begin{proof}[\ref{one-point-ext}]
If $p\in {\mathcal D}^\eta_{\zeta,\sss \alpha}$ then take $q=p$ and there is nothing more to do.

So suppose that $p\not\in {\mathcal D}^\eta_{\zeta,\sss \alpha}$. 
If $\alpha\not\in (t^\eta)^p$ observe that $\tup{A^p,B^p,t^p\cup\set{ (\eta,\alpha) }, f^p\concat \tup{\emptyset}}$ 
is a condition in $Q(w)$ and refines $p$. (The concatenated 1-tuple consisting of the empty set records that the 
value of the new $f$-part of the condition $\tup{A^p,B^p,t^p\cup\set{ (\eta,\alpha) }, f^p\concat \tup{\emptyset}}$ 
at the co-ordinate $(\eta, \alpha)$ is the empty function.)
Consequently we may assume without loss of generality that $\alpha\in (t^\eta)^p$ and $\zeta\notin
(d^\eta_\alpha)^p$.

Pick $\nu\in \kappa\setminus \discretionary{}{}{} 
\bigcup_{(\eta,\beta)\in t^p}\rge((\pi_1\cdot f^\eta_\beta)^p)$, let
$(f^\eta_\alpha)^q=(f^\eta_\alpha)^p\cup \set{\tup{\zeta , (b_\alpha\on\zeta,\nu)}}$ and set 
$(f^{\eta'}_\beta)^q=(f^{\eta'}_\beta)^p$
for $(\eta',\beta)\in t^p$ with $(\eta',\beta)\ne (\eta,\alpha)$.  Then $q=\tup{A^p, B^p, t^p,
\tupof{ (f^\eta_\beta)^q}{(\eta,\beta)\in t^p}}$ is clearly an element of $Q^*(w)$ and
refines $p$.  

Moreover, $q$ satisfies (5) of the definition of $Q(w)$. In brief this is true because 
$q$ inherits the truth of (5) from $p$ for every collection of data for instances of (5) except one, and for that
collection we have ensured (5) holds vacuously.

In more detail, suppose $\eta'\in a^{(A^q,B^q)}\cap \sup(a^{(A^q,B^q)})$, $\beta$, $\gamma\in (t^{\eta'})^q$, 
$y$ is a lower part for $\Rad{w}$ harmonious with $A^q$ 
past $\eta'$ and $\zeta'$, $\zeta^{''} \in (d^{\eta'}_\beta)^q\cap (d^{\eta'}_\gamma)^q$
Either 
$\eta'\in a^{(A^p,B^p)}\cap \sup(a^{(A^p,B^p)})$, $\beta$, $\gamma\in (t^{\eta'})^p$, 
$y$ is a lower part for $\Rad{w}$ harmonious with $A^p$ 
past $\eta'$ and $\zeta'$, $\zeta^{''} \in (d^{\eta'}_\beta)^p\cap (d^{\eta'}_\gamma)^p$, and hence
\begin{align*} 
f^{\eta'}_\gamma(\zeta')=f^{\eta'}_\beta(\zeta') \ne &
f^{\eta'}_\gamma(\zeta'')=f^{\eta'}_\beta(\zeta'')\sss\sss\sss\Longrightarrow\\
& y\concat (w,B^p)\forces_{\Rad{w}} \hbox{``}\sss \zeta' \ssss\name{E}_\gamma \ssss\zeta^{''}
\longleftrightarrow \zeta'\ssss\name{E}_\beta \ssss\zeta{''} \sss\hbox{''}.\\
\end{align*}
Or $\eta'=\eta$, $\gamma=\alpha$ and $\zeta''=\zeta$ and we have that
$\nu$ was chosen specifically so that $(f^\eta_{\alpha})^q(\zeta)\ne (f^\eta_{\beta})^p(\zeta')$ and hence the antecedent of the 
instance of the implication in (5) of the definition of $Q(w)$ is false and so for this collection of data (5) holds. 
\end{proof}

\begin{corollary}\label{density-for-t-f-for-Q}
For each $\eta<\kappa$, each $\alpha<\Upsilon$ and each $\zeta\in \kappa^+$,
the set ${\mathcal D}^\eta_{\zeta,\alpha}=\setof{p\in Q(w)}{\alpha\in (t^\eta)^p\,\,\&\,\, \zeta\in (d^\eta_\alpha)^p)}$ 
is dense and open in $\setof{p\in Q(w)}{\eta\in a^p}$.
\end{corollary}

\begin{proof}[\ref{density-for-t-f-for-Q}]
Immediate from Lemma (\ref{one-point-ext}).
\end{proof}

\begin{corollary} 
If $G$ is $Q(w)$-generic over $V$, $\eta \in \bigcup \setof{ a^p }{ p \in G }$
and $\alpha < \Upsilon$ then $\bigcup \setof{ f^{p, \eta}_\alpha }{ p \in G, (\eta, \alpha) \in t^p }$
is a function with domain $\kappa^+$.
\end{corollary}

\begin{lemma}\label{splitting-for-Q}
$Q(w)$ has splitting.
\end{lemma}

\begin{proof}[\ref{splitting-for-Q}]
This is proved using an argument very similar to the one
for Lemma (\ref{one-point-ext}).  As observed there, given a condition $p$ we can easily find
a condition in $Q(w)$ refining it with a non-empty $t$-part, so without loss
of generality we may as well assume $t^p\neq\emptyset$.  Let
$(\eta,\alpha)\in t^p$. Now pick $\zeta\in\kappa^+\setminus (d^\eta_\alpha)^p$ and
working exactly as before we can define two extensions $q_0$, $q_1$ of
$p$ which are both in ${\mathcal D}^\eta_{\zeta,\alpha}$ but satisfy
$(f^\eta_\alpha)^{q_0}(\zeta)\neq (f^\eta_\alpha)^{q_1}(\zeta)$, simply by choosing, 
at the ultimate stage of that argument, two distinct elements $\nu_0$ and $\nu_1$ 
of  $\kappa\setminus \discretionary{}{}{}\bigcup_{(\eta,\beta)\in t^p}\rge((\pi_1\cdot f^\eta_\beta)^p)$, where 
previously we merely picked one. Then $q_0$ and $q_1$ are incompatible.\break $\vphantom{WWW}$
\end{proof}

\begin{lemma}\label{kwcompact-and-glbs}
\begin{enumerate}
\item\label{kwcompact} The forcing poset $Q(w)$ is $\lk$-compact. 
\item\label{desc-seqs-in-Q-have-glbs} 
Descending sequences from $Q(w)$ of length less than $\kappa$ have greatest lower bounds.
\end{enumerate}
\end{lemma}

\begin{proof}[\ref{kwcompact-and-glbs}]

We start by proving (\ref{kwcompact}). Let $C = \setof{q_i}{i<\mu}$ be a centred subset of $Q(w)$ for some $\mu < \kappa$,  and  
fix for each finite set $Y \subseteq \mu$  some condition $s_Y\in Q(w)$ such that $s_Y \le q_i$ for all
$i \in Y$.
We will define a quadruple $(A^*, B^*, t^*, f^*)$ of the appropriate type, and prove that  
it is a condition in $Q(w)$ and forms a lower bound for $C$.

Let $\rho^* = \sup \setof{\rho^{q_i}}{i < \mu}$.  By the hypothesis that $C$ is centred it is easy to see that
there is a unique sequence $A^* = \tupof{A_\rho}{\rho < \rho^*}$ such that 
$A^* \restriction \rho^{q_i} = A^{q_i}$ for all $i < \mu$. For if $i < j < \mu$ there is some 
$p\in Q(w)$ such that $p\le q^i$, $q^j$ and hence by the definition of $\le$ (see Definition (\ref{Mforcing}))
one has that $A^p \on \rho^{q_i} = A^{q_i}$ and $A^p\on \rho^{q_j} = A^{q_j}$; 
thus either $A^{q_i}$ is an initial segment of $A^{q_j}$ or vice versa.
We write $\kappa_\rho$ for the common value of $\kappa_v$ for sequences $v \in A_\rho$,
and $a^*$ for $\setof{\kappa_\rho}{\rho<\rho^*}$. 

Let $B^* = \bigcap_{Y \in [\mu]^{<\omega}} B^{s_Y}$. By the completeness of the filter
${\mathcal F}(w)$ we see that  $B^* \in {\mathcal F}(w)$. It also follows from the
definition of $Q(w)$ that $\kappa_\rho < \kappa_v$ for every $\rho < \rho^*$ and every $v \in B^*$, for
if $\kappa_\rho \in a^*$ there is some $i<\mu$ such that $\kappa^\rho\in a^{q_i}$ and hence for all
$v\in B^{q_i}$ one has $\kappa_\rho <\kappa_v$, and thus, as $B^*\subseteq B^{q_i}$ one has $\kappa_\rho<\kappa_v$.

Let $t^* = \bigcup_{i < \mu} t^{q_i}$.  Clearly $t^* \subseteq (a^* \cap \sup(a^*)) \times \Upsilon$
and $\vert t^* \vert < \kappa$.

For each $(\eta, \alpha) \in t^*$, let  
$d^{*,\eta}_\alpha = \bigcup \setof{d^{q_i,\eta}_\alpha}{(\eta, \alpha) \in t^{q_i},\, i < \mu}.$

By the hypothesis that $C$ is centered it is easy to see that there is a unique sequence of functions
$f^* = \tupof{f^{*,\eta}_\alpha}{(\eta,\alpha) \in t^*}$ with
$\dom(f^{*,\eta}_\alpha) = d^{*,\eta}_\alpha$, and 
$f^{*,\eta}_\alpha(\zeta) = f^{q_i,\eta}_\alpha(\zeta)$ for all $\zeta \in d^{q_i,\eta}_\alpha$.

It should be clear that $(A^*, B^*, t^*, f^*)$ is a condition in
$Q^*(w)$, and that if $(A^*, B^*, t^*, f^*)$ is a condition in $Q(w)$
then it forms a lower bound for $C$.  So to finish the proof, we must
verify clause (5) in the definition of $Q(w)$.

Let $\eta\in a^* \cap \sup(a^*)$ with $\eta < \sup(a^*)$, let
$\alpha$, $\beta \in t^{*,\eta}$, let $y = \tupof{(u_i, B_i)}{i \le  n}$ 
be a lower part harmonious with $A^*$ past $\eta$, and let
$\zeta$, $\zeta'\in d^{*,\eta}_\alpha\cap d^{*,\eta}_\beta$ be such
that $f^{*,\eta}_\alpha(\zeta)=f^{*,\eta}_\beta(\zeta) \ne
f^{\eta,*}_\alpha(\zeta')=f^{*,\eta}_\beta(\zeta')$.

Choose a finite set $Y$ large enough that 
\begin{itemize}
\item There is $i_0 \in Y$ with  $\eta \in a^{q_{i_0}}$,  $\eta < \sup(a^{q_{i_0}})$ and $\kappa_{u_n} \le \sup(a^{q_{i_0}})$.
\item There is $i_1 \in Y$ with $(\eta, \alpha) \in t^{q_{i_1}}$, and $\zeta \in d^{q_{i_1}, \eta}_\alpha$. 
\item There is $i_2 \in Y$ with $(\eta, \beta) \in t^{q_{i_2}}$, and $\zeta \in d^{q_{i_2}, \eta}_\beta$. 
\item There is $i_3 \in Y$ with $(\eta, \alpha) \in t^{q_{i_3}}$, and $\zeta' \in d^{q_{i_3}, \eta}_\alpha$. 
\item There is $i_4 \in Y$ with $(\eta, \beta) \in t^{q_{i_4}}$, and $\zeta' \in d^{q_{i_4}, \eta}_\alpha$. 
\end{itemize}

Now consider the condition $s_Y$. By the construction 
$A^{q_{i_0}}$ is an initial segment of $A^{s_Y}$ and so by the choice of $i_0$ 
one has that  $y$ is harmonious with $A^{s_Y}$ past $\eta$.
Since $s_Y \in Q(w)$ it follows easily that 
\[
y\concat (w,B^{s_Y})\forces_{\Rad{w}} \hbox{``}\sss \zeta \ssss\name{E}_\alpha \ssss\zeta'
\longleftrightarrow \zeta\ssss\name{E}_\beta \ssss\zeta' \sss\hbox{''}.
\]
By construction $B^* \subseteq B^{s^Y}$, and so $y\concat (w,B^*)$ also
forces this equivalence and we are done. 

In order to see that (\ref{desc-seqs-in-Q-have-glbs}) holds, suppose that 
$C$ as in the proof of (\ref{kwcompact}) is a descending sequence of conditions.
Then $(A^*,B^*,t^*,f^*)$ is a greatest lower bound for $C$.
\end{proof}
\vskip6pt

\begin{corollary}\label{kwdirclosed-and-cntblycmpt} $Q(w)$ is both $\lk$-directed closed and countably compact. 
\end{corollary}

The following result is useful in the proof, which appears in the next section,  that $Q(w)$ has the stationary $\kappa^+$-chain condition.

\begin{lemma}\label{squaring-off-the-ds} Let $p=(A,B,t,f)\in Q(w)$. Let
${\mathcal t}  = \bigcup \setof{t^{\eta}}{\eta\in a^p}$ and let
${\mathcal d} = \bigcup \setof{d^\eta_\alpha}{\eta\in a^p \sss\&\sss \alpha \in t^\eta}$.
Then there is $q\in Q(w)$ with $q\le p$, $A^q=A^p$, $B^q=B^p$, $t^\eta={\mathcal t}$ 
for all $\eta\in a^p$ and $(d^\eta_\alpha)^q = {\mathcal d}$ for all $\eta\in a^p$ and 
$\alpha\in t^\eta$.
\end{lemma}

\begin{proof}[\ref{squaring-off-the-ds}]
As in the second paragraph of the proof of Lemma (\ref{one-point-ext}), we have that
$r=\tup{A^p,B^p,t^p\cup\setof{ (\eta,\alpha) }{\eta\in a^p \sss\sss\&\sss\sss \alpha\in {\mathcal t}\setminus t^{\eta}}, 
f^p\concat \tupof{\emptyset}{\eta\in a^p \sss\sss\&\sss\sss \alpha\in {\mathcal t}\setminus t^{\eta}}}$ 
is a condition in $Q(w)$ which refines $p$ and is such that $a^r=a^p$ and
$(t^{\eta})^{r}={\mathcal t}$ for all $\eta\in a^r$.

Now enumerate $D =\setof{(\eta,\alpha,\zeta)}{\eta\in a^p \sss\sss\&\sss\sss \alpha\in {\mathcal t} \sss\sss\&\sss\sss
\zeta\in {\mathcal d} \setminus d^{\eta}_{\alpha}}$ as $\setof{(\eta_i,\alpha_i,\zeta_i)}{i<\gamma}$ for some 
$\gamma\le \card{a^p\times {\mathcal t} \times {\mathcal d}}$. 

If $D=\emptyset$ then $p$ itself satisfies the properties required for $q$.

Otherwise, carry out an induction on $\gamma$, using Lemma (\ref{one-point-ext}) at 
initial and successor stages and Lemma (\ref{kwcompact-and-glbs}.(\ref{desc-seqs-in-Q-have-glbs}))
 at limits, 
to construct a descending sequence of conditions
$\tupof{r_i}{i<\gamma}$ with $r_0\le r$ and such that $\zeta_i \in (d^{\eta_i}_{\alpha_i})^{r_i}$ for each $i<\gamma$.

Finally, if $\gamma$ is a limit ordinal use Lemma (\ref{kwcompact-and-glbs}.(\ref{desc-seqs-in-Q-have-glbs})) to choose a lower bound 
$q$ for $\tupof{r_i}{i< \gamma}$, and if $\gamma=\gamma'+1$ is a successor set $q=r_{\gamma'}$.\break $\vphantom{WW}$
\end{proof} 

We now introduce the notion of a \emph{weakening} of a condition.

\begin{definition} \label{rho-star-weakening}
Let $p$ be a condition in $Q(w)$ such that  $\rho^p$ is a successor ordinal and let $\rho^* < \rho^p$. 
We define the {\em $\rho^*$-weakening of $p$} to be the quadruple $r = (A^r, B^r, t^r, f^r)$ where
$A^r = A^p \restriction \rho^*$, $a^r = \setof{ \kappa^p_\rho }{ \rho < \rho^* }$,
$B^r = B^p \cup \bigcup \setof{ A^p_\rho }{ \rho^* \le \rho < \rho^p }$,
$t^r = \setof{  (\eta, \alpha) \in t^p }{ \eta \in a^r \cap \sup(a^r) }$ and 
$f^r = \tupof{ f^{p,\eta}_\alpha}{ (\eta, \alpha) \in t^r }$.
\end{definition}

\begin{lemma} \label{abacus1}
Let $p \in Q(w)$ with $\rho^p$ a successor ordinal, let $\rho^* < \rho^p$ 
and let $r$ be the $\rho^*$-weakening of $p$. Then
\begin{enumerate}
\item $r \in Q(w)$,
\item $p \le r$,
\item \label{thirdbit}  and for every $\eta \in a^r$ and every lower part $y$ which is harmonious
with $A^r$ past $\eta$, 
if $f^{p,\eta}_\alpha(\zeta) =  f^{p,\eta}_\beta(\zeta) \neq f^{p,\eta}_\alpha(\zeta') =  f^{p,\eta}_\beta(\zeta')$,
then $y \concat (w, B^r) \forces \hbox{``}\sss 
\zeta \name{E}_{\alpha} \zeta'  \Longleftrightarrow \zeta \name{E}_{\beta} \zeta' .\sss\hbox{''}$
\end{enumerate}  
\end{lemma} 

\begin{proof}[\ref{abacus1}]It is easy to see that $r \in Q^*(w)$, and that $p$ refines $r$ in
$Q^*(w)$. Since $f^{p,\eta}_\alpha = f^{r,\eta}_\alpha$ when $(\eta, \alpha) \in t^r$,
Clause (\ref{thirdbit}) in the conclusion  implies that $r$ satisfies
Clause (5) in Definition (\ref{qstar-forcing}) (and hence that $r$ is in $Q(w)$). It will therefore 
suffice to verify Clause (\ref{thirdbit}).

Suppose for a contradiction that $y \concat (w, B^r)$ does not force the desired equivalence, then
there is an extension $y' \concat (w, B)$ forcing that the equivalence  is false.
Shrinking $B$ if necessary we may assume that $B \subseteq B^p$.

We will break up $y'$ as $y_0 \concat y_1 \concat y_2$, where
$y_0 \in \Rad{y}$, 
$y_1$ consists of pairs $(u, C)$ such that
$u$ is drawn from $\bigcup \setof{ A^p_\rho }{ \rho^* \le \rho < \rho^p }$,
and  $y_2$ consists of pairs $(u, C)$ such that  $u$ is drawn from $B^p$.
This is possible because, by the definition of extension in $\Rad{w}$,
all pairs $(u, C)$ in $y'$ with $\kappa_u > \kappa_y$ have $u \in B^r$. 

We claim that  
$y_0 \concat y_1$ is harmonious with $A^p$ above $\eta$.
An appeal to Lemma (\ref{harmonious-refinement}) shows that $y_0$ is harmonious with $A^r$ above $\eta$,
which handles the pairs appearing in $y_0$.
Let $(u, C)$ be a pair appearing in $y_1$, and  observe that 
$\eta \in a^r = \setof{ \kappa^p_\rho}{ \rho < \rho^* }$,  
while $u\in \setof{ A^p_\rho }{ \rho^* \le \rho < \rho^p }$,
so that  $\eta < \kappa_{\rho^*} \le \kappa_u$.
By the definition of the ordering of $\Rad{w}$ we have $C \subseteq B^r$, 
and since $\kappa_u = \kappa_\rho$ for some $\rho < \rho^p$ in fact
$C \subseteq \bigcup \setof{ A^p_\rho }{ \rho^* \le \rho < \rho^p }$. 
In particular as $\eta < \kappa_{\rho^*}$ we have that $C \cap V_{\eta  +1 } = \emptyset$.

We claim that by shrinking measure one sets appearing in $y_2$ (if necessary)
we may assume that 
\[ y' \concat (w, B) = y_0\concat y_1 \concat y_2 \concat (w,B) \le y_0 \concat y_1 \concat (w, B^p). \]
For pairs $(u, C)$ appearing in $y_2$ we have that $u \in B^p$ and
$C \subseteq B^r = B^p \cup \bigcup \setof{ A^p_\rho }{ \rho^* \le \rho < \rho^p }$; 
since $\rho^p$ is a successor ordinal we have that 
$\kappa_v \le \kappa_{\rho^p - 1} < \kappa_u$ for all $v \in \bigcup \setof{ A^p_\rho }{ \rho^* \le \rho < \rho^p }$,
so that we may shrink $C$ to obtain a pair $(u, D)$ with $D \subseteq B^p$.

Since $y_0 \concat y_1$ is harmonious with $A^p$ above $\eta$ and $p$ is a condition,
$y_0 \concat y_1 \concat (w, B^p)$ forces the equivalence 
$\zeta \name{E}_{\alpha} \zeta'  \Longleftrightarrow \zeta \name{E}_{\beta} \zeta'$.
This is a contradiction as $y' \concat (w, B)$ is an extension 
of $y_0 \concat y_1 \concat (w, B^p)$ and forces that the same equivalence
fails.
\end{proof}

\begin{corollary} \label{abacus2} 
Let $p$ and $q$ be conditions in $Q(w)$ with $p \le q$ and $\rho_p$
a successor ordinal. Let $\rho^q \le \rho^* < \rho^p$ and let $r$ be the
$\rho^*$-weakening of $p$. Then $p \le r \le q$. 
\end{corollary} 

\begin{proof}[\ref{abacus2}] 
By Lemma (\ref{abacus1}) we already know that $p \le r$ so we need only to check that
$r \le q$. It is routine to check that 
$A^r \restriction \rho^q = A^p \restriction \rho^q = A^q$, 
$t^r \supseteq t^q$, $f^{r, \eta}_\alpha = f^{p, \eta}_\alpha \supseteq f^{q, \eta}_\alpha$
for all $(\eta, \alpha) \in t^q$, and finally
$B^r \cup \bigcup \setof{ A^r_\rho }{ \rho^q \le \rho < \rho^r }
= B^p \cup\bigcup \setof{ A^p_\rho }{ \rho^q \le \rho < \rho^p } \subseteq B^q$.
\end{proof}

\section{$Q(w)$ has the stationary $\kappa_w^+$--chain condition}\label{chain}

For the duration of this section we fix an ultrafilter sequence $w$ with $\lh(w) > 1$,  and 
let $\kappa = \kappa_w$. As in the previous section, let ${\mathcal T}$ be a $\kappa^+$-Kurepa tree, 
let $\tupof{b_\alpha}{\alpha<\Upsilon}$ be an enumeration of a set of branches through ${\mathcal T}$ and
let $\tupof{\name{E}_{\alpha}}{\alpha <\Upsilon }$ be a list of canonical
$\Rad{w}$-names for binary relations on $\kappa^+$. Let $Q = Q(w)$.
\begin{notation}\label{enum-levels-of-Kurepa-trees}
For each $\alpha<\kappa^+$ let ${\mathcal T}_\alpha=\setof{x\in {\mathcal T}}{\lh(x)=\alpha}$ and let
${\mathcal e}_\alpha:{\mathcal T}_\alpha\longrightarrow\card{{\mathcal T}_\alpha}$ be a 
bijection enumerating ${\mathcal T}_\alpha$. 
\end{notation}
As ${\mathcal T}$ is a $\kappa^+$-Kurepa tree one has for $\alpha<\kappa^+$ that 
$\card{{\mathcal T}_\alpha}\le\kappa$.\footnote{ We remark that we use Kurepa trees in the definition
of the $Q(w)$, rather than arbitrary $\kappa^+$-trees with $\Upsilon$-many branches, 
because the fact that each ${\mathcal T}_{\alpha}$ has size at most $\kappa$ 
saves us a little work in the proof of the $\kappa^+$-stationary chain condition given in this section -- 
see the definitions of the $g^{\eta,i}$ and $X^{\eta,i}$ and their roles in the definition of the function $h$ below.
In \cite{Cummings-Dzamonja-Morgan} an analogous chain condition argument is given without 
this nicety.}

\vskip12pt
\begin{notation} For $\alpha<\beta<\Upsilon$ let $\Delta(\alpha,\beta)$
be the least $\zeta$ such that $b_\alpha\on \zeta \ne b_\beta \on \zeta$. 
\end{notation}

\begin{notation} If $p= (A,B',t,f) \in Q$ and there are ${\mathcal t}\in [\Upsilon]^{<\kappa}$ and
${\mathcal d}\in [\kappa^+]^{<\kappa}$ such that ${\mathcal t}=t^{\eta}$ for all $\eta\in a$ and
${\mathcal d}=\dom(f^{\eta}_\alpha)$ for all $\eta\in a$ and  $\alpha\in {\mathcal t}$, 
let  $Z^p$ be the set of all tuples $(\eta, y, x,\zeta, \zeta', \delta, \delta')$ where
$\eta\in a \cap \sup(a)$, $y$ is a lower part harmonious with $A$ past $\eta$,  $x\in {\mathcal T}$, 
$\zeta<\zeta'<\kappa^+$, $\zeta, \zeta'\in {\mathcal d}$ and $\delta$, $\delta'<\kappa$.
\end{notation}

\begin{lemma}\label{use-of-truthvalue-reduction} Suppose $p'= (A,B',t,f) \in Q$ and there are 
${\mathcal t}\in [\Upsilon]^{<\kappa}$ and
${\mathcal d}\in [\kappa^+]^{<\kappa}$ such that ${\mathcal t}=t^{\eta}$ for all $\eta\in a$ and
${\mathcal d}=\dom(f^{\eta}_\alpha)$ for all $\eta\in a$ and  $\alpha\in {\mathcal t}$. 

Then there is some $B\in {\mathcal F}(w)$ with $B\subseteq B'$, so that $p=(A,B,t,f)\le p'$, and such that 
whenever $z=(\eta, y, x,\zeta, \zeta', \delta, \delta')\in Z^p$ and $\alpha\in {\mathcal t}$ is such that  $x=b_\alpha\on\zeta'$, 
$f^{\eta}_\alpha(\zeta) = (x \restriction \zeta, \delta)$ and $f^{\eta}_\alpha(\zeta') = (x, \delta')$
there is an $\Rad{y}$-name $\name{\sigma}_z$ such that $y \concat (w, B)$ forces that
$\name{\sigma}_z$ is the truth value of the assertion that $\zeta \name{E}_\alpha \zeta'$.
\end{lemma}

\begin{proof}[\ref{use-of-truthvalue-reduction}] First of all notice that $Z^p=Z^{p'}$.
For each $z\in Z^p$ we may define an $\Rad{w}$-name $\name{\tau}_z$ such that
for all $\alpha\in {\mathcal t}$ with $f^{\eta,i}_\alpha(\zeta) = (x \restriction \zeta, \delta)$
and $f^{\eta,i}_\alpha(\zeta') = (x, \delta')$,  $y \concat (w, {\bar B}^i)$ forces that
$\name{\tau}_z$ is the truth value of the assertion that $\zeta \name{E}_\alpha \zeta'$.  (Condition (5) in the definition 
of being a condition shows $\name{\tau}_z$ does not depend on $\alpha$.)

By Lemma (\ref{truthvalue-reduction}), for each $z\in Z^p$ there is a set $B_z \subseteq {\bar B}$
such that $y \concat (w, B_z)$ reduces $\name{\tau}_z$ to an $\Rad{y}$-name
$\name{\sigma}_z$.  As $\card{ Z^p} < \kappa$ any $B \in {\mathcal F}(w)$ such that $B \subseteq \bigcap_{z\in Z^p} B_z$ 
suffices.
\end{proof}
\vskip6pt

\begin{proposition}\label{Qhaschaincond} 
$Q$ has the stationary $\kappa^+$-chain condition.
\end{proposition}

\begin{proof}[\ref[{Qhaschaincond}]
Let $\setof{p'^{i}}{i<\kappa^+} \in [Q]^{\kappa^+}$. 
Let ${p}'^{i}=(A^i,B'^i,t'^{i},f'^{i})$ for each $i<\kappa^+$, write $a^i$ for $a^{{\scriptstyle p}^i}$, $\rho^i$ for 
$\rho^{{\scriptstyle p}^i}$, and for each $\rho<\rho^i$
write $\kappa^i_\rho$ for $(\kappa_\rho)^{{\scriptstyle p}^i}$. 

We start by tidying up, for each $i<\kappa^+$, the collection of domains of 
constituents of the $f'^{i}$.  We want to emphasize that this step in the argument 
is not strictly necessary. At the cost of a more elaborate case analysis below, 
proofs of the stationary chain condition can be given which
do \emph{not} rely on the forcing having the property of having greatest lower bounds for 
descending sequences of length less than $\kappa$, a property which we do use 
in carrying out this tidying up. 

By applying Lemma (\ref{squaring-off-the-ds}) and then
Lemma (\ref{use-of-truthvalue-reduction})
we can find $\setof{{p}^i}{i<\kappa^+} \in [Q]^{\kappa^+}$
such that for each $i<\kappa^+$ we have ${p}^{i}=(A^i,B^i,t^{i},f^{i}) \le {p}'^{i}$
and there is some ${\mathcal t}^i\in [\Upsilon]^{<\kappa}$ and some
${\mathcal d}^i\in [\kappa^+]^{<\kappa}$ such that  
\begin{itemize}
\item ${\mathcal t}^i=t^{\eta,i}$ for all $\eta\in a^i$,
\item ${\mathcal d}^i=\dom(f^{\eta,i}_\alpha)$ for all $\eta\in a^i$ and 
$\alpha\in {\mathcal t}^{i}$, and 
\item $B^i$ has the property stated in the conclusion of Lemma (\ref{use-of-truthvalue-reduction}).
\end{itemize}

Now we make a plethora of auxiliary definitions for each $i$,
set out in the table below.

Let $\setof{\alpha_\gamma}{\gamma<\gamma^*}$ be an 
enumeration of $\bigcup \setof{{\mathcal t^{i}}}{i<\kappa^+}$, for some 
$\gamma^*\le \kappa^+$. Let $\setof{\alpha^{i}_\gamma}{\gamma<\gamma^{i}}$ 
be the increasing enumeration of ${\mathcal t}^{i}$, for some $\gamma^{i}<\kappa$, for each $i<\kappa^+$.

Next, for each $i<\kappa^+$ and $\eta\in a^i$ let
\begin{itemize}
\item $\theta^{i}_0= \ssup(\setof{\gamma}{\alpha_\gamma\in {\mathcal t^{i}}})$,  
$\theta^{i}_1=\ssup({\mathcal d}^{i})$, 
\item  $\Gamma^{i}= \setof{\Delta(\alpha ,\alpha')}{\alpha,\sss \alpha'\in {\mathcal t}^{i}}$,
$\nu^{i}=\ssup(\Gamma^{i})$,
\item $T^{i}=\setof{\gamma<i}{\alpha_\gamma\in {\mathcal t}^{i}}$,  $\Delta^{i}=\Gamma^{i}\cap i$,
$D^{i}={\mathcal d}^{i}\cap i$, 
\item $Y^{\eta,i}= \setof{y}{y\hbox{ is  a lower part harmonious with }A^i\hbox{ past }\eta}$.
\item $g^{\eta,i}(\zeta,\gamma) =  
({\mathcal e}_{\zeta}(\pi_0(f^{\eta,i}_{\alpha_\gamma}(\zeta))),\pi_1(f^{\eta,i}_{\alpha_\gamma}(\zeta)))$ 
for $\gamma\in T^i$  and $\zeta\in D^{i}$.
\item
$X^{\eta,i}=\setof{\tup{y,\varepsilon,\zeta,\zeta',F^{\eta}_\varepsilon,\name{\sigma}}}{
y\in Y^{\eta,i} \discretionary{}{}{}  \hskip5pt\&\hskip5pt \varepsilon<\gamma^{i}  \discretionary{}{}{} 
\hskip5pt\&\hskip5pt
\zeta,\sss \zeta'\in D^i \sss\sss\& \discretionary{}{}{} \zeta<\zeta' \hskip5pt\&\hskip5pt 
\hbox{on setting }\discretionary{}{}{} 
\alpha = \alpha^{i}_\varepsilon 
\discretionary{}{}{} \hbox{ one has }
F^\eta_\varepsilon: D^{i}\longrightarrow \kappa\times\kappa \hbox{ is}\discretionary{}{}{} \hbox{given by}\discretionary{}{}{} 
F^\eta_\varepsilon (\zeta^*)= ({\mathcal e}_{\zeta^*}(\pi_0(f^{\eta,i}_{\alpha}(\zeta^*))),\pi_1(f^{\eta,i}_{\alpha}(\zeta^*)))
\hbox{ for } \zeta^*\in D^{i} \discretionary{}{}{}
\&\hbox{ letting }
\discretionary{}{}{} x=b_{\alpha}\on \zeta' \hbox{, }  
f^{\eta,i}_\alpha(\zeta)\discretionary{}{}{} = \discretionary{}{}{}  {(x \on \zeta, \delta)}  \discretionary{}{}{} \hbox{ and } 
\discretionary{}{}{} f^{\eta,i}_\alpha(\zeta') = (x, \delta') \hbox{,}
\discretionary{}{}{} \hbox{we have }
 {\name{\sigma}=\name{\sigma}^i_{(y,x,\eta,\zeta,\zeta',\delta,\delta')} }
}$ .
\end{itemize}
Recall that $\pi_l$ is projection onto the $l$th co-ordinate, so in the definition of 
$g^{\eta,i}$, for example, 
$\pi_0(b_{\alpha_\gamma}\on \zeta,\nu)=b_{\alpha_\gamma}\on\zeta$ and
$\pi_1(b_{\alpha_\gamma}\on \zeta,\nu)=\nu$, while, as per Notation (\ref{enum-levels-of-Kurepa-trees}), 
${\mathcal e}_\zeta$ is the function 
enumerating the $\zeta$th level of the $\kappa^+$-Kurepa tree ${\mathcal T}$, 
so that ${\mathcal e}_\zeta(b_{\alpha_\gamma}\on\zeta) \in \kappa$.
Thus $g^{\eta,i}: D^{i}\times T^{i}\longrightarrow \kappa\times\kappa$ and
$\dom(g^{\eta,i}(.,\gamma)) =D^i$ for $\gamma\in T^i$.

Notice, also, that the various $F^\eta_\varepsilon$, $\name{\sigma}$, $x$, $\delta$ and $\delta'$ appearing 
in the definition of $X^{\eta,i}$ are uniquely determined by $i$, $\eta$ and the $\tup{y,\varepsilon,\zeta,\zeta'}$.

For the convenience of the reader we record the types of the objects which we have just defined.
$\alpha_\gamma \in \Upsilon$ for $\gamma < \gamma^* \le \kappa^+$,
$\alpha^{i}_\gamma \in \Upsilon$ for $\gamma < \gamma^{i} < \kappa$,
$\theta^{i}_0 < \kappa^+$, $\Gamma^{i} \in [\kappa^+]^{<\kappa}$,
$\nu^{i} < \kappa^+$, $\theta^{i}_1 < \kappa^+$,
$T^{i} \in [i]^{<\kappa}$, $\Delta^{i} \in [i]^{<\kappa}$,
$D^{i} \in [i]^{<\kappa}$,
$Y^{\eta, i} \in V_\kappa$, and $g^{\eta, i}$ is a  partial function of size less than
$\kappa$ from $i \times i$ to $\kappa\times \kappa$.

We next argue similarly to the proof of Claim (\ref{delta-sys-technical-claim}) and define a regressive 
function on a set which is club relative to $S^{\kappa^+}_{\kappa}$. Recall from the section on notation
that if $I$ and $J$ are sets and $\kappa$ is a cardinal then $\hbox{Fn}(I,J,\kappa)$
is the set of partial functions from $I$ to $J$ of size less than $\kappa$.

Let $h^*$ be an injection from $V_\kappa\times
([\kappa^+]^{<\kappa} \times
\hbox{Fn}(\kappa^+\times\kappa^+,\kappa\times\kappa,\kappa)
\times [\kappa^+]^{<\kappa} \times [\kappa^+]^{<\kappa} \times \kappa
\times (V_\kappa\times \kappa\times  \kappa^+\times\kappa^+\times V_{\kappa}\times  
\hbox{Fn}(\kappa^+, \kappa\times\kappa, \kappa) )^{<\kappa}  ) 
\times\kappa$ into $\kappa^+$. Let us write $H$ for the domain of $h^*$.

For $\kappa\le i<\kappa^+$ let 
$H_i$ be defined similarly to $H$ with $i$ in place of $\kappa^+$:
$H_i = V_{\kappa}\times
([\tau]^{<\kappa} \times
\hbox{Fn}(\tau\times\tau,\kappa\times\kappa,\kappa)
\times [\tau]^{<\kappa} \times [\tau]^{<\kappa} \times \kappa
\times (V_{\kappa}\times \kappa\times  \tau\times\tau\times V_{\kappa}\times  
\hbox{Fn}(\tau, \kappa\times\kappa, \kappa) )^{<\kappa}  ) 
\times\kappa$. 

Define $k:[\kappa,\kappa^+)\longrightarrow \kappa^+$ by $k(i)$ is the least $i^*<\kappa^+$ such that 
$H_i \subseteq h^{*-1}``\tau^*$. 

Let $\widetilde{C}=\setof{j<\kappa^+}{\forall i<j\sss (\theta^{i}_0,\sss\theta^{i}_1,\sss\nu^{i},\sss k(i) <j)}.$  

As the intersection of the sets of closure points of the four given functions, $\widetilde{C}$ is a club subset of $\kappa^+$.

Let $h(i) = h^*(A^i,\tupof{(T^{i},g^{\eta,i},D^{i},\Delta^{i},\otp({\mathcal t}^{i}),X^{\eta,i})}{\eta\in a^i},\obar{\kappa}^i)$ 
for $i\in \widetilde{C}\cap S^{\kappa^+}_\kappa$ and $h(i)=0$ otherwise. 

We have that $h^{*-1}(h(i)) \in H_i$ for all $i\in [\kappa,\kappa^+)$.
If $i\in \widetilde{C}\cap S^{\kappa^+}_\kappa$, since $\card{h^{*-1}(h(i))} < \kappa$,
there is some $i'<i$ such that $h^{*-1}(h(i))\in H_{i'}$, and hence 
there is some $\widetilde{i}<i$ such that $h(i)<k(\widetilde{i})$. 

Hence, as $i$ is (amongst other things) a closure point of $k$, we have $h(i)<i$ for all nonzero $i<\kappa^+$. 
\vskip12pt

Now suppose that $i, j \in \widetilde{C} \cap S^{\kappa^+}_\kappa$, $i < j$, and $h(i) = h(j)$. 
In particular we have that $A^i=A^j$ (and hence $a^i=a^j$) and 
$\obar{\kappa}^i=\obar{\kappa}^j$, that  
$\theta^{i}_0$, $\theta^{i}_1$, $\nu^{i}<j$,
and finally for all $\eta \in a^i$
\[
\tup{T^{i},g^{\eta,i}, D^{i},\Delta^{i},\otp({\mathcal t}^{i}),X^{\eta,i}}=
 \tup{T^{j},g^{\eta,j},D^{j},\Delta^{j},\otp({\mathcal t}^{j}),X^{\eta,j}}. 
\]

We prove a series of lemmas which together describe the common parts of
$p^i$ and $p^j$. Let $A$ be the common value of $A^i$ and $A^j$ and $a$ that of $a^i$ and $a^j$.

\begin{lemma}\label{intersection} 
(a) ${({\mathcal t}^{j}\setminus {\mathcal t}^{i})\cap\setof{\alpha_\gamma}{i\le \gamma<j} = \emptyset}$, and 
(b) ${\mathcal t}^{i}\cap {\mathcal t}^{j}\subseteq \setof{\alpha_\gamma}{\gamma<i}$.
\end{lemma}

\begin{proof}[\ref{intersection}] Suppose $\alpha_\gamma\in {\mathcal t}^{j}$. If $\gamma<j$ then $\gamma\in T^{j}$. 
But $T^{j}=T^{i}$, so $\gamma\in T^{i}$. Hence $\gamma<i$ and $\alpha_\gamma\in {\mathcal t}^{i}$, 
proving (a). If $\alpha_\gamma\in {\mathcal t}^{i}$ then $\gamma<\theta^{i}_0<j$. (For the definition 
of $\theta^i_0$ immediately gives that $\gamma<\theta^i_0$; and since $i$, $j\in \widetilde{C}$ and $i<j$
one has that $\theta^i_0<j$.) Thus if $\alpha_\gamma\in {\mathcal t}^{i}\cap {\mathcal t}^j$ we have $\gamma<i$ by (a). So (b) holds.
\end{proof}

\begin{lemma}\label{intersectiond} ${\mathcal d}^{i}\cap {\mathcal d}^{j}\subseteq i$. 
\end{lemma}

\begin{proof}[\ref{intersectiond}] If  $\zeta\in {\mathcal d}^{i}\cap {\mathcal d}^{j}$ then 
$\zeta<\theta^{i}_1<j$. So $\zeta\in {\mathcal d}^{j}\cap j=D^{j}=D^{i}={\mathcal d}^{i}\cap i$. $\vphantom{W}$
\end{proof}

\begin{lemma}\label{intersectionf}
If $\eta\in a$ and $\alpha\in {\mathcal t}^{i}\cap {\mathcal t}^{j}$ then 
$f^{\eta,i}_\alpha\on {\mathcal d}^{i}\cap {\mathcal d}^{j}= f^{\eta,j}_\alpha\on {\mathcal d}^{i}\cap {\mathcal d}^{j}$. 
\end{lemma}

\begin{proof}[\ref{intersectionf}]
Let $\eta\in a$ and $\alpha\in {\mathcal t}^{i}\cap {\mathcal t}^{j}$.
By Lemma (\ref{intersection}), if $\alpha=\alpha_\gamma$ then $\gamma<i$.
By Lemma (\ref{intersectiond}), ${\mathcal d}^{i}\cap {\mathcal d}^{j}\subseteq i$, and hence
${\mathcal d}^{i}\cap {\mathcal d}^{j}\subseteq i\cap {\mathcal d}^i = D^i=D^j$, the first equality 
being the definition of $D^i$ and the second holding since $h(i)=h(j)$.

As unpacked in the paragraph starting ``Recall \dots'' immediately after the bullet-pointed table of definitions above,
$D^i = \dom(g^{\eta,i}(\ssss . \ssss,\gamma))$ and $D^j =\dom(g^{\eta,j}(\ssss . \ssss,\gamma))$, and hence
$\dom(g^{\eta,i}(\ssss . \ssss,\gamma)) = \dom(g^{\eta,j}(\ssss . \ssss,\gamma))$.

As $\alpha\in {\mathcal t}^{i}\cap {\mathcal t}^{j}$, if $\zeta\in {\mathcal d}^{i}\cap {\mathcal d}^{j}$ we have 
${\mathcal e}_\zeta(\pi_0 ( f^{\eta,i}_\alpha(\zeta))) = {\mathcal e}_\zeta(\pi_0 ( f^{\eta,j}_\alpha(\zeta)))$ and, 
as ${\mathcal e}_\zeta$ is a bijection, $\pi_0 ( f^{\eta,i}_\alpha(\zeta)) = \pi_0 ( f^{\eta,j}_\alpha(\zeta)) = b_\alpha\on\zeta$. 
Hence we have the claimed agreement.
\end{proof}

Recall that $f^i + f^j$ is the unique function $f$ such that $\dom(f) = \dom(f^i) \cup \dom(f^j)$,
$f \on \dom(f^i) = f^i$ and  $f \on\dom(f^j) = f^j$.

Let $q=(A,B^i\cap B^j,t^i\cup t^j,f^i + f^j)$. 

\begin{lemma}\label{constructionA} 
The quadruple $q$ is a condition in $Q^*$.
\end{lemma}

\begin{proof}[\ref{constructionA}]
Recall that we defined
$A=A^i=A^j$. Lemma (\ref{intersectionf}) shows that $f^i$ and $f^j$ are compatible.
\end{proof}

\begin{lemma}\label{intersectionG} $\Gamma^{i}\cap \Gamma^{j}  \subseteq i$.
\end{lemma}

\begin{proof}[\ref{intersectionG}]
If $\zeta\in \Gamma^{i} \cap \Gamma^{j}$ then $\zeta<\nu^{i}<j$ (the latter since $i$, $j\in \widetilde{C}$ and $i<j$). 
So $\zeta\in \Delta^{j}$.  But $\Delta^{j}=\Delta^{i}$, so $\zeta<i$. 
\end{proof}

\begin{notation}\label{defn-dq-alpha} For $\alpha\in {\mathcal t}^i \cup {\mathcal t}^j$ define
$d^q_\alpha = {\mathcal d}^i\cup {\mathcal d}^j$ if $\alpha \in  {\mathcal t}^i \cap {\mathcal t}^j$,
$d^q_\alpha = {\mathcal d}^i$ if $\alpha \in  {\mathcal t}^i \setminus {\mathcal t}^j$, and 
$d^q_\alpha = {\mathcal d}^j$ if $\alpha \in  {\mathcal t}^j \setminus {\mathcal t}^i$.
\end{notation}

\begin{lemma}\label{characterization-of-dom-fq-eta-alpha}  For 
$\alpha\in {\mathcal t}^i \cup {\mathcal t}^j$ and $\eta\in a$  we have
$d^q_\alpha = \dom(f^{\eta,q}_\alpha)$. \end{lemma}

\begin{proof}[\ref{characterization-of-dom-fq-eta-alpha}] Immediate from the definition of $q$, specifically the definition
of $f^q$ as $f^i + f^j$, the fact that $\dom(f^{\eta,i}_\alpha)={\mathcal d}^i$ for all $\eta\in a^i$ and $\alpha\in {\mathcal t}^i$
-- see the third paragraph of the proof of Proposition (\ref{Qhaschaincond}), 
and the definition by cases of the notation $d^q_\alpha$ just given.
\end{proof}

\begin{lemma}\label{zeta0} Suppose
$\alpha$, $\beta\in {\mathcal t}^{i}\cap {\mathcal t}^{j}$, 
$\zeta$,  $\zeta'\in d^q_\alpha \cap d^q_\beta$, $\zeta<\zeta'$, $\eta\in a$ and 
$f^{\eta,q}_\alpha(\zeta)=f^{\eta,q}_\beta(\zeta)\ne f^{\eta,q}_\alpha(\zeta') = f^{\eta,q}_\beta(\zeta')$.
Then $\zeta$, $\zeta'\in {\mathcal d}^{i}\cap {\mathcal d}^j$.
\end{lemma}

\begin{proof}[\ref{zeta0}] Let $\alpha=\alpha_\gamma$ and $\beta=\alpha_{\gamma'}$. 
Then $\gamma$, $\gamma'<i$ by Lemma (\ref{intersection})
and $\Delta(\alpha,\beta)\in \Gamma^{i}\cap \Gamma^{j}\subseteq i$, by Lemma (\ref{intersectionG}). 
As $b_\alpha\on \zeta' = \pi_0 ( f^{\eta,q}_\alpha(\zeta'))=
\pi_0( f^{\eta,q}_\beta(\zeta'))=b_\beta\on\zeta'$ we have $\zeta'<\Delta(\alpha,\beta)$ and hence
$\zeta$, $\zeta'<i$. 
Thus $\zeta'\in {\mathcal d}^j\cap i \subseteq D^j=D^i\subseteq {\mathcal d}^{i}$.
Similarly $\zeta\in {\mathcal d}^{j}$.
\end{proof}

Finally, we can now show that $q$ satisfies (5) of the definition of $Q$.

\begin{lemma}\label{zetafinal} Suppose
$\alpha$, $\beta\in {\mathcal t}^{i}\cup {\mathcal t}^{j}$, 
$\zeta$,  $\zeta'\in d^q_\alpha \cap d^q_\beta$, $\eta\in a$ and 
$f^{\eta,q}_\alpha(\zeta)=f^{\eta,q}_\beta(\zeta)\ne f^{\eta,q}_\alpha(\zeta') = f^{\eta,q}_\beta(\zeta')$
and $y$ is harmonious with $A^q$ past $\eta$. 
Then $y\concat (w,B^q)\forcesq{\Rad{w}}{ \zeta \name{E}_\alpha \zeta' \Longleftrightarrow \zeta \name{E}_\beta \zeta' }$. 
\end{lemma}

\begin{proof}[\ref{zetafinal}]
If for some $k\in \set{i,j}$ we have $\alpha$, $\beta\in {\mathcal t}^k$ 
then, using Lemma (\ref{zeta0}) if $\alpha$, $\beta\in {\mathcal t}^i\cap {\mathcal t}^j$, 
$\zeta$, $\zeta'\in {\mathcal d}^k$. As $p^k\in Q$ we have
$y\concat (w,B^k)\forcesq{}{\zeta \name{E}_\alpha \zeta' 
\Longleftrightarrow \zeta \name{E}_\beta \zeta'}$. But $B^q=B^i\cap B^j\subseteq B^k$, so
$y\concat (w,B^q) \le y\concat (w,B^k)$ and
$y\concat(w,B^q)\forcesq{}{\zeta \name{E}_\alpha \zeta' 
\Longleftrightarrow \zeta \name{E}_\beta \zeta' }$.

Otherwise, either we have
$\alpha\in {\mathcal t}^i\setminus {\mathcal t}^j$, 
$\beta \in {\mathcal t}^j\setminus {\mathcal t}^i$
and $\zeta$, $\zeta'\in {\mathcal d}^i\cap {\mathcal d}^j$, or we have the symmetric case
with the roles of $i$ and $j$ exchanged. We treat the former; for the symmetric case
exchange $i$ and $j$ throughout.

Suppose $\beta$ is the $\varepsilon$-th element
of ${\mathcal t}^j$. As $\otp(t^i)=\otp(t^j)$, we can define $\beta'$ to be the 
$\varepsilon$-th element of $t^i$. As $X^{\eta,i}=X^{\eta,j}$ we have that the
`$F^\eta_\varepsilon$' for the tuple that starts
$\tup{y,\varepsilon,\zeta,\zeta',\dots}$ is the same for both $i$ and $j$.
Hence ${\mathcal e}_{\zeta'} (\pi_0 ( f^{\eta,i}_{\beta'}(\zeta')))={\mathcal e}_{\zeta'}( \pi_0( f^{\eta,j}_{\beta}(\zeta')))$.
As ${\mathcal e}_{\zeta'}$ is a bijection and $\pi_0$ is projection onto the first co-ordinate, this gives
$b_\beta\on\zeta'=b_{\beta'}\on\zeta'$.
Let $x = b_\beta\on\zeta'=b_{\beta'}\on\zeta'$, 
$f^{\eta,i}_{\beta'}(\zeta)=f^{\eta,j}_{\beta}(\zeta) = (x\on\zeta, \delta)$,
$f^{\eta,i}_{\beta'}(\zeta')=f^{\eta,j}_{\beta}(\zeta') = (x, \delta')$,
and, finally, $z = (\eta, y, x , \zeta, \zeta', \delta, \delta')$.

Then $z \in Z^{p^i} \cap Z^{p^j}$, and $\dot\sigma^i_z = \dot\sigma^j_z =\dot \sigma$
say, where $\dot \sigma$ is an $\Rad{y}$-name for a truth value. 
Since $B^q = B^i \cap B^j$,  $y \concat (w, B^q)$ simultaneously reduces the truth
values of the statements ``$\zeta \sss\name{E}_\beta \sss\zeta'\ssss$'' and 
``$\zeta \sss\name{E}_{\beta'} \sss\zeta'\ssss$'' to $\dot \sigma$, so
$y\concat (w,B^q)\forces  \hbox{``}\sss\zeta\sss \name{E}_{\beta'}\sss \zeta'
\Longleftrightarrow 
\zeta\sss \name{E}_{\beta}\sss \zeta'\sss\hbox{''}$. If $\beta'=\alpha$ we are done. Otherwise,
as $p^i\in Q$ we also have
that $y\concat (w,B^i)\forces \hbox{``} \sss\zeta\sss \name{E}_{\alpha}\sss \zeta'
\Longleftrightarrow \zeta\sss \name{E}_{\beta'}\sss \zeta'\sss\hbox{''}$, and 
hence $y\concat (w,B^q)\forces \hbox{``} \sss\zeta\sss \name{E}_{\alpha}\sss \zeta'
\Longleftrightarrow \zeta\sss \name{E}_{\beta}\sss \zeta'\sss\hbox{''}$, as required.
\end{proof}

Lemma (\ref{constructionA}) showed $q\in Q^*$ and Lemma (\ref{zetafinal}) shows that in fact $q\in Q$ as well. 
Since by construction $q\le p^i$, $p^j$, and both $p^i\le p'^{i}$ and $p^j\le p'^{j}$, we have shown
that there is a suitable function $h$ such that
if $i$, $j\in \widetilde{C}\cap S^{\kappa^+}_{\kappa}$, $i<j$ and $h(i)=h(j)$ then 
$p'^{i}$ and $p'^{j}$ are compatible. Hence $Q$ has the $\kappa^+$-stationary chain condition.

\end{proof}

\section{The main iteration}\label{maini}

Let $V$ be a model in which $\kappa$ is supercompact,
$\lambda$  is a regular cardinal less than $\kappa$, and
$\chi$ is a successor cardinal with predecessor $\chi^-$ such that 
$\cf(\chi^-)\ge\kappa^{++}$. Suppose also that GCH holds in $V$
and hence, by \cite{Shelah-diamonds}, $\diamondsuit_\chi(S^\chi_{\kappa^+})$ holds.

Let $\Lbb$ be the Laver iteration (\cite{Laver}), as defined in $V$, making
the supercompactness of $\kappa$ indestructible under $\lower.05em\hbox{$<$}\kappa$-directed closed forcing.

Let $\name{e}$ be a canonical $\Lbb$-name such that 
$\forcesq{\Lbb}{\name{e}:\kappa\longrightarrow V_\kappa\hbox{ is a bijection.}}$

Using Theorem (\ref{double-star-and-statcc-itn-thm}),
Lemma (\ref{sh-conds-imply-double-star}), Proposition (\ref{dnotes}), Lemma (\ref{kwcompact-and-glbs}) and 
Proposition (\ref{Qhaschaincond}), we will define an
$\Lbb$-name $\name{\Pbb}_\chi$ so that 
$\forces_{\Lbb} \hbox{``}\name{\Pbb}_\chi$ is an iteration of length $\chi$ 
consisting of $\name{\Qbb}_0$, a name for the usual forcing to add a 
$\kappa^+$-Kurepa tree with $\chi^-$-many branches followed by a
$\lk$-support iteration of $\lk$-directed closed, countably
parallel-closed  $\kappa^+$-stationary cc forcings each of size less
than $\chi$.''  Thus the iteration will add $\chi$-many subsets of
$\kappa$, but at each intermediate stage $2^\kappa = \chi^-$. 
The constituents of the iteration will depend on a
fixed sequence $\tupof{\name{x}_\alpha}{\alpha < \chi^-}$ of 
$\Lbb * \name{\Qbb}_0$-names for distinct branches through the 
$\kappa^+$-Kurepa tree.   As we
define the iteration we also build an enumeration of 
$\Lbb * \name{\Pbb}_\chi$ as $\tupof{p_\xi}{\xi<\chi}$. As $\chi$ is a regular
cardinal and since $\forcesq{\Lbb}{\hbox{the iterands are of size less than }\chi}$ 
there will be a club set relative to $S^\chi_{\ge\kappa}$ of $\xi<\chi$ 
such that $\Lbb * \name{\Pbb}_\xi = \setof{p_\varepsilon}{\varepsilon<\xi}$. 
(Formally, such that for all $\varepsilon<\xi$ we have $\supp(p_\varepsilon)\subseteq \xi$ and 
$\Lbb * \name{\Pbb}_\xi = \setof{p_\varepsilon\on\xi}{\varepsilon<\xi}$. See the proof of Proposition (\ref{dnotes}).)
We simultaneously inductively define $\Lbb *\name{\Pbb}_\xi$-names $\name{S}_\xi$ 
as in \S\ref{preservation} and derive $\Lbb *\name{\Pbb}_\xi$-names $\name{u}^\xi$.

As each $\name{S}_\xi$ is a canonical name for a subset of $\xi$, when
$\xi=\kappa \lambda \xi$ (ordinal multiplication) we can easily
convert it into a name for a set of order type $\lambda$ of sequences of
$\xi$ many subsets of $\kappa$. In order to do this, for each
$\tau<\lambda$, set $\name{u}^\xi_{1+\tau}$ to be the name derived
from $\name{S}_\xi$ for 
$\setof{\name{e}``\setof{\eta<\kappa}{\kappa\lambda\varepsilon + \kappa\tau+\eta 
\in \name{S}_\xi\cap \widehat{[\kappa\lambda \varepsilon+\kappa\tau,\kappa\lambda
\varepsilon+\kappa(\tau+1) ) }}}{\varepsilon<\xi}$ and set
$\name{u}^\xi=\hat{\kappa}\concat\tupof{\name{u}^\xi_{1+\tau}}{1+\tau<\lambda}$.

We use the sequence $\name{u}^\xi$ to help define the next stage in
the iteration. Let $\name{\mathcal U}^\xi$ be a $\Lbb * \name{\Pbb}_\xi$-name for
the class of all ultrafilter sequences. (As for each $\xi<\chi$ we have  
$\forcesq{\Lbb}{\name{\Pbb}_\xi\hbox{ is }<\kern-2.5pt\kappa\hbox{-directed closed}}$,
we will have for all $\xi<\xi'\le\chi$ that
$\forces_{\Lbb * \Pbb_{\xi'}}\name{\mathcal U}^\xi_\kappa = \name{\mathcal U}^{\xi'}_\kappa $.)

If $\cf(\xi) = \kappa^+$, $\Lbb*\name{\Pbb}_\xi=\setof{p_\varepsilon}{\varepsilon<\xi}$, 
$\xi=\kappa \lambda \xi$ and $\forces_{\Lbb * \Pbb_\xi} \hbox{``}\name{u}^\xi \in \name{{\mathcal U}}^\xi$ 
and $\name{\kappa}_{\name{u}^\xi}=\kappa$''  let 
$\tupof{\name{E}_\alpha^\xi}{\alpha<\chi^-}$ enumerate the canonical 
$\Lbb * \name{\Pbb}_\xi * \nameRad{\name{u}^\xi}$-names for graphs on $\kappa^+$ and
let $\name{\Qbb}_\xi = \name{Q}(\name{u}^\xi)$, where
$\name{Q}(\name{u}^\xi)$ is an $\Lbb * \name{\Pbb}_\xi$-name for the forcing
defined from ${\name{u}}^\xi$,  $\tupof{\name{x}_\alpha}{\alpha < {\chi^-}}$ 
and $\tupof{{\name{E}}_\alpha^\xi}{\alpha<{\chi^-}}$ as 
in \S\ref{Pm}. Otherwise let $\name{\Qbb}_\xi$ name trivial forcing.

Fix $G$ which is $\Lbb$-generic over $V$ and 
$H$ which is $\name{\Pbb}^G_\chi$-generic over $V[G]$. For $\xi<\chi$ let $H_\xi$ be the restriction 
of $H$ to $\name{\Pbb}^G_\xi$,
let $\Qbb_\xi = \name{\Qbb}_\xi^{G * H_\xi}$, 
and let $u^\xi=(\name{u}^\xi)^{G * H_\xi}$. 

If $V[G][H_\xi]\thinks u^\xi\in {\mathcal U}$, $\kappa_{u^\xi}= \kappa$ and $\lh(u^\xi)=\lambda$,
let 
\begin{itemize}
\item $K_\xi$ be the $Q(u^\xi)$-generic over $V[G][H_\xi]$ induced by $H$,  
\item $A^\xi=\tupof{A_\rho}{\rho<\kappa \sss\sss\&\sss\sss \exists p\in K_\xi \sss\sss A_\rho = (A_\rho)^p}$,
\item and $a^\xi = \bigcup \setof{ a^p }{ p \in K_\xi }$.
\end{itemize}

Fix an enumeration $\tupof{\name{D}_\xi}{\xi<\chi}$ of the 
$\Lbb *\name{\Pbb}_\chi$-names for subsets of $V_\kappa$ such that each
$\name{D}_\xi$ is a $\Lbb *\name{\Pbb}_\xi$-name.

Let $j:V\longrightarrow M$ witness that $\kappa$ is $2^\chi$-supercompact, 
such that in the iteration $j(\Lbb)$ we force with $\name{\Pbb}_\chi$ at stage $\kappa$ and then do
trivial forcing at all stages between $\kappa$ and $(2^\chi)^+$. 
Let $j(\Lbb)= \Lbb * \name{\Pbb}_\chi * \name{{\mathcal L}}$, where $\name{{\mathcal L}}$ is a $\Lbb * \name{\Pbb}_\chi$-name
and we note that by the choice of $j$ we have
$\forcesq{\Lbb * \name{\Pbb}_\chi}{\name{{\mathcal L}} \hbox{ is $\lower0.05em\hbox{$<$}(2^\chi)^+$-closed.}}$

We now carry out, in $V[G][H]$, an inductive construction of length $\chi$ in which we build a chain of 
conditions $(r_\xi,\name{q}_\xi)\in \name{\mathcal L}^{G * H} * j(\name{\Pbb}_\chi)$ for $\xi<\chi$. 
Note that forcing with $\name{\mathcal L}^{G * H}$ over $V[G][H]$ always adds a generic embedding 
$j:V[G]\longrightarrow M[j(G)]$ -- see \cite[\S{}9]{Cummings} -- so that the name $j(\name{\Pbb}_\chi)$ is well defined.
The construction includes arranging for each $\xi<\chi$ that 
$\forcesq{\name{\mathcal L}^{G * H}}{\name{q}_\xi\in j(\name{\Pbb}_\xi)}$ 
and that for all $p\in H_\xi$ we have that $r_\xi\forcesq{\name{\mathcal L}^{G * H}}{\name{q}_\xi \le j(p).}$

Define $\mathfrak{U}=(\setof{v\in {\mathcal U}}{\kappa_v=\kappa \sss\&\sss \lh(v)\le\lambda})^{V[G][H]}$. 
As $V[G][H]\thinks \card{2^\kappa} = \chi$, $\chi$ is a successor cardinal and $\lambda<\kappa$,
we have that $V[G][H]\thinks \card{\mathfrak{U}}\le ( (2^{2^\kappa})^\lambda) = 2^\chi$. 

\emph{Inductive case:  $\xi=0$ or limit $\xi$.}  By the 
$\lower.05em\hbox{$<$}\chi$-closure of $\name{\mathcal L}^{G * H} * j(\name{\Pbb}_\chi)$,
just choose some $(r_\xi,\name{q}_\xi)$ such that 
$\forcesq{\name{\mathcal L}^{G * H}}{\name{q}_\xi\in j(\name{\Pbb}_\xi)}$,
and for all $\xi'<\xi$ we have $(r_\xi,\name{q}_\xi) \le (r_{\xi'},\name{q}_{\xi'})$.
Then  for each $p\in H_\xi$ we have $r_\xi\forcesq{\name{\mathcal L}^{G * H}}{\name{q}_\xi \le j(p).}$

\emph{Inductive case: $\xi+1$.}  If $\xi = 0$ then $\Qbb_0 $ is the usual forcing to add a $\kappa^+$-Kurepa tree with 
$\chi^-$-many branches as computed in ${V[G]}$,
and we set $r_1$ to be the trivial condition and $\name{q}_1$ to be a name for
the greatest lower bound of the union of the pointwise image of the $\Qbb_0$-generic filter.
We now assume that $\xi > 0$.

If $V[G][H_\xi]\thinks u^\xi\ni {\mathcal U}$ or $u^\xi\in {\mathcal U}$ but $\kappa_{u^\xi}\ne \kappa$ or 
$u^\xi\in {\mathcal U}$ and $\kappa_{u^\xi}= \kappa$ but $\lh(u^\xi)\ne\lambda$ there is nothing to do and we can take 
$(r_{\xi+1},\name{q}_{\xi+1})$ = $(r_\xi,\name{q}_\xi \concat \name{\onebb}_{j(\Qbb_\xi)})$.
So assume otherwise: $V[G][H_\xi]\thinks u^\xi\in {\mathcal U}$, $\kappa_{u^\xi}= \kappa$ and $\lh(u^\xi)=\lambda$. 

By the construction so far
$(r_\xi,\name{q}_\xi)\forces_{\name{\mathcal L}^{G * H} * j(\name{\Pbb}_\xi)}$ ``\sss there is a lifting 
$j:V[G][H_\xi]\longrightarrow M[j(G)][j(H_\xi)]$ of $j$.''

As $2^\chi$ is less than the closure of $\name{\mathcal L}^{G * H} * j(\name{\Pbb}_\chi)$ 
we may, by shrinking if necessary, assume that there is some 
$(r'_\xi,\name{q}'_\xi)\le (r_\xi,\name{q}_\xi)$  such that
$(r'_\xi,\name{q}'_\xi)\decides ``v\in j(\name{D}_\xi)\hbox{''}$ for every $v \in \mathfrak{U}$.

\begin{definition} If $v\in {\mathcal U}^{V[G][H]}$, $\lh(v)=\lambda$ and for all $\tau<\lambda$ we have
$v_\tau \cap V[G][H_\xi] = u^\xi_\tau$ we say $v$ \emph{fills out} $u^\xi$.
\end{definition}

Note that if $v$ \emph{fills out} $u^\xi$  then, \emph{a priori}, $v\in {\mathfrak U}$.
\vskip12pt

\begin{lemma}\label{master}  Suppose there is a lifting 
$j:V[G][H_\xi]\longrightarrow M[j(G)][j(H_\xi)]$ of $j$.
Set $\obar{B}=\bigcap\setof{j(B^q)}{(A^q,B^q,t^q,f^q)\in K_\xi}$.  
Suppose there is some $v\in {\mathcal U}^{V[G][H]}$ which fills out $u^\xi$
and for all $\tau<\lambda$ we have $v\on\tau\in \obar{B}$.

Then there is a master condition for $j(\Qbb_\xi)$ such that 
on forcing below the master condition there is a lifting of $j$ 
to a map $j:V[G][H_\xi][K_\xi]\longrightarrow M[j(G)][j(H_\xi)][j(K_\xi)]$ such that
for every $v\in {\mathcal U}\cap \obar{B}$ filling out $u^\xi$ and every
$\tau<\lambda$  we have $v\on\tau\in j(A^\xi)_\kappa$.
\end{lemma}

\begin{proof}[\ref{master}]
We construct a suitable master condition $p^* = (A^*, B^*, t^*, f^*)$.
\begin{itemize}
\item  $A^*={A}^\xi \concat \setof{v\on\tau}{\tau<\lambda\sss\sss\&\sss\sss v\in \obar{B}\cap ({\mathcal U}_{\kappa^+}\setminus {\mathcal U}_\kappa )
\sss\sss\&\sss\sss v \hbox{ fills out } u^\xi }$.
\item $B^*=\obar{B} \setminus {\mathcal U}_{\kappa^+}$. 
\item $t^* = \kappa \times j`` \chi^-$.
\item For each $\eta \in a^\xi$ and $\alpha < \chi^-$, $f^{*, \eta}_{j(\alpha)} = \bigcup \setof{j(f^{p, \eta}_\alpha)}{ p \in K_\xi }$.
\end{itemize}
We note that since $f^{p, \eta}_\alpha$ is a partial function of size less than $\kappa$, and we have the density lemmas
Lemma (\ref{jdensity1}) and Corollary (\ref{density-for-t-f-for-Q}), it is easy to see that
$d^{*, \eta}_{j(\alpha)} = \dom(f^{*, \eta}_{j(\alpha)}) = j``\kappa^+$
and $f^{*, \eta}_{j(\alpha)}(j(\zeta)) = j(f^{p, \eta}_\alpha(\zeta))$ for any $p \in K_\xi$ such that
$(\eta, \alpha) \in t^p$.  We also note that $a^{p^*} = a^\xi \cup \set{ \kappa }$. 

It is routine to verify that $p^* \in j(\Qbb^*_\xi)$ and that $p^* \le_{j(\Qbb^*_\xi)} j(q)$ for $q\in K_\xi$. 
We must now show that $p^*$ satisfies $(5)$ in the definition of $j(\Qbb_\xi)$. 

Let $\eta\in a^\xi$, and suppose that
$f^{*, \eta}_{j(\alpha)}(j(\zeta))=f^{*,\eta}_{j(\beta)}(j(\zeta)) \ne f^{*, \eta}_{j(\alpha)}(j(\zeta'))=f^{*,\eta}_{j(\beta)}(j(\zeta'))$
for some $\alpha, \beta \in \chi^-$ and $\zeta, \zeta' \in \kappa^+$. 
Let $y$ be harmonious with $A^*$ past $\eta$. 
If $\kappa_y < \kappa$ let $y'=y$. Otherwise let $y'$, 
$\tau<\lambda$, $v \in B^*$ which fills out $u^\xi$, and ${B}\in {\mathcal F}(v\on\tau)$ 
be such that $y=y' \concat (v \on \tau,{B})$. 

Using Corollary (\ref{jdensity2}) and the definition of $p^*$ we may find  $p\in K_\xi$ with $\rho^p$ a successor ordinal
such that
\begin{itemize}
\item  $\eta \in a^p$ with $\eta < \max(a^p)$,
\item  $\kappa_{y'} < \max(a^p)$,
\item  $(\eta,\alpha)$, $(\eta,\beta) \in t^p$,
\item  $\zeta, \zeta' \in d^{p, \eta}_\alpha \cap d^{p, \eta}_\beta$,
\item  $f^{p, \eta}_\alpha(\zeta)=f^{p,\eta}_\beta(\zeta) \ne f^{p, \eta}_\alpha(\zeta')=f^{p,\eta}_\beta(\zeta')$,
\item  and $y'$ is harmonious with $A^p$ past $\eta$. 
\end{itemize}

If $y=y'$ we have that $y$ is harmonious with $A^p$ past $\eta$ and 
\[
y \concat (j(u^\xi),B^*) \le y \concat (j(u^\xi),j(B^p)) = j(y \concat (u^\xi, B^p)).
\] 
As $p$ is a condition in $\Qbb_\xi$ we have
$y \concat (u^\xi, B^p) \forces \hbox{`` }
\zeta \name{E^\xi_\alpha} \zeta' \longleftrightarrow  \zeta \name{E^\xi_\beta} \zeta'\hbox{ ''}$,
and  we are done by elementarity.

So we may assume that we are in the other case, that is $y=y' \concat (v \on \tau,{B})$
for some $v \in \obar{B}$ which fills out $u^\xi$. We note that $v \restriction \tau \in \obar{B}$,
since each of the sets $B^q$ is closed under taking initial segments.
Let $\rho^* < \rho^p$ be minimal such that $\max (\set{ \eta, \kappa_{y'} }) < \kappa_{\rho^*}$. 

Since $y$ was chosen to be harmonious with 
$A^*$ past $\eta$,  $B^*  \cap V_\kappa = \emptyset$,  and $\eta < \kappa = \kappa_v$,
we see that
${B} \subseteq \bigcup \setof{ A^\xi_\rho}{ \eta < \kappa_\rho, \rho < \kappa }$ and
$\kappa_w > \eta$ for all $w \in {B}$. 
Since $y' \concat (v \restriction \tau, {B})$ is a condition, we also have that
$\kappa_w > \kappa_{y'}$ for all $w \in {B}$. 
So by the choice of $\rho^*$,  
${B} \subseteq \bigcup \setof{ A^\xi_\rho}{ \rho^* \le \rho < \kappa }$.

Since the condition $p \in K_\xi$, $A_\sigma = A_\sigma^p$ for  
$\rho^* \le \sigma < \rho^p$ and $A_\sigma \subseteq B^p$
for $\rho^p \le \sigma < \kappa$. It follows that 
${B} \subseteq \bigcup \setof{A^p_\sigma}{ \rho^* \le \sigma < \rho^p } \cup B^p$.

We now appeal to Lemma (\ref{abacus1}) to obtain the $\rho^*$-weakening of $p$, that is to say
the  condition $q \in \Qbb_\xi$ such that $\rho^q = \rho^*$,  $p \le q$,
$A^q = A^p \restriction \rho^*$,  
$B^q  =  \bigcup \setof{A_\sigma}{ \rho^* \le \sigma < \rho^p } \cup B^p$,
$t^q = \setof{ (\eta', \alpha') \in t^p }{ \eta' \in a^q \cap \sup(a^q) }$ and 
$f^q = \tupof{ f^{p, \eta'}_{\alpha'}}{ (\eta', \alpha') \in t^q }$.

Since $p \le q$ we have $q \in K_\xi$, and so $B^* \subseteq j(B^q)$. 
By the choice of $\rho^*$, since $\rho^q=\rho^*$, $A^q=A^p\on\rho^*$,
$\max(\set{\eta,\kappa_{y'}}) < \kappa_{\rho^*}$ and $y'$ is harmonious with $A^p$ past $\eta$,
we have that $y'$ is harmonious with $A^q$ past $\eta$.  

By clause (3) of the conclusion of Lemma (\ref{abacus1}), the condition $q$ enjoys
a strengthened form of condition (5) in the definition of $Q(w)$, which implies in
this case that 
$y' \concat (u^\xi, B^q) \forcesq{}{ \zeta \name{E^\xi_\alpha} \zeta' \longleftrightarrow  \zeta \name{E^\xi_\beta} \zeta' .}$

We claim that
$y \concat (j(u^\xi),B^*) \le y' \concat (j(u^\xi),j(B^q)) = j(y' \concat (u^\xi, B^q)).$
 
The only non-trivial point is that the pair $(v \restriction \tau, {B})$ can be added to
the condition  $y' \concat (j(u^\xi),j(B^q))$.  This holds because
$v \restriction \tau \in \obar{B} \subseteq j(B^q)$, and
${B} \subseteq \bigcup \setof{A^p_\sigma}{ \rho^* \le \sigma < \rho^p } \cup B^p = B^q = j(B_q) \cap V_\kappa$.
As in the case when $y = y'$,  we are now done by elementarity.
\vskip12pt

Now that we have shown that $p^*$ is a condition in $j(\Qbb_\xi)$ we have that
$p^* \le_{j(\Qbb_\xi)} j(q)$ for $q\in K_\xi$, and so $p^*$ is a master condition.
The last thing to check in order to complete the proof of the lemma is that
$p^*$ is a master condition of the type required. 
However, this is almost immediate from its definition.

For as $j:V[G][H_\xi][K_\xi]\longrightarrow M[j(G)][j(H_\xi)][j(K_\xi)]$ 
is a lifting we have that $p^* \in j(K^\xi)$  (see \cite{Cummings}, Proposition 9.1). We also have,
by definition, that $A^\xi = \bigcap_{q\in K^\xi}  [A^q,B^q]$, and so
$ j(A^\xi)  = j(\bigcap_{q\in K^\xi}  [A^q,B^q] ) = \bigcap_{q\in j(K^\xi)}  [A^q,B^q]. $
Hence if $q\in j(K^\xi)$ and $\rho^q\ge \kappa+1$ then,
by the definition of compatibility under $\le_{\Qbb_\xi}$ and since $K^\xi$ is a filter, 
we must have $A^q_\kappa = A^*_\kappa$.

Consequently, if $v$ fills out $u^\xi$ and $v\on\tau \in \obar{B}$ then by 
the definition of $p^*$ we have $v\on\tau \in A^*_\kappa$ and hence, 
by the previous paragraph, $v\on\tau \in j(A^\xi)$.
\end{proof}

Having established the preceding lemma, we can now complete the inductive step.

\emph{Case (i).} There is some $(r,\name{q})\le (r'_\xi, \name{q}'_\xi)$ 
with $\forces_{\name{\mathcal L}^{G * H}} \name{q}\in j(\name{\Pbb}_\xi)$ 
such that 
\begin{multline*}
(r,\name{q}) \forces_{\name{\mathcal L}^{G * H} * j(\name{\Pbb}_\xi)}
\hbox{``} 
\exists v\in {\mathcal U} \sss(\sss v\hbox{ fills out } u^\xi\sss\sss\&\sss\sss \forall \tau <\lambda\\
v\on\tau \in \obar{B} = {\textstyle\bigcap}\setof{j(B^q)}{(A^q,B^q,t^q,f^q)\in K_\xi }\sss)
\hbox{''}.
\end{multline*}

Let $r_{\xi+1}=r$ and use Lemma (\ref{master}) to choose $\name{q}_{\xi+1}$ such that\break
$\name{q}_{\xi+1}\on \xi = \name{q}$, $(r_{\xi+1},\name{q}_{\xi+1}\on \xi)\forces$ ``$\name{q}_{\xi+1}(\xi)$ is a lower 
bound for $j``K_\xi$'' and $(r_{\xi+1},\name{q}_{\xi+1})\forces$ `` for every
$v\in {\mathcal U}\cap \obar{B} $ which fills out $u^\xi$ and every $\tau<\lambda$ we have $v\on\tau\in j(A^\xi)$.''

\emph{Case (ii).} Otherwise.  Again, by the closure of 
$\name{\mathcal L}^{G * H} * j(\name{\Pbb}_\chi)$,
let $(r_{\xi+1},\name{q}_{\xi+1})\le (r'_\xi,\name{q}'_\xi)$ be such that
$(r_{\xi+1},\name{q}_{\xi+1}\on \xi)\forces$ ``$\name{q}_{\xi+1}(\xi)$ is a lower 
bound for $j``K_{\xi}$'' since $Q_{\xi}$ is trivial.
$\endof_{\hbox{\sevenrm Inductive construction}}$

When the construction is complete use the $\chi$-closure of 
${\name{\mathcal L}^{G * H} * j(\name{\Pbb}_\chi)}$ again and take a lower bound $(r^*,\name{q}^*)$ for
$\tupof{(r_{\xi},\name{q}_{\xi})}{\xi<\chi}$ such that for all $q\in H$ we have ${r^*\forces \name{q}^*\le j(q)}$.
Thus $(r^*,\name{q}^*)$ forces that $j$ can be lifted to some $j:V[G][H]\longrightarrow M[j(G)][j(H)]$ with 
$j(G)=G * H * j(G)/ (G * H)$.

\begin{claim}\label{masterclaim}
If we generate $u=\tupof{u_\tau}{\tau<\lambda}$ from such a lifting of 
$j$ to $V[G][H]$ in the usual inductive way, 
with $u_0=\kappa$ and by setting ${u_{\tau}=\setof{D\in V[G][H]}{u\on\tau \in j(D)}}$ for $0<\tau<\lambda$, 
then in fact we have $u\in V[G][H]$ (and not merely $u\in V[G][H][j(G)/(G* H)][j(H)]$) and $u\in {\mathcal U}^{V[G][H]}$.
\end{claim}

\begin{proof}[\ref{masterclaim}]

The argument is an inductive repetition of a typical one in the context of Laver forcing.

By induction on $\tau<\lambda$ suppose that $u\on\tau\in V[G][H]$. \emph{A priori}, 
$u_\tau\in V[G][H][j(G)/(G * H)][j(H)]$.
However $j(G)/(G * H) * j(H)$ is generic for a highly closed forcing -- it is certainly
$(2^{\kappa})^+$-closed. So in fact $u_\tau\in V[G][H]$. As $V[G][H]$ is closed under sequences of length less than or 
equal to $\lambda$ we also have $u\on\tau+1\in V[G][H]$. Similarly, we inductively obtain that $u\on\tau\in V[G][H]$ for
limit $\tau$, and at the end of the induction that $u\in V[G][H]$.

We also need that normality and the two coherence conditions 
from Definition (\ref{ultrafilterseq}) hold for $u_\tau$. 

Suppose $f:{\mathcal U}_\kappa\longrightarrow V_\kappa$ is a function in $V[G][H]$. Observe that 
 $V_\kappa^{V[G][H][j(G)/(G * H)][j(H)]} = V_\kappa^{V[G][H]}$. 
Then $\setof{w\in {\mathcal U}_\kappa}{f(w)\in V_{\kappa_w}}\in u_{\tau}$ if and only if
$j(f)(u\on\tau) \in V_\kappa^{V[G][H]}$, if and only if there is some
$x\in V_\kappa^{V[G][H]}$ such that $j(f)(u\on\tau) = x$,
if and only if there is some $x\in V_\kappa^{V[G][H]}$ such that 
$\setof{w\in {\mathcal U}_\kappa}{f(w)=x}\in u_\tau$. Hence $u$ satisfies normality in $V[G][H]$.

Similarly, suppose $f: {\mathcal U}_\kappa\longrightarrow \kappa$. Then
$\setof{w\in {\mathcal U}_\kappa}{f(w) < \lh(w)}\in u_\tau$ if and only if 
$j(f)(u\on\tau) < \tau$, if and only if there is some $\sigma<\tau$ such that
$j(f)(u\on\tau) =\sigma$. For this $\sigma$ we then have for $X\subseteq V_\kappa$ that
$X\in u_\sigma$ if and only if $u\on\sigma \in j(X)$ if and only if $j(X)\cap V_\kappa\in u_{j(f)(u\on\tau)}$
if and only if $\setof{w\in {\mathcal U}_\kappa}{X \cap V_{\kappa_w} \in w_{f(w)}} \in u_\tau$.

Finally,  if $\sigma < \tau$ and $X \in u_\sigma$ then $u\on \sigma \in j(X)$ and so 
$\setof{w\in {\mathcal U}_\kappa}{\exists \bar\sigma<\lh(w)\sss X \cap V_{\kappa_w} 
\in w_{\obar{\scriptstyle \sigma}}} \in u_\tau$. So $u$ also satisfies the two coherence conditions.
 \end{proof}

\begin{claim}\label{stationary}
In $V[G][H]$ there is a stationary set $S \subseteq S^\chi_{\kappa^+}$ such that
\begin{itemize}
\item  For every $\xi \in S$ one has $\lh(u^\xi) = \lambda$
and for every $\tau < \lambda$ that $u_\tau\cap V[G][H_\xi] = u^\xi_\tau$.
\item  For every $\xi$ in the closure of $S$, $S \cap \xi \in V[G][H_\xi]$.
\item  For every $\xi \in S$, $u^\xi \in {\mathcal U}^{V[G][H_\xi]}_{\kappa^+}$.
\item  For every $\xi \in S$, 
$\setof{v}{\exists \rho < \kappa \sss\sss v\in A^\xi_\rho} \in {\mathcal F}(u)$.
\end{itemize}
\end{claim}

\begin{proof}[\ref{stationary}]

We start by outlining the motivation for the first step in the proof. 
As $S^{\chi}_{\kappa^+}$ is a stationary subset of $S^{\chi}_{\ge\kappa}$,
Proposition (\ref{dnotes}) tells us that there is a 
$\diamondsuit_\chi(S^{\chi}_{\kappa^+})$-sequence in $V[G][H]$,
(with the properties given in Proposition (\ref{dnotes} (1) and (2))), which
predicts subsets of $\chi$. The statement of the claim talks about predicting 
sequences of measures. So in order to apply Proposition (\ref{dnotes}) we must
`code' such sequences of measures by subsets of $\kappa$.

Working in $V[G][H]$,  define a set $T\subseteq \chi$ `coding' $u$ by 
enumerating each $u_\tau$ in order type $\chi$ as $\tupof{y_{\tau\xi}}{\xi<\chi}$ 
and setting $T\cap [\kappa \tau \xi,\kappa \tau(\xi+1))= 
\setof{\kappa\tau\xi+\eta}{\eta\in y_{\tau\xi}}$.
Appealing to Proposition (\ref{dnotes}), we may find a stationary set
$S \subseteq S^\chi_{\kappa^+}$ such that $S$ satisfies the first two clauses
of the claim.

For  $\xi \in S$ and $0 < \tau < \lambda$, we have
that $u^\xi_\tau$ is a measure on $V_{\kappa}$
in $V[G][H_\xi]$.
Since 
${\mathcal U}_\kappa^{V[G]}={\mathcal U}_\kappa^{V[G][H]}$ it is clear that 
each $u^\xi_\tau$ concentrates on ${\mathcal U}_\kappa$.

\emph{Normality.} Suppose $f: {\mathcal U}_\kappa\longrightarrow V_\kappa^{V[G][H_\xi]}$ 
and $\setof{w\in {\mathcal U}_\kappa}{f(w)\in V_{\kappa_w}}\in u^\xi_{\tau}$.
As $u^\xi_\tau\subseteq u_\tau$ we can apply normality for $u$ to get some
$x\in V_\kappa$ such that $\setof{w\in {\mathcal U}_\kappa}{f(w)=x}\in u_\tau$. But $x$, $f\in V[G][H_\xi]$, 
hence $\setof{w\in {\mathcal U}_\kappa}{f(w)=x}\in u_\tau \cap V_\kappa^{V[G][H_\xi]} = u^\xi_\tau$.

\emph{Coherence (i).} Suppose $f: {\mathcal U}_\kappa\longrightarrow \kappa$ with 
$f\in V[G][H_\xi]$ and $\setof{w\in {\mathcal U}_\kappa}{f(w) < \lh(w)}\in u^\xi_\tau$.
Again as $u^\xi_\tau\subseteq u_\tau$ we can apply the first coherence condition 
for $u$ to get some $\sigma<\tau$ with $x\in u_\sigma$ if and only if 
$\setof{w\in U_\kappa}{x\cap V_{\kappa_w} \in w_{f(w)}} \in u_\tau$. 
If $x\in V[G][H_\xi]$, so that $x\in u^\xi_\sigma$, 
then, recalling that $f\in V[G][H_\xi]$, we have 
$\setof{w\in U_\kappa}{x\cap V_{\kappa_w} \in w_{f(w)}} \in V[G][H_\xi]$, and hence
$\setof{w\in U_\kappa}{x\cap V_{\kappa_w} \in w_{f(w)}} \in u^\xi_\tau$. Conversely,
if $\setof{w\in U_\kappa}{x\cap V_{\kappa_w} \in w_{f(w)}} \in u^\xi_\tau$ then, as
$u^\xi_\tau\subseteq u_\tau$, we have $x\in u_\sigma$.

\emph{Coherence (ii).} If $\tau'<\tau$ and $x\in u^\xi_{\tau'}$ then 
$\setof{w\in {\mathcal U}_\kappa}{\exists \sigma<\lh(w)\sss x\cap 
V_{\kappa_w} \in w_\sigma} \in u_\tau$
and the set is clearly in $V[G][H_\xi]$ as $x$ is, and hence is an element of $u^\xi_\tau$.

Finally, at each $\xi \in S$ when we did the inductive construction of
$(r_{\xi+1},\name{q}_{\xi+1})$ we must have been in `Case (i)' and
used Lemma (\ref{master}) because $(r^*,\name{q}^*\on j(\xi))$
would be an appropriate witness.  Consequently for all $\tau<\lambda$
we have that $u\on\tau \in j(A^\xi)$, and hence $A^\xi\in {\mathcal F}(u)$.
\end{proof}

\begin{proposition}\label{measureseq} Let $S\in V[G][H]$ be as given by 
Claim (\ref{stationary}). Let $\xi$ be a limit of elements of the set 
$S$ of cofinality at least $\kappa^+$. For each $\tau<\lambda$ 
let $v_\tau=\bigcup\setof{u_\tau^\varepsilon}{\varepsilon\in S\cap \xi}$. Then $V[G][H_\xi]\thinks v$ is an ultrafilter sequence
and $\forall\varepsilon\in S\cap \xi\sss\sss A^\varepsilon \in {\mathcal F}(v)$ .
\end{proposition}

\begin{proof}[\ref{measureseq}]  By Claim (\ref{stationary}) we have that $S\cap \xi \in V[G][H_\xi]$, thus $v\in V[G][H_\xi]$.
It is clear that $v$ is a sequence of measures concentrating on 
${\mathcal U}_\kappa^{V[G][H_\xi]}$. In order to see that the normality and coherence 
conditions hold it is enough to observe that since $\cf(\xi)\ge\kappa^+$ we have that 
for each $z\in {\mathcal P}(\kappa)$ and $f:{\mathcal U}_\kappa\longrightarrow V_\kappa$ 
if $z$, $f\in V[G][H_\xi]$ there is some $\varepsilon\in S\cap \xi$ such that $z$, $f\in V[G][H_\varepsilon]$.
Again by Claim (\ref{stationary}) we have that $A^\varepsilon\in {\mathcal F}(v)$.
\end{proof}

\section{Proof that we do get small universal families}\label{small}

Let $S\in V[G][H]$ be as given by Claim (\ref{stationary}).
As in Proposition (\ref{measureseq}),
we choose $\xi$ a limit point of $S$ with $\cf(\xi) = \kappa^{++}$.
Let $G * H_\xi$ be the $\Lbb * \name{\Pbb}_\xi$-generic filter over $V$ induced by $G * H$,
and define an ultrafilter sequence $v \in V[G * H_\xi]$ by setting
$v_\tau=\bigcup\setof{u_\tau^\varepsilon}{\varepsilon\in S\cap \xi}$ for $0 < \tau < \lambda$. 
Let $g^*$ be $\Rad{v}$-generic over $V[G * H_\xi]$.
As we discussed in \S\ref{Radinmat}, by forcing below a suitable condition we may arrange that
\begin{itemize}
\item the generic object induced by $g^*$ is a $\lambda$-sequence $\tupof{u_i}{i < \lambda}$ of ultrafilter sequences,
\item defining $\kappa_i = \kappa_{u_i}$ for $i < \lambda$, the set
$C = \setof{\kappa_i}{i < \lambda}$ is a club subset of $\kappa$, 
\item $\min(C) > \lambda$, and so $\cf(\kappa) = \lambda$ in $V[G * H_\xi * g^*]$.  
\end{itemize}

\begin{notation}\label{notn-for-objects-with-index-varepsilon}
For $\varepsilon\in S\cap \xi$ set $H_\varepsilon$ to be the induced $\Pbb_\varepsilon$-generic filter over $V[G]$,
$Q_\varepsilon= \name{\Qbb}_\varepsilon^{G * H_\varepsilon}$, 
$K_\varepsilon$ to be the $Q_\varepsilon$ generic filter over $V[G][H_\varepsilon]$ induced by $G * H_\xi$,
${
A^{*\varepsilon}= \tupof{A_\rho}{\rho<\kappa \sss\sss\&\sss\sss
\exists p\in K_\varepsilon \sss ( \rho<\rho^{\ssss p} \sss\sss\&\sss\sss A_\rho=A^p_\rho)},
}$
$\kappa^\varepsilon_\rho$ to be the common value of $\kappa_w$ for $w \in A^{* \varepsilon}_\rho$,  
and 
${a^{*\varepsilon}=\setof{\kappa^\varepsilon_\rho}{\rho < \kappa}}$.

As per Observation(\ref{Radin-generics-go-down}), 
the characterisation  of genericity for Radin forcing implies that for every 
$\varepsilon \in S \cap \xi$,
the sequence $\tupof{u_i}{i < \lambda}$  is $\Rad{u^\varepsilon}$-generic  over $V[G][H_\varepsilon]$.
Let $g^{\varepsilon}$ be the $\Rad{u^\varepsilon}$-generic filter over $V[G][H_\varepsilon]$ induced by this sequence,
so that easily $g^{\varepsilon} = g^* \cap \Rad{u^\varepsilon}$. 
\end{notation}

Note that, by the characterisation of Radin-genericity from Theorem
(\ref{Mitchell-characterization}), for each $\varepsilon\in S\cap\xi$, 
since $A^\varepsilon \in {\mathcal F}(v)$ by Proposition (\ref{measureseq}),
we have that $u_j \in \bigcup \setof{A^{*\varepsilon}_\rho}{\rho<\kappa}$ for all large 
$j < \lambda$. Let $i_\varepsilon < \lambda$ be the least successor ordinal
such that $u_j \in \bigcup \setof{A^{*\varepsilon}_\rho}{\rho<\kappa}$ for 
$j \ge i_\varepsilon$, and let $\eta_\varepsilon = \kappa_{u_{i_\varepsilon}}$. 
Since $\eta_\varepsilon$ is a successor point of the generic club $C$,
the sequence  $u_{i_\varepsilon} = \tup{ \eta_\varepsilon }$. 

We note that since $u_{i_\varepsilon} = \tup{ \eta_\varepsilon } \in
\bigcup \setof{A^{*\varepsilon}_\rho}{\rho<\kappa}$, we may define
$\sigma_\varepsilon$ as the unique $\rho$ such that 
$\tup{ \eta_\varepsilon } \in A^{*\varepsilon}_\rho$, and by definition we have 
$\kappa_{u_{i_\varepsilon}} = \eta_\varepsilon = \kappa^\varepsilon_{\sigma_\varepsilon}$.  
It follows that for any $q \in K_{\varepsilon}$ with $\sigma_\varepsilon < \rho^q$ 
we have that $\tup{ \eta_\varepsilon } \in A^q_{\sigma_\varepsilon}$ and
$\eta_\varepsilon \in a^q$.

\begin{definition}
For $\varepsilon \in S\cap \xi$
let $\mathfrak{E}^{\sss\varepsilon}_\alpha= (\name{E}^\varepsilon_\alpha)^{G * H_\varepsilon * g^\varepsilon}$, 
and, for $\alpha<\chi^-$, let $\mathfrak{f}^\varepsilon_\alpha = 
\bigcup\setof{(f^{\eta_\varepsilon}_\alpha)^q}{q\in K_\varepsilon \sss\&\sss 
\alpha \in t^{q,\eta_\varepsilon}}.$
\end{definition}

As we proved in Section \ref{Pm} (see Corollary (\ref{density-for-t-f-for-Q})), 
$\mathfrak{f}^\varepsilon_\alpha$ is a function with domain 
$\kappa^+$
such that $\mathfrak{f}^\varepsilon_\alpha (\zeta) \in \set{ b_\alpha \on \zeta } \times \kappa$
for every $\zeta < \kappa^+$. 

\begin{definition}\label{defn-the-graph-cal-E-varepsilon-alpha}
For each $\varepsilon \in S\cap \xi$ and $\alpha<\chi^-$ define 
${\mathcal E}^{\sss\varepsilon}_\alpha$ on $\rge(\mathfrak{f}^\varepsilon_\alpha)$ by
$z \mathrel{{\mathcal E}^{\sss\varepsilon}_\alpha } z'$ if and only if 
$\height(\pi_0(z)) \mathrel{\mathfrak{E}^{\sss\varepsilon}_\alpha} \height(\pi_0(z'))$. That is, 
${\mathcal E}^{\sss\varepsilon}_\alpha = \mathfrak{f}^\varepsilon_\alpha``\mathfrak{E}^{\sss\varepsilon}_\alpha$.
(For if $\zeta<\kappa^+$ and $\mathfrak{f}^\varepsilon_\alpha(\zeta) = z$ then $\height(\pi_0(z))=\zeta$.)
\end{definition}

We now prove a short technical lemma which will allow us to give an equivalent characterization of 
${\mathcal E}^{\sss\varepsilon}_\alpha$, which in turn facilitates the proof that 
${\mathcal E}^{\sss\varepsilon}_\alpha$ and ${\mathcal E}^{\sss\varepsilon}_\beta$ are coherent for $\alpha\neq \beta$.

\begin{lemma}\label{lemma1-2-james} Let $\varepsilon \in S\cap \xi$.
Suppose $y \concat (u^\varepsilon, D) \in g^{\varepsilon}$ and $q=(A^q,B^q,t^q,f^q)\in K_\varepsilon$
are such that
\begin{itemize}
\item $\rho^q$ is a successor ordinal. 
\item $\eta_\varepsilon, \kappa_y \le \sup(a^q)$.
\item  $B^q \subseteq D$.  
\end{itemize}
Then there is a lower part $y'$ such that
\begin{itemize}
\item$y' \concat (u^\varepsilon, B^q)  \le y \concat (u^\varepsilon, D)$
\item$y' \concat (u^\varepsilon, B^q)  \in g^\varepsilon$. 
\item$y'$ is harmonious with $A^q$ past $\eta_\varepsilon$. 
\end{itemize} 
\end{lemma}

\begin{proof}[\ref{lemma1-2-james}]  
Let $j_0 < \lambda$ be the largest ordinal such that $\kappa_{u_{j_0}} \le \max(a^q)$.  
For all $j$ with $j_0 < j < \lambda$
we have that $u_j \in \bigcup_{\sigma < \kappa} A^{*\varepsilon}_\sigma$ and
$\kappa_{u_j} > \max(a^q)$, so that $u_j \in B^q$.

Appealing to the second part of Lemma (\ref{conforming-and-sequence}) and to Lemma (\ref{filterandsequence3}), 
we may now extend  $y \concat (u^\varepsilon, D)$ to a condition
$\obar{y} \concat (u^\varepsilon, B^q) \in g^\varepsilon$ by first adding
in (as necessary) the pair $(\tup{ \eta_\varepsilon },\emptyset)$ and 
a pair with first entry $u_{j_0}$, then shrinking $D$ to $B^q$. Note that
$\kappa_{\obar{\scriptstyle y}} = \max(\set{\eta_\varepsilon,\kappa_y, \kappa_{u_{j_0}}}) \le \max(a^q)$. 

Appealing to Lemma (\ref{new-harmony-lemma}), there is $y'$ directly extending $\bar y$
in $\Rad{\obar{\scriptstyle y}}$ such that $y'$ conforms with $g^\varepsilon$, and $y'$ is harmonious
with $A^q$ past $\eta_\varepsilon$.
\end{proof}

\begin{proposition}\label{equivalent-characteristion-for-graph-cal-E-varepsilon-alpha}
For each $\varepsilon \in S\cap \xi$, $\alpha<\chi^-$ and $z$, $z' \in \rge(\mathfrak{f}^\varepsilon_\alpha)$
we have $z \mathrel{{\mathcal E}^{\sss\varepsilon}_\alpha } z'$ 
if and only if there exist a condition $q \in K_\varepsilon$ and a lower part $y$ harmonious with $A^q$ past $\eta_\varepsilon$
such that $q$ and $y\concat (u^\varepsilon,B^q)$ witness that $z \mathrel{{\mathcal E}^{\sss\varepsilon}_\alpha } z'$, {\em i.e.}, 
letting $\zeta = \height(\pi_0(z))$ and $\zeta' = \height(\pi_0(z'))$, such that  
\begin{itemize} 
\item  $\alpha \in t^{q, \eta_\varepsilon}$,  $\zeta,\sss \zeta'\in d^{q, \eta_\varepsilon}_\alpha$, 
$f^{q, \eta_\varepsilon}_\alpha(\zeta)=z$ and 
$f^{q, \eta_\varepsilon}_\alpha(\zeta')=z'$.
\item  $y$ is  harmonious with $A^q$ past  ${\eta}_\varepsilon$.
\item  $y\concat (u^\varepsilon,B^q) \in g^\varepsilon$.
\item  $y\concat (u^\varepsilon, B^q)  \forces^{V[G][H_\varepsilon]}_{\Rad{u^\varepsilon}} \zeta  \name{E}^\varepsilon_{\alpha} \zeta'$.
\end{itemize}
\end{proposition}

\begin{proof}[\ref{equivalent-characteristion-for-graph-cal-E-varepsilon-alpha}]  
If $q$ and $y$ are as in the equivalent, then $y\concat (u^\varepsilon,B^q) \in g^\varepsilon$
and  $y\concat (u^\varepsilon, B^q)  \forces^{V[G][H_\varepsilon]}_{\Rad{u^\varepsilon}} \zeta  \name{E}^\varepsilon_{\alpha} \zeta'$,
so that $\zeta\mathrel{\mathfrak{E}^{\sss\varepsilon}_\alpha} \zeta'$.

For the converse direction, suppose that $\zeta\mathrel{\mathfrak{E}^{\sss\varepsilon}_\alpha} \zeta'$.
Choose a condition $y \concat (u^\varepsilon, D) \in g^\varepsilon$ such that 
$y\concat (u^\varepsilon, D)  \forces^{V[G][H_\varepsilon]}_{\Rad{u^\varepsilon}} \zeta  \name{E}^\varepsilon_{\alpha} \zeta'$ and 
$\tup{ \eta_\varepsilon }$ appears in $y$.
By the choice of $\eta_\varepsilon$, the sequence $\tup{ \eta_\varepsilon }$ and all subsequent
sequences appearing in $y$ are members of $\bigcup_{\rho < \kappa} A^{* \varepsilon}_\rho$. 

Choose a condition $q \in K^\varepsilon$ such that $\eta_\varepsilon$,  $\kappa_y \in a^q$, $\rho^q$ is a successor ordinal,
and $B^q \subseteq D$.  Appealing to Lemma (\ref{lemma1-2-james}) we find a condition $y' \concat (u^\varepsilon, B^q)$
which is exactly of the kind needed to form a witness (together with $q$) that
$z\mathrel{ {\mathcal E}^{\sss\varepsilon}_\alpha } z'$.
\end{proof} 

Next we show that
${\mathcal E}^{\sss\varepsilon}_\alpha$ and  ${\mathcal E}^{\sss\varepsilon}_\beta$ cohere  for $\alpha \neq \beta$.

\begin{lemma}  \label{coherence-lemma}  Let $\varepsilon \in S\cap \xi$, $\alpha \neq \beta$ and 
$z$ and $z'\in \rge(\mathfrak{f}^\varepsilon_\alpha) \cap  \rge(\mathfrak{f}^\varepsilon_\beta)$.
Then $z \mathrel{{\mathcal E}^{\sss\varepsilon}_\alpha} z' $
if and only $z \mathrel{{\mathcal E}^{\sss\varepsilon}_\beta} z' $.
\end{lemma}

\begin{proof}[\ref{coherence-lemma}] Choose $r \in K_\varepsilon$ and $y$ a lower part witnessing 
the equivalent conditions listed in Proposition (\ref{equivalent-characteristion-for-graph-cal-E-varepsilon-alpha})
that $z \mathrel{{\mathcal E}^{\sss\varepsilon}_\alpha} z'$. Choose $p \le r$ such that
$p \in K_\varepsilon$, $\rho^p$ is a successor ordinal, and $p$ contains enough information to verify that 
$z$, $z' \in \rge(\mathfrak{f}^\varepsilon_\beta)$, that is to say 
that $\beta \in t^{p, \eta_\varepsilon}$, $\zeta, \zeta' \in d^{p, \eta_\varepsilon}_\beta$,
$f^{p, \eta_\varepsilon}_\beta(\zeta) = z$ and $f^{p, \eta_\varepsilon}_\beta(\zeta') = z'$. 

We now recall the notion of $\rho^*$-weakening from Definition (\ref{rho-star-weakening}).
Let $q$ be the $\rho^r$-weakening of $p$, and note that by definition 
$A^q = A^p \restriction \rho^r = A^r$ and $a^q = a^r$.
It follows from Lemma (\ref{abacus1}) and Corollary (\ref{abacus2}) that $p \le q \le r$, in particular
$q \in K_\varepsilon$ and $B^q \subseteq B^r$.

We claim that $y$ and $B^q$ will serve as witnesses that 
$z\mathrel{{\mathcal E}^{\sss\varepsilon}_\beta} z'$, that is that $y \concat (u^\varepsilon,B^q) \in g^\varepsilon$
and that  $y\concat (u^\varepsilon, B^q)  \forces \zeta  \name{E}^\varepsilon_{\beta} \zeta'$.

By Lemma (\ref{filterandsequence3}), to check that $y \concat (u^\varepsilon,B^q) \in g^\varepsilon$ we must show that
$y$ conforms with $g^\epsilon$ and $u \in B^q$ for all $u$ appearing on the generic sequence with $\kappa_y < \kappa_u$. 
Since $y \concat (u^\epsilon, B^r) \in g^\epsilon$, we see that the lower part $y$ conforms with $g^\epsilon$
and also that $u \in B^r$ for all $u$ appearing on the generic sequence with $\kappa_y < \kappa_u$. 
Fix such a sequence $u$. 
Since $r$ is a condition and $u \in B^r$ it follows that  
$ \ssup(a^q) = \ssup(a^r) \le \kappa_u $. By the choice of $\eta_\varepsilon$ and
$y$, we have $u \in \bigcup \setof{A^{*\varepsilon}_\rho}{\rho<\kappa}$, and
since $q \in K_\varepsilon$ we have that $u \in B^q$.

Since $B^q \subseteq B^r$, we see that $y \concat (u^\epsilon, B^q)$ is a refinement
of $y \concat (u^\epsilon, B^r)$, and in particular 
$y \concat (u^\epsilon, B^q) \forces \zeta  \name{E}^\varepsilon_{\alpha} \zeta'$.
Since $A^q = A^r$ and $q$ is a condition, it follows from Clause (5)
in Definition (\ref{qstar-forcing}) that $y \concat (u^\epsilon, B^q) \forces \zeta  \name{E}^\varepsilon_{\beta} \zeta'$.
\end{proof}

\begin{definition}
Working in the model $V[G][H_\xi][g^*]$, we define for each $\varepsilon\in S \cap \xi$ 
a relation ${\mathcal E}^\varepsilon$ on 
the set  
$\bigcup\setof{\setof{b_\alpha\on\zeta}{\zeta<\kappa^+}\times \kappa}{\alpha< \chi^-}$ by 
\[
z\sss {\mathcal E}^\varepsilon \sss z' \longleftrightarrow \exists \alpha<\chi^- \;
z\sss {\mathcal E}^{\sss\varepsilon}_\alpha \sss z'.
\]
\end{definition}

\begin{proposition}\label{prop5} 
$V[G][H_\xi][g^*]\thinks {\mathcal E}^\varepsilon$ is of size $\kappa^+$ and is universal for 
$\setof{\mathfrak{E}^\varepsilon_\alpha}{\alpha<\chi^-}$.
\end{proposition}

\begin{proof}[\ref{prop5}] ${\mathcal E}^\varepsilon$ has size $\kappa^+$ by the choice of 
$\setof{b_\alpha}{\alpha<\chi^-}$, as for $\alpha<\chi^-$ and $\zeta<\kappa^+$ we have that 
$b_\alpha\on\zeta \in {\mathcal T}$, where ${\mathcal T}$ is the $\kappa^+$-Kurepa tree
added by $\Qbb_0$.

The universality follows from the 
definitions of $\mathfrak{E}^\varepsilon_\alpha$, ${\mathcal E}^\varepsilon_\alpha$ and ${\mathcal E}^\varepsilon$, and 
Lemma (\ref{coherence-lemma}) which ensures that each $\mathfrak{f}^\varepsilon_\alpha$ is an embedding of 
$\mathfrak{E}^\varepsilon_\alpha$ into ${\mathcal E}^\varepsilon$: for Definition (\ref{defn-the-graph-cal-E-varepsilon-alpha})
ensures that each $\mathfrak{f}^\varepsilon_\alpha$ is an embedding of 
$\mathfrak{E}^\varepsilon_\alpha$ into ${\mathcal E}^\varepsilon_\alpha\subseteq {\mathcal E}^{\varepsilon}$,
and Lemma (\ref{coherence-lemma}) ensures the compatibility of these embeddings.
\end{proof}

\begin{proposition}\label{prop6} In $V[G][H_\xi][g^{*}]$ 
suppose $\Xi$ ($\in V[G][H_\xi][g^{*}]$) is cofinal in $S\cap\xi$ and 
for all $\varepsilon\in \Xi$ we have that 
$\setof{\name{E}^\varepsilon_\alpha}{\alpha<\chi^-}$ 
is a list of canonical names for all graphs on $\kappa^+$.
Then $\setof{{\mathcal E}^\varepsilon}{\varepsilon\in \Xi}$ is a universal family in
the collection of graphs on $\kappa^+$. 
\end{proposition}

\begin{proof}[\ref{prop6}] Let $\mathfrak{E}$ be a graph on $\kappa^+$ 
in $V[G][H_\xi][g^*]$ and let $\name{E}$ be a 
canonical $\Lbb * \name{\Pbb}_\xi * \nameRad{v}$-name such that
$\mathfrak{E}=\name{E}^{G * H_\xi * g^{*}}$. As $\name{E}$ is a canonical name 
it has size at most $\kappa^+$ and hence, as $\cf(\xi)>\kappa^+$,
there is some $\varepsilon \in \Xi$ such that $\name{E}$ is an
$\Lbb * \name{\Pbb}_\varepsilon * \nameRad{u^\varepsilon}$-name and hence
$\mathfrak{E}=\name{E}^{G * H_\varepsilon * {g^{\varepsilon}}}$. By Proposition (\ref{prop5}) we thus 
have that $\mathfrak{E}$ embeds into ${\mathcal E}^\varepsilon$.
\end{proof}

\begin{theorem}\label{universalg} Suppose $\kappa$ is a supercompact cardinal,
$\lambda<\kappa$ is a regular cardinal and 
$\Theta$ is a cardinal with $\cf(\Theta)\ge\kappa^{++}$ and $\kappa^{+3}\le\Theta$.
There is a forcing extension in which cofinally many cardinals below $\kappa$, $\kappa$ itself and 
all cardinals greater than $\kappa$ are preserved, $\cf(\kappa)=\lambda$,  $2^{\kappa}=2^{\kappa^+} =
\Theta$ and there is a universal family of graphs on $\kappa^+$ of size $\kappa^{++}$.
\end{theorem}

\begin{proof}[\ref{universalg}] Let $\chi=\Theta^+$. As mentioned in the first paragraph of \S{}6, it is standard that
if $\kappa$ is supercompact there is a forcing extension in which $\kappa$ remains
supercompact and GCH and $\diamondsuit_\chi(S^\chi_{\kappa^+})$ hold. 
Now force with $\Lbb * \Pbb_\xi * \Rad{v}$ and work in $V[G][H_\xi][g^*]$. 
Cofinally many cardinals below $\kappa$,
$\kappa$ itself and all cardinals above $\kappa$ are preserved from $V$ to $V[G]$
and the remaining factors preserve all cardinals. 
The final factor Radin forcing makes $\lambda$ the cofinality of the previously regular
(indeed supercompact) cardinal $\kappa$. The second factor of the forcing makes  $2^{\kappa}=2^{\kappa^+} =
\Theta$ and the final Radin factor does not increase these values as it has size $2^{\kappa}$. By Proposition (\ref{prop6}) 
$\setof{{\mathcal E}^\varepsilon}{\varepsilon\in \Xi}$ is a universal family in
the collection of graphs on $\kappa^+$ and has size $\card{\Xi}=\kappa^{++}$. 
$\vphantom{www}$
\end{proof}

\bibliography{5authorsbib}
\bibliographystyle{plain}

\end{document}